\documentclass{amsart}
\usepackage{amssymb}
\usepackage{amsfonts}
\usepackage{amsmath}
\usepackage{amsthm, amsmath, amssymb, enumerate}
\usepackage{amssymb, enumitem}
\usepackage{bbm}
\usepackage{amscd}
\usepackage[all]{xy}
\usepackage[hidelinks]{hyperref}
\usepackage{amsmath}
\usepackage{amssymb}
\usepackage{amsthm}
\usepackage[sort&compress,numbers]{natbib}
\usepackage[left=2cm,right=2cm,bottom=3cm,top=3cm]{geometry}
\usepackage{tikz}
\usepackage{tikz-cd}
\usepackage{wrapfig}

\newtheorem{theorem}{Theorem}[section]

\newtheorem{proposition}[theorem]{Proposition}
\newtheorem{lemma}[theorem]{Lemma}
\newtheorem{corollary}[theorem]{Corollary}
\theoremstyle{definition}
\newtheorem*{claim*}{Claim}
\newtheorem{definition}[theorem]{Definition}

\newtheorem{remark}[theorem]{Remark}
\newtheorem{problem}[theorem]{Problem}

\newtheorem{example}[theorem]{Example}

\AtBeginDocument{   \def\MR#1{}}

\begin{document}
\title{Homological algebra of pro-Lie Polish abelian groups}
\author{Matteo Casarosa}
\address{Institut de Math\'{e}matiques de Jussieu - Paris Rive Gauche
(IMJ-PRG), Universit\'{e} Paris Cit\'{e}, B\^{a}timent Sophie Germain, 8
Place Aur\'{e}lie Nemours, 75013 Paris, France, and Dipartimento di
Matematica, Universit\`{a} di Bologna, Piazza di Porta S. Donato, 5, 40126
Bologna, Italy}
\email{matteo.casarosa@unibo.it}
\email{matteo.casarosa@imj-prg.fr}
\urladdr{https://webusers.imj-prg.fr/~matteo.casarosa/}
\urladdr{https://www.unibo.it/sitoweb/matteo.casarosa/en}
\author{Alessandro Codenotti}
\address{Dipartimento di Matematica, Universit\`{a} di Bologna, Piazza di
Porta S. Donato, 5, 40126 Bologna,\ Italy}
\email{alessandro.codenotti@unibo.it}
\urladdr{}
\author{Martino Lupini}
\address{Dipartimento di Matematica, Universit\`{a} di Bologna, Piazza di
Porta S. Donato, 5, 40126 Bologna,\ Italy}
\email{martino.lupini@unibo.it}
\urladdr{http://www.lupini.org/}
\thanks{The authors were partially supported by the Starting Grant
101077154\textquotedblleft Definable Algebraic Topology\textquotedblright\
from the European Research Council.}
\subjclass[2000]{Primary 54H05 , 20K45, 18F60; Secondary 26E30 , 18G10, 46M15%
}
\keywords{Polish group, pro-Lie group, non-Archimedean group, abelian group,
topological torsion group, abelian category, quasi-abelian category, derived
functor, extensions}
\date{\today }

\begin{abstract}
In this paper, we initiate the study of pro-Lie Polish abelian groups from
the perspective of homological algebra. We extend to this context the
type-decomposition of locally compact Polish abelian groups of Hoffmann and
Spitzweck, and prove that the category $\mathbf{proLiePAb}$ of pro-Lie
Polish abelian groups is a thick subcategory of the category of Polish
abelian groups. We completely characterize injective and projective objects
in $\mathbf{proLiePAb}$. We conclude that $\mathbf{proLiePAb}$ has enough
projectives but not enough injectives and homological dimension $1$. We also
completely characterize injective and projective objects in the category of
non-Archimedean Polish abelian groups, concluding that it has enough
injectives and projectives and homological dimension $1$. Injective objects
are also characterized for the categories of topological torsion Polish
abelian groups and for Polish abelian topological $p$-groups, showing that
these categories have enough injectives and homological dimension $1$.
\end{abstract}

\maketitle

\renewcommand{\themaintheorem}{\Alph{maintheorem}}

\section{Introduction}

In this paper, we initiate the study of pro-Lie Polish abelian groups from
the viewpoint of homological algebra. \emph{Pro-Lie groups }\cite%
{hofmann_lie_2007} are the topological groups that can be written as limits
of an inverse system of Lie groups. We will focus on pro-Lie groups that are
furthermore \emph{Polish}, in which case they can written as limits of an
inverse sequence of Lie groups. The class of pro-Lie Polish groups is a
natural extension of the class of Lie groups, and it satisfies desirable
closure properties, as it is closed under countable products, closed
subgroups, and quotients by closed subgroups within the class of Polish
groups.

We will focus on the case of \emph{abelian }pro-Lie Polish groups, as they
form a category enriched over abelian groups, where the machinery from
homological algebra can be applied. The category of abelian Polish pro-Lie
groups contains all the\emph{\ locally compact} Polish abelian groups, as
well as the \emph{non-Archimedean }Polish abelian groups (which are inverse
limits of countable discrete groups). In fact, the class of pro-Lie Polish
abelian groups is the smallest class of Polish abelian groups that contains
all locally compact Polish abelian groups and non-Archimedean Polish abelian
groups and it is closed under extensions.

The study of homological invariants in the context of topological abelian
groups followed rapidly their introduction in the purely algebraic setting,
going as far back as the 1940s \cite%
{chevalley_cohomology_1948,nagao_extension_1949}; see also \cite%
{calabi_sur_1951,mackey_ensembles_1957,brown_extensions_1971,fulp_extensions_1971,fulp_splitting_1972,fulp_homological_1970,moskowitz_homological_1967}%
. This study has been most commonly undertaken for \emph{locally compact }%
topological Hausdorff (Polish) abelian groups. Injective and projective
objects in this category were characterized by\ Moskowitz in \cite%
{moskowitz_homological_1967}, showing in particular that this category does 
\emph{not }have enough injective and projective objects. The Yoneda $\mathrm{%
Ext}$ groups in terms of extensions were introduced and studied by Fulp and
Griffith in \cite{fulp_extensions_1971}, where it is proved that $\mathrm{Ext%
}^{2}$ vanishes. The functor $\mathrm{Ext}$ for locally compact Hausdorff
topological abelian groups was later studied from the viewpoint of derived
categories by Hoffmann and Spitzweck \cite{hoffmann_homological_2007}, who
recognized it as the cohomological right derived functor of $\mathrm{Hom}$
in the sense of Verdier \cite{verdier_categories_1977}, despite the category
of locally compact Hausdorff topological groups not having enough injectives
or projectives.

Derived categories and functors had been traditionally studied in the
context of \emph{abelian categories}. However, this framework does not
include the category of locally compact topological Hausdorff (Polish)
abelian groups, as the inclusion map $\mathbb{Q}\rightarrow \mathbb{R}$ of
the rationals with the discrete topology into the reals with the Euclidean
topology is an arrow that is both monic and epic but not an isomorphism.
This and other categories of algebraic structures endowed with a topology
can be seen as \emph{quasi-abelian categories}, where the axioms of abelian
categories (such as the requirement that every epimorphism be the cokernel
of its kernel) are replaced with suitable relaxations. Most of the notions
and constructions usually considered in the context of abelian categories
can be generalized to the quasi-abelian case.\ Furthermore, Schneiders
describes in \cite{schneiders_quasi-abelian_1999}, building on work of
Beilinson--Bernstein--Deligne \cite{beilinson_faisceaux_1982}, a canonical
way to enlarge a given quasi-abelian category $\mathcal{A}$ to an abelian
category $\mathrm{LH}\left( \mathcal{A}\right) $ called its \emph{left heart}%
.

The quasi-abelian viewpoint affords Hoffmann and Spitzwek to \emph{refine }%
the functor $\mathrm{Ext}$ for locally compact topological (Polish) abelian
groups, by seeing it as a functor to the \emph{left heart }of the category
of topological (Polish) abelian groups. The category $\mathbf{PAb}$ of
Polish abelian groups, or any of its full subcategories closed by taking
closed subgroups, quotients by closed subgroups, and extensions (\emph{thick
subcategories}), is easily seen to be quasi-abelian. Using tools from
descriptive set theory, an explicit description of its left heart as
concrete category was recently provided in \cite{lupini_looking_2022} in
terms of \emph{groups with a Polish cover }and \emph{Borel-definable group
homomorphisms}; see also \cite{bergfalk_definable_2020}.

Building on the work of Hoffmann and Spitzweck, Bergfalk, Sarti, and the
third-named author \cite{bergfalk_applications_2023} gave a short purely
homological proof of the fact that the functor $\mathrm{Hom}$ on the
category $\mathbf{LCPAb}$ of locally compact Polish abelian groups has a 
\emph{total }right derived functor, from which $\mathrm{Ext}^{n}$ for $n\geq
0$ is recovered by taking cohomology, and the homological dimension of $%
\mathbf{LCPAb}$ is $1$. The latter result recovers the theorem that $\mathrm{%
Ext}^{n}=0$ for $n\geq 2$ that had previously been obtained by Fulp and
Griffith via a much more laborious argument. The enrichment of $\mathrm{Ext}$
as group with a Polish cover is used in \cite{bergfalk_applications_2023} to
completely characterize injective and projective objects in the left heart
of $\mathbf{LCPAb}$ and many of its salient thick subcategories, showing in
particular that $\mathrm{LH}\left( \mathbf{LCPAb}\right) $ has no nonzero
injectives.

In this paper, we extend the applications of Borel-definable methods to the
category of pro-Lie Polish abelian groups. We begin with showing that the
category $\mathbf{proLiePAb}$ of pro-Lie Polish abelian groups is
quasi-abelian, and in fact a thick subcategory of $\mathbf{PAb}$. In order
to prove that an extension of pro-Lie Polish abelian groups is pro-Lie, we
use recently developed tools from the theory of groups with a Polish cover
and Borel and continuous cocycles, including a continuous selection theorem
from \cite{bergfalk_definable_2020}. We then completely characterize
projective and injective objects in $\mathbf{proLiePAb}$, showing in
particular that $\mathbf{proLiePAb}$ has enough projectives (but not enough
injectives) and homological dimension $1$. Thus $\mathbf{proLiePAb}$ is in
some ways better-behaved than $\mathbf{LCPAb}$ from the viewpoint of
homological algebra, and the total derived functor of $\mathrm{Hom}$ on $%
\mathbf{proLiePAb}$ can be defined in terms of projective resolutions, as in
the case of discrete groups. In the case when the first argument is locally
compact, and the second one merely pro-Lie, we prove that $\mathrm{Ext}$ can
be regarded as a functor to the left heart of $\mathbf{PAb}$, and hence
enriched with the structure of group with a Polish cover, extending from the
locally compact case the analysis in \cite{bergfalk_applications_2023}.

We perform a similar analysis on the category $\mathbf{PAb}_{\mathrm{nA}}$
of \emph{non-Archimedean }Polish abelian groups, which is also a thick
subcategory of $\mathbf{PAb}$. In this case, we completely characterize
injective and projective objects, and show that $\mathbf{PAb}_{\mathrm{nA}}$
has enough injectives and projectives and homological dimension $1$. We also
completely characterize injective and projective objects in the categories
of pro-$p$ pro-Lie Polish abelian groups and topological torsion pro-Lie
Polish abelian groups, and show that these categories have enough injectives
and homological dimension $1$.

We also extend from the locally compact case---considered by Hoffmann and
Spitzweck in \cite{hoffmann_homological_2007}---to pro-Lie Polish abelian
groups the notions of type $\mathbb{Z}$, type $\mathbb{A}$, and type $%
\mathbb{S}^{1}$ groups, as well as the decomposition of an arbitrary pro-Lie
Polish abelian groups in terms of groups of these types.

Finally, we show that the functor $\mathrm{Hom}:\mathbf{LCPAb}\times \mathbf{%
PAb}\rightarrow \mathbf{PAb}$ admits a total right derived functor,
extending a result from \cite{bergfalk_applications_2023} in the locally
compact case.

The rest of this paper is divided into three sections. In Section \ref%
{Section:categories} we recall some preliminary notions from category
theory, including the notions of exact and quasi-abelian category and the
corresponding notions of derived category and functor. In Section \ref%
{Section:pro-Lie} we obtain our main results outlined above concerning the
category of pro-Lie Polish abelian groups. Finally, in Section \ref%
{Section:thick} we consider several natural thick (or just fully exact)
subcategories of $\mathbf{proLiePAb}$, characterizing injective and
projective objects in each of them, and determining whether they have enough
injectives or projectives.

\subsubsection*{Acknowledgments}

We thank Jeffrey Bergfalk, Andr\'{e} Nies, and Joe Zielinski for many useful
conversations, and Luca Marchiori for a careful reading of a preliminary
version of this paper and a large number of helpful suggestions. 

\section{Category theory background\label{Section:categories}}

\subsection{Exact categories}

An\emph{\ additive category} \cite[Section VIII.2]{mac_lane_categories_1998}
is a category enriched over the category $\mathbf{Ab}$ of abelian groups
that also has a terminal object (which is necessarily also initial, and
called the \emph{zero object}) and binary products (which are necessarily
also coproducts, and called \emph{biproducts}). A kernel-cokernel pair in an
additive category $\mathcal{A}$ is a pair $\left( f,g\right) $ of arrows in $%
\mathcal{A}$ such that $f$ is the kernel of $g$ and $g$ is the cokernel of $%
f $. This can be seen as an object in the double arrow category $\mathcal{A}%
^{\cdot \rightarrow \cdot \rightarrow \cdot }$ of $\mathcal{A}$.

An \emph{exact structure }$\mathcal{E}$ on $\mathcal{A}$ \cite[Definition 2.1%
]{buhler_exact_2010} is a collection of short exact sequences closed under
isomorphism in $\mathcal{A}^{\cdot \rightarrow \cdot \rightarrow \cdot }$
that satisfies the following axioms as well as their duals obtained by
reversing all the arrows.\ In order to formulate the axioms, we say that a
morphism $f$ is an admissible monic (or inflation) or an admissible epic (or
deflation) if it appears in an exact sequence in $\mathcal{E}$:

\begin{enumerate}
\item for every object $A$ of $\mathcal{A}$, the corresponding identity
arrow $1_{A}$ is an admissible monic;

\item the collection of admissible monics is closed under composition;

\item the push-out of an admissible monic along an arbitrary morphism exists
and it is an admissible monic.
\end{enumerate}

An \emph{exact category }is a pair $\left( \mathcal{A},\mathcal{E}\right) $
where $\mathcal{A}$ is an additive category and $\mathcal{E}$ is an exact
structure on $\mathcal{A}$. A short exact sequence in $\left( \mathcal{A},%
\mathcal{E}\right) $ is by definition an element of $\mathcal{E}$. In an
exact category, an isomorphism is an admissible monic and an admissible
epic; see \cite[Remark 2.3]{buhler_exact_2010}. Furthermore, by \cite[Remark
2.3, Proposition 2.15]{buhler_exact_2010} we have the following---see \cite[%
Sections 5.2, 5.3]{awodey_category_2006} for the notion of pull-back along a
morphism:

\begin{lemma}
\label{Lemma:pull-back-admissible-monic}In an exact category, the pull-back
of an admissible monic along an admissible epic is an admissible monic.
\end{lemma}

Suppose that $\mathcal{A}$ and $\mathcal{B}$ are exact categories, and $F:%
\mathcal{A}\rightarrow \mathcal{B}$ is a functor. One says that $F$ is \emph{%
exact }if it maps short exact sequences in $\mathcal{A}$ to short exact
sequences in $\mathcal{B}$ \cite[Remark 2.3, Proposition 2.15]%
{buhler_exact_2010}. By \cite[Proposition 5.2]{buhler_exact_2010}, this
implies that $F$ preserves push-outs along admissible monics, and pull-backs
along admissible epics.

\begin{remark}
In the context of finitely complete and finitely cocomplete categories, a
stronger notion of exact functor is considered. If $\mathcal{C}$ and $%
\mathcal{D}$ are categories that are finitely complete and finitely
cocomplete (i.e., have finite limits and finite colimits), then a functor $F:%
\mathcal{C}\rightarrow \mathcal{D}$ is called \emph{exact }if it is finitely
continuous and finitely cocontinuous (i.e., commutes with finite limits and
finite colimits); see \cite[Section 3.3]{kashiwara_categories_2006}. An
exact functor in the sense of \cite[Proposition 5.2]{buhler_exact_2010}
between finitely complete and finitely cocomplete exact categories is not
necessarily exact in the sense of \cite[Section 3.3]%
{kashiwara_categories_2006}.
\end{remark}

Suppose that $\mathcal{A}$ is an exact category. A\emph{\ fully exact}
subcategory \cite[Definition 10.21]{buhler_exact_2010} of $\mathcal{A}$ is a
full subcategory $\mathcal{B}$ of $\mathcal{A}$ such that, whenever $%
0\rightarrow B^{\prime }\rightarrow A\rightarrow B^{\prime \prime
}\rightarrow 0$ is a short exact sequence in $\mathcal{A}$ with $B^{\prime }$
and $B^{\prime \prime }$ in $\mathcal{B}$, one also has that $A$ is
isomorphic to an object of $\mathcal{B}$. In this case, we have that $%
\mathcal{B}$ is an exact category, where a short exact sequence in $\mathcal{%
B}$ is a short exact sequence in $\mathcal{A}$ whose morphisms are in $%
\mathcal{B}$ \cite[Definition 10.20]{buhler_exact_2010}.

Let $\mathcal{A}$ be an exact category and $f$ be a morphism in $\mathcal{A}$%
. Then we say that $f$ is \emph{admissible }\cite[Definition 8.1]%
{buhler_exact_2010} if $f=m\circ e$ for some admissible epic $e$ and
admissible monic $m$. The pair $\left( e,m\right) $ is called an admissible
factorization of $f$ and its essentially unique \cite[Lemma 8.4]%
{buhler_exact_2010}. The notion of admissible morphism recovers the notion
of admissible monic (respectively, epic) in the particular case when the
given morphism is monic (respectively, epic) \cite[Remark 8.3]%
{buhler_exact_2010}. The class of admissible morphisms is closed under
push-out along admissible monics and pull-back along admissible epics \cite[%
Lemma 8.7]{buhler_exact_2010}.

\begin{definition}
\label{Definition:exact}A sequence of admissible morphisms%
\begin{equation*}
A\overset{f}{\rightarrow }B\overset{g}{\rightarrow }A^{\prime }
\end{equation*}%
in an exact category $\mathcal{A}$ is exact if%
\begin{equation*}
\bullet \overset{m}{\rightarrow }A\overset{e^{\prime }}{\rightarrow }\bullet
\end{equation*}%
is short exact, where $f=m\circ e$ and $g=m^{\prime }\circ e^{\prime }$ are
admissible factorizations.
\end{definition}

\subsection{Derived categories}

Let $\mathcal{A}$ be an additive category. One can then define the category $%
\mathrm{Ch}\left( \mathcal{A}\right) $ of complexes, and the homotopy
category $\mathrm{K}\left( \mathcal{A}\right) $; see \cite[Section 9]%
{buhler_exact_2010}. We have that $\mathrm{K}\left( \mathcal{A}\right) $ has
a canonical structure of \emph{triangulated category }\cite[Remark 9.8]%
{buhler_exact_2010}, graded by the translation functor $T$ defined by $%
TA=A[1]$ where $A[k]^{n}=A^{n+k}$ for $n,k\in \mathbb{Z}$. The distinguished
triangles in $\mathrm{K}\left( A\right) $ are those isomorphic to a strict
triangle, which is one of the form%
\begin{equation*}
\left( A,B,\mathrm{cone}\left( f\right) ,f,i,j\right)
\end{equation*}%
for some morphism of complexes $f:A\rightarrow B$, where $i:B\rightarrow 
\mathrm{cone}\left( f\right) $ and $j:\mathrm{cone}\left( f\right)
\rightarrow TA$ are the canonical morphisms of complexes.

A morphism $f:A\rightarrow A$ in $\mathcal{A}$ is idempotent if $f=f\circ f$%
. The category $\mathcal{A}$ is idempotent-complete if every idempotent
morphism has a kernel. Suppose now that $\mathcal{A}$ is an \emph{%
idempotent-complete exact category}. A complex $A$ over $\mathcal{A}$ is 
\emph{acyclic }\cite[Section 10]{buhler_exact_2010}\emph{\ }if it has
admissible differentials and it is exact at each degree as in Definition \ref%
{Definition:exact}. We let $\mathrm{N}\left( \mathcal{A}\right) $ be the
full subcategory of $\mathrm{K}\left( \mathcal{A}\right) $ spanned by
acyclic complexes. Then we have that $\mathrm{N}\left( \mathcal{A}\right) $
is a \emph{thick subcategory }\cite[Definition 2.1.6]%
{neeman_triangulated_2001} of the triangulated category $\mathrm{K}\left( 
\mathcal{A}\right) $ \cite[Corollary 10.11]{buhler_exact_2010}. This means
that $\mathrm{N}\left( \mathcal{A}\right) $ is a \emph{triangulated
subcategory }of $\mathrm{K}\left( \mathcal{A}\right) $ \cite[Definition 1.5.1%
]{neeman_triangulated_2001} and contains all the direct summands of its
objects.

One can thus define the corresponding Verdier quotient $\mathrm{D}\left( 
\mathcal{A}\right) :=\mathrm{K}\left( \mathcal{A}\right) /\mathrm{N}\left( 
\mathcal{A}\right) $ \cite[Theorem 2.1.8]{neeman_triangulated_2001}, which
is called the \emph{derived category }of $\mathcal{A}$ \cite[Section 10]%
{buhler_exact_2010}. By definition, this is the localization $\mathrm{K}%
\left( \mathcal{A}\right) [\Sigma ^{-1}]$ where $\Sigma $ is the
multiplicative system in $\mathrm{K}\left( \mathcal{A}\right) $ consisting
of homotopy classes of chain maps whose mapping cone is acyclic, called 
\emph{quasi-isomorphisms }\cite[Definition 10.16]{buhler_exact_2010}. If $Q_{%
\mathcal{A}}:\mathrm{K}\left( \mathcal{A}\right) \rightarrow \mathrm{D}%
\left( \mathcal{A}\right) $ is the quotient functor, then a morphism in $%
\mathrm{K}\left( \mathcal{A}\right) $ is a quasi-isomorphism if and only if
its image in $\mathrm{D}\left( \mathcal{A}\right) $ is invertible, i.e.\ the
multiplicative system $\Sigma $ is \emph{saturated }\cite[Remark 10.19]%
{buhler_exact_2010}.

The full subcategories $\mathrm{K}^{+}\left( \mathcal{A}\right) $, $\mathrm{K%
}^{-}\left( \mathcal{A}\right) $, and $\mathrm{K}^{\mathrm{b}}\left( 
\mathcal{A}\right) $ of $\mathrm{K}\left( \mathcal{A}\right) $ and the
corresponding full subcategories $\mathrm{D}^{+}\left( \mathcal{A}\right) $, 
$\mathrm{D}^{-}\left( \mathcal{A}\right) $, and $\mathrm{D}^{\mathrm{b}%
}\left( \mathcal{A}\right) $ of $\mathrm{D}\left( \mathcal{A}\right) $ are
defined in a similar fashion, by considering only left-bounded,
right-bounded, and bounded complexes respectively.

\subsection{Abelian and quasi-abelian categories}

A \emph{quasi-abelian category }\cite[Definition 4.1]{buhler_exact_2010} is
a finitely complete and finitely cocomplete additive category such that the
push-out of a kernel along an arbitrary morphism is a kernel, and dually the
pull-back of a cokernel along an arbitrary morphism is a cokernel. If $%
\mathcal{A}$ is a quasi-abelian category, then $\left( \mathcal{A},\mathcal{E%
}\right) $ is an exact category, where $\mathcal{E}$ is the collection of
all the kernel-cokernel pairs in $\mathcal{A}$ \cite[Proposition 4.4]%
{buhler_exact_2010}. In what follows, we will regard every quasi-abelian
category as an exact category with respect to such a canonical (maximal)
exact structure. An \emph{abelian category }is a quasi-abelian category
where every morphism is admissible or, equivalently, every monomorphism is a
kernel and every epimorphism is a cokernel \cite[Remark 4.7]%
{buhler_exact_2010}.

Let $\mathcal{A}$ be a quasi-abelian category, and $\mathcal{B}$ be a (not
necessarily full) subcategory of $\mathcal{A}$. Then we say that $\mathcal{B}
$ is a quasi-abelian subcategory of $\mathcal{A}$ if $\mathcal{B}$ is a
quasi-abelian category and the inclusion functor $\mathcal{B}\rightarrow 
\mathcal{A}$ is finitely continuous and finitely cocontinuous. We say that $%
\mathcal{B}$ is a\emph{\ fully exact quasi-abelian subcategory }if it is a
quasi-abelian subcategory that is also a fully exact subcategory (where $%
\mathcal{A}$ is endowed with its canonical exact structure). In this case,
we also say that $\mathcal{B}$ is a \emph{thick} subcategory of $\mathcal{A}$%
; see \cite[Definition 8.3.21]{kashiwara_categories_2006}. (Notice that this
is different from the notion of thick subcategory of a triangulated category
from \cite[Definition 2.1.6]{neeman_triangulated_2001}.)

Let $\mathcal{A}$ be a quasi-abelian category and $\mathrm{D}\left( \mathcal{%
A}\right) $ be its derived category. The \emph{left heart }of $\mathcal{A}$
is by definition the heart of $\mathrm{D}\left( \mathcal{A}\right) $ with
respect to its canonical left truncation structure \cite[Definition 1.2.18]%
{schneiders_quasi-abelian_1999}. Concretely, $\mathrm{LH}\left( \mathcal{A}%
\right) $ is the full subcategory spanned by complexes $A$ with $A^{n}=0$
for $n\in \mathbb{Z}\setminus \left\{ -1,0\right\} $ and such that the
differential $A^{-1}\rightarrow A^{0}$ is monic. Then $\mathrm{LH}\left( 
\mathcal{A}\right) $ is an abelian category that contains $\mathcal{A}$ as a
thick subcategory \cite[Proposition 1.2.29]{schneiders_quasi-abelian_1999}.
Furthermore, the inclusion $\mathcal{A}\rightarrow \mathrm{LH}\left( 
\mathcal{A}\right) $ satisfies the following universal property: for every
abelian category $\mathcal{M}$ and finitely continuous and exact functor $F:%
\mathcal{A}\rightarrow \mathcal{M}$ there exists an essentially unique
functor $\mathrm{LH}\left( \mathcal{A}\right) \rightarrow \mathcal{M}$ whose
restriction to $\mathcal{A}$ is isomorphic to $F$ \cite[Proposition 1.2.34]%
{schneiders_quasi-abelian_1999}. The inclusion $\mathcal{A}\rightarrow 
\mathrm{LH}\left( \mathcal{A}\right) $ extends to an equivalence of
categories $\mathrm{D}\left( \mathcal{A}\right) \rightarrow \mathrm{D}\left( 
\mathrm{LH}\left( \mathcal{A}\right) \right) $ \cite[Proposition 1.2.32]%
{schneiders_quasi-abelian_1999}, and it has a left adjoint $\kappa :\mathrm{%
LH}\left( \mathcal{A}\right) \rightarrow \mathcal{A}$ \cite[Definition
1.2.26 and Proposition 1.2.27]{schneiders_quasi-abelian_1999}.

The canonical left truncation structure on $\mathrm{D}\left( \mathcal{A}%
\right) $ yields the \emph{cohomology functors }$\mathrm{H}^{n}:\mathrm{D}%
\left( \mathcal{A}\right) \rightarrow \mathrm{LH}\left( \mathcal{A}\right) $
for $n\in \mathbb{Z}$ \cite[Definition 1.2.18]{schneiders_quasi-abelian_1999}%
. We have that a morphism $f$ in $\mathrm{D}\left( \mathcal{A}\right) $ is
an isomorphism if and only if $\mathrm{H}^{n}\left( f\right) $ is an
isomorphism for every $n\in \mathbb{Z}$.

Recall that a \emph{torsion pair }\cite[Definition 2.4]{tattar_torsion_2021}
in a quasi-abelian category $\mathcal{M}$ is a pair $\left( \mathcal{T},%
\mathcal{F}\right) $ of full subcategories, where $\mathcal{T}$ is called a
torsion class and $\mathcal{F}$ called a torsion-free class, such that:

\begin{enumerate}
\item for all objects $T$ of $\mathcal{T}$ and $F$ of $\mathcal{F}$, $%
\mathrm{Hom}\left( T,F\right) =0$;

\item for all objects $M$ of $\mathcal{M}$ there exists a short exact
sequence%
\begin{equation*}
0\rightarrow {}_{\mathcal{T}}M\rightarrow M\rightarrow M_{\mathcal{F}%
}\rightarrow 0
\end{equation*}%
where $_{\mathcal{T}}M$ is in $\mathcal{T}$ and $M_{\mathcal{F}}$ is in $%
\mathcal{F}$.
\end{enumerate}

In this case, we have that an object $X$ of $\mathcal{M}$ is in $\mathcal{T}$
if and only if $\mathrm{Hom}\left( X,C\right) =0$ for all $C$ in $\mathcal{F}
$, and it is in $\mathcal{F}$ if and only if $\mathrm{Hom}\left( T,X\right)
=0$ for all $T$ in $\mathcal{T}$. If $\left( \mathcal{T},\mathcal{F}\right) $
is a torsion pair for $\mathcal{M}$, then $\mathcal{T}$ and $\mathcal{F}$
are full subcategories of $\mathcal{M}$ (essentially) closed under
extensions \cite[Theorem 2]{rump_almost_2001}. Conversely, if $\mathcal{A}$
is a quasi-abelian category, let $\kappa :\mathrm{LH}\left( \mathcal{A}%
\right) \rightarrow \mathcal{A}$ be the left adjoint of the inclusion $%
\mathcal{A}\rightarrow \mathrm{LH}\left( \mathcal{A}\right) $. Thus, we have
that%
\begin{equation*}
\mathrm{Hom}_{\mathcal{A}}\left( \kappa \left( X\right) ,B\right) \cong 
\mathrm{Hom}_{\mathrm{LH}\left( \mathcal{A}\right) }\left( X,B\right)
\end{equation*}%
for all objects $B$ of $\mathcal{A}$. Define $\mathcal{T}$ to be the full
subcategory of $\mathrm{LH}\left( \mathcal{A}\right) $ spanned by the
objects $X$ such that $\kappa \left( X\right) =0$, and define $\mathcal{F}$
to be equal to $\mathcal{A}$. Then we have that $\left( \mathcal{T},\mathcal{%
F}\right) $ is a torsion pair in $\mathrm{LH}\left( \mathcal{A}\right) $ 
\cite[Theorem 2]{rump_almost_2001}.

Recall that if $\mathcal{C}$ is a category with finite products and initial
object $1$, then an (abelian) group in a category is a tuple $\left(
G,m,i,u\right) $ where $G$ is an object of $\mathcal{C}$, $m:G\times
G\rightarrow G$, $i:G\rightarrow G$, and $u:1\rightarrow G$ are morphisms in 
$\mathcal{C}$ that satisfy the \textquotedblleft group
axioms\textquotedblright\ (and the commutativity of the operation),
expressed in terms of morphisms of $\mathcal{C}$ \cite[Definition 4.1]%
{awodey_category_2006}. A homomorphism between group objects in $\mathcal{C}$
is a morphism in $\mathcal{C}$ between group objects that commutes with the
\textquotedblleft group operations\textquotedblright\ \cite[Definition 4.2]%
{awodey_category_2006}. In the category of (abelian) groups, every object
has a unique group structure, which is automatically abelian \cite[%
Proposition 4.5]{awodey_category_2006}.

\begin{definition}
A category of abelian groups is a category $\mathcal{C}$ with finite
products such that every object of $\mathcal{C}$ has a unique abelian group
structure.
\end{definition}

We regard an object of a category of abelian groups as an abelian group
object in $\mathcal{C}$ with respect to its unique abelian group structure.
Notice that a morphism in $\mathcal{C}$ is automatically a homomorphism
between group objects, with respect to such a unique abelian group structure.

For example, the category of abelian Polish groups is a category of abelian
groups. Likewise, any quasi-abelian subcategory of the category of abelian
Polish groups is a category of abelian groups.\ If $R$ is a ring, then the
category of (left) $R$-modules is a category of abelian groups.

Recall that an $\mathbf{Ab}$-category is a category enriched over the
(monoidal) category $\mathbf{Ab}$ of abelian groups. An additive category is
an $\mathbf{Ab}$-category with that has a \emph{zero object} and \emph{%
biproducts }\cite[Chapter VIII]{mac_lane_categories_1998}.

Suppose that $\mathcal{M}$ is a category of abelian groups. For objects $A,B$
of $\mathcal{M}$ and morphisms $\varphi ,\psi :A\rightarrow B$ define $%
\varphi +\psi :A\rightarrow B$ to be the morphism $m_{B}\circ \left( \varphi
\times \psi \right) \circ \Delta _{A}$ where $\Delta _{A}:A\rightarrow
A\times A$ is the canonical diagonal morphism, $m_{B}$ is the group
operation in $B$.

\begin{proposition}
Let $\mathcal{M}$ be a locally small category of abelian groups. Then $%
\mathcal{M}$ is an additive category with respect to the operation on $%
\mathrm{Hom}$-sets defined above.
\end{proposition}

\begin{proof}
Associativity of the operation follows from associativity of $m_{B}$. It is
easily seen that the unique morphism $0:A\rightarrow B$ that factors through
the initial object of $\mathcal{M}$ is the neutral element with respect to
the operation, while $i_{B}\circ \varphi $ is the opposite of a morphism $%
\varphi $. Since by assumption $\mathcal{M}$ has all finite products, it is
an additive category \cite[Chapter VIII, Section 2, Proposition 1 and
Theorem 2]{mac_lane_categories_1998}.
\end{proof}

We define a (quasi-)abelian category of abelian groups to be a locally small
category of abelian groups that is (quasi-)abelian with respect to the
additive category structure as in the previous proposition.

We would like to consider the natural generalization of $\mathbf{Ab}$%
-category where $\mathbf{Ab}$ is replaced with a category of abelian groups
that is not necessarily monoidal, such as the category of pro-Lie Polish
abelian groups.

\begin{definition}
\label{Definition:M-category}Let $\mathcal{M}$ be a quasi-abelian category
of abelian groups. An $\mathcal{M}$-category is a category $\mathcal{A}$
such that for every object $A$ and $B$ of $\mathcal{A}$, $\mathrm{Hom}_{%
\mathcal{A}}\left( A,B\right) $ is an object of $\mathcal{M}$, and $\mathrm{%
Hom}_{\mathcal{A}}\left( -,B\right) $ an $\mathrm{Hom}_{\mathcal{A}}\left(
A,-\right) $ are functors $\mathcal{A}\rightarrow \mathcal{M}$. Given $%
\mathcal{M}$-categories $\mathcal{A}$ and $\mathcal{B}$, an additive $%
\mathcal{M}$-functor $F:\mathcal{A}\rightarrow \mathcal{B}$ is a functor
such that, for every object $A,B$ in $\mathcal{A}$, $\mathrm{Hom}\left(
A,B\right) \rightarrow \mathrm{Hom}\left( FA,FB\right) $ is a morphism in $%
\mathcal{M}$.
\end{definition}

\begin{example}
If $R$ is a commutative ring, then $R$-modules form an abelian category of
abelian groups $\mathrm{Mod}\left( R\right) $. A $\mathrm{Mod}\left(
R\right) $-category in the sense of Definition \ref{Definition:M-category}
is the same as an $R$-linear category.
\end{example}

Notice that, if $\mathcal{A}$ is an $\mathcal{M}$-category, then $\mathcal{A}%
\times \mathcal{A}$ has a natural $\mathcal{M}$-category structure obtained
by the identification%
\begin{equation*}
\mathrm{Hom}\left( \left( A_{0},A_{1}\right) ,\left( B_{0},B_{1}\right)
\right) =\mathrm{Hom}\left( A_{0},B_{0}\right) \times \mathrm{Hom}\left(
A_{1},B_{1}\right) \text{.}
\end{equation*}%
Likewise, the category of arrows $\mathcal{A}^{\rightarrow }$ in $\mathcal{A}
$ has a natural $\mathcal{M}$-category structure. More generally, for every
finite category $J$, the category $\mathcal{A}^{J}$ of diagrams of shape $J$
in $\mathcal{A}$ is an $\mathcal{M}$-category.

\begin{definition}
Let $\mathcal{M}$ be a quasi-abelian category of abelian groups.

\begin{itemize}
\item An \emph{additive} $\mathcal{M}$-category is an $\mathcal{M}$-category 
$\mathcal{A}$ that has a zero object and biproducts, and such that the
biproduct functor $\mathcal{A}\times \mathcal{A}\rightarrow \mathcal{A}$ is
an $\mathcal{M}$-functor;

\item A \emph{quasi-abelian }$\mathcal{M}$-category is an additive $\mathcal{%
M}$-category that is also quasi-abelian, and such that the kernel and
cokernel functors are $\mathcal{M}$-functors, which implies that all the
finite limit and colimit functors are $\mathcal{M}$-functors;

\item A \emph{triangulated }$\mathcal{M}$-category is an additive $\mathcal{M%
}$-category $\mathcal{T}$ that is also triangulated, and such that the shift
functor $\Sigma $ on $\mathcal{T}$ is an $\mathcal{M}$-functor, and for al
objects $A,B$ of $\mathcal{T}$, $\mathrm{Hom}\left( A,-\right) $ is a
homological functor in the sense of \cite[Definition 1.1.7]%
{neeman_triangulated_2001} when regarded as a functor $\mathcal{T}%
\rightarrow \mathrm{LH}\left( \mathcal{M}\right) $, and $\mathrm{Hom}\left(
-,B\right) $ is a cohomological functor when regarded as a functor $\mathcal{%
T}^{\mathrm{op}}\rightarrow \mathrm{LH}\left( \mathcal{M}\right) $.
\end{itemize}
\end{definition}

\subsection{Derived functors}

Suppose that $\mathcal{A}$ and $\mathcal{R}$ are exact categories, and $F:%
\mathcal{A}\rightarrow \mathcal{R}$ is a functor. Then $F$ induces a \emph{%
triangulated }functor $\mathrm{K}^{+}\left( \mathcal{A}\right) \rightarrow 
\mathrm{K}^{+}\left( \mathcal{R}\right) $ \cite[Definition 2.1.1]%
{neeman_triangulated_2001}, which we still denote by $F$. A \emph{total
right derived functor }for $F:\mathrm{K}^{+}\left( \mathcal{A}\right)
\rightarrow \mathrm{K}^{+}\left( \mathcal{R}\right) $ \cite[Section 10.6]%
{buhler_exact_2010}---see also \cite[Definition 1.3.1]%
{schneiders_quasi-abelian_1999}---is a triangulated functor $\mathrm{R}F:%
\mathrm{D}^{+}\left( \mathcal{A}\right) \rightarrow \mathrm{D}^{+}\left( 
\mathcal{R}\right) $ together with a morphism $\mu :Q_{\mathcal{R}}\circ
F\Rightarrow \mathrm{R}F\circ Q_{\mathcal{A}}$ of triangulated functors that
satisfies the following universal property: for every triangulated functor $%
G:\mathrm{D}^{+}\left( \mathcal{A}\right) \rightarrow \mathrm{D}^{+}\left( 
\mathcal{R}\right) $ and morphism $\nu :Q_{\mathcal{R}}\circ F\rightarrow
G\circ Q_{\mathcal{A}}$ of triangulated functors, there exists a unique
morphism of triangulated functors $\sigma :\mathrm{R}F\Rightarrow G$ such
that $\nu =\left( \sigma Q_{\mathcal{A}}\right) \circ \mu $.

If $\mathcal{R}$ is quasi-abelian, then $\mathrm{H}^{0}\circ F$ is a \emph{%
cohomological functor }\cite[Section 1.1, Definition 3.1]%
{verdier_categories_1977}; see also \cite[Definition 1.1.7]%
{neeman_triangulated_2001}. A \emph{right cohomological derived functor }for 
$\mathrm{H}^{0}\circ F$ is defined in a similar fashion, via a suitable
universal property \cite[Section 2.2, Definition 1.4]%
{verdier_categories_1977}.

\begin{definition}
Suppose that $\mathcal{A}$ is an exact category, and $\mathcal{C}$ is a full
subcategory of $\mathcal{A}$. We say that $\mathcal{C}$ is:

\begin{enumerate}
\item \emph{generating} \cite[Definition 8.3.21(v)]%
{kashiwara_categories_2006} if for each object $A$ of $\mathcal{A}$ there
exists an admissible epimorphism $C\rightarrow A$ with $C$ in $\mathcal{C}$;

\item \emph{cogenerating} if $\mathcal{C}^{\mathrm{op}}$ is generating in $%
\mathcal{A}^{\mathrm{op}}$;

\item \emph{closed under quotients }if for all short exact sequences 
\begin{equation*}
0\rightarrow A^{\prime }\rightarrow C\rightarrow A^{\prime \prime
}\rightarrow 0
\end{equation*}%
in $\mathcal{A}$ with $C$ in $\mathcal{C}$, $A^{\prime \prime }$ is
isomorphic to an object of $\mathcal{C}$;

\item \emph{closed under subobjects }if $\mathcal{C}^{\mathrm{op}}$ is
closed under quotients in $\mathcal{A}^{\mathrm{op}}$.
\end{enumerate}
\end{definition}

We have the following result \cite[Proposition 13.2.2(b)]%
{kashiwara_categories_2006}; see also \cite[Theorem 10.22 and Remark 10.23]%
{buhler_exact_2010}, \cite[Proposition 13.2.2]{kashiwara_categories_2006}, 
\cite[Lemma 1.3.3 and Lemma 1.3.4]{schneiders_quasi-abelian_1999}, and \cite[%
Corollary 3.10]{hoffmann_homological_2007}.

\begin{lemma}
\label{Lemma:same-derived}Suppose that $\mathcal{A}$ is an exact category
and $\mathcal{C}$ is a cogenerating fully exact subcategory of $\mathcal{A}$
closed under quotients. Then for any bounded complex $A$ in $\mathcal{A}$
there exists a bounded complex $C$ in $\mathcal{C}$ and a quasi-isomorphism $%
\eta :A\rightarrow C$ with $\eta ^{k}:A^{k}\rightarrow C^{k}$ an admissible
monic for every $k\in \mathbb{Z}$.

Furthermore, the inclusion $\mathcal{C}\rightarrow \mathcal{A}$ induces an
equivalence of categories%
\begin{equation*}
\mathrm{D}^{b}\left( \mathcal{C}\right) \rightarrow \mathrm{D}^{b}\left( 
\mathcal{A}\right) \text{.}
\end{equation*}
\end{lemma}

As a consequence of Lemma \ref{Lemma:same-derived} we have the following;
see \cite[Section 10.6]{buhler_exact_2010}, \cite[Definition 10.3.2,
Proposition 10.3.3, Corollary 13.3.8]{kashiwara_categories_2006}, \cite[%
Proposition 1.3.5]{schneiders_quasi-abelian_1999}, and \cite[Remark 4.10]%
{hoffmann_homological_2007}.

\begin{proposition}
\label{Proposition:explicitly-right-derivable}Suppose that $F:\mathcal{A}%
\rightarrow \mathcal{R}$ is a functor between exact categories, and $%
\mathcal{C}$ is an cogenerating full subcategory of $\mathcal{A}$ closed
under quotients such that $F|_{\mathcal{C}}$ is exact. For a bounded complex 
$A$ over $\mathcal{A}$, pick a bounded complex $C_{A}$ over $\mathcal{C}$
together with a quasi-isomorphism $\eta _{A}:A\rightarrow C_{A}$.\ For
bounded complexes $A,B$ over $\mathcal{A}$ and morphisms $f:A\rightarrow B$
in $\mathrm{D}^{b}\left( \mathcal{A}\right) $, define%
\begin{equation*}
\left( \mathrm{R}F\right) \left( A\right) :=F\left( C_{A}\right)
\end{equation*}%
and 
\begin{equation*}
\left( \mathrm{R}F\right) \left( f\right) :=Q_{\mathcal{R}}\left( F\left(
g\right) \right) \circ Q_{\mathcal{R}}\left( F\left( \sigma \right) \right)
^{-1}\text{,}
\end{equation*}%
where $\sigma :C\rightarrow C_{A}$ and $g:C\rightarrow C_{B}$ are morphisms
in $\mathrm{K}\left( \mathcal{C}\right) $ such that 
\begin{equation*}
Q_{\mathcal{A}}\left( g\right) Q_{\mathcal{A}}\left( \sigma \right) ^{-1}=Q_{%
\mathcal{A}}(\eta _{B})\circ f\circ Q_{\mathcal{A}}(\eta _{A})^{-1}\text{.}
\end{equation*}

This yields a triangulated functor $\mathrm{R}F:\mathrm{D}^{b}\left( 
\mathcal{A}\right) \rightarrow \mathrm{D}^{b}\left( \mathcal{R}\right) $.
Defining $\mu _{A}:=F\left( \eta _{A}\right) :F\left( A\right) \rightarrow
\left( \mathrm{R}F\right) \left( A\right) $ for each bounded complex $A$
over $\mathcal{A}$ yields a morphism $\mu :Q_{\mathcal{R}}F\Rightarrow
\left( \mathrm{R}F\right) Q_{\mathcal{A}}$ of triangulated functors. We have
that $\left( \mathrm{R}F,\mu \right) $ is the total right derived functor of 
$F$. Furthermore, if $\mathcal{R}$ is a quasi-abelian category, then $%
\mathrm{H}^{0}\circ \mathrm{R}F$ is a cohomological right derived functor of 
$\mathrm{H}^{0}\circ F$.
\end{proposition}

Suppose that $\mathcal{A}$, $\mathcal{B}$, $\mathcal{R}$ are exact
categories, and $F:\mathcal{A}^{\mathrm{op}}\times \mathcal{B}\rightarrow 
\mathcal{R}$ is a functor. Given bounded complexes $A$ over $\mathcal{A}$
and $B$ over $\mathcal{B}$, we have a corresponding double complex $F\left(
A,B\right) $ over $\mathcal{R}$. We let $F^{\bullet }\left( A,B\right) $ be
the total complex associated with $F\left( A,B\right) $, which is
well-defined and bounded since $A$ and $B$ are bounded and hence $F\left(
A,B\right) $ has only finitely many nonzero entries in each diagonal. This
defines a triangulated functor $F^{\bullet }:\mathrm{K}^{b}\left( \mathcal{A}%
\right) ^{\mathrm{op}}\times \mathrm{K}^{b}\left( \mathcal{B}\right)
\rightarrow \mathrm{K}^{b}\left( \mathcal{R}\right) $; see \cite[Proposition
11.6.4]{kashiwara_categories_2006}.\ The same proof as Proposition \ref%
{Proposition:explicitly-right-derivable} gives the following.

\begin{proposition}
\label{Proposition:explicitly-right-derivable-bifunctor}Suppose that $F:%
\mathcal{A}^{\mathrm{op}}\times \mathcal{B}\rightarrow \mathcal{R}$ is a
functor between exact categories, $\mathcal{C}$ is an generating full
subcategory of $\mathcal{A}$ closed under subobjects, and $\mathcal{D}$ is a
cogenerating full subcategory of $\mathcal{B}$ closed under quotients.
Assume that for every object $C$ of $\mathcal{C}$ and $D$ of $\mathcal{D}$, $%
F\left( C,-\right) $ is exact on $\mathcal{D}$ and $F\left( -,D\right) $ is
exact on $\mathcal{C}$. For bounded complexes $A$ and $B$ over $\mathcal{A}$
and $\mathcal{B}$, respectively, pick bounded complexes $C_{A}$ and $D_{B}$
over $\mathcal{C}$ and $\mathcal{D}$, respectively, together with
quasi-isomorphisms $\eta _{A}:C_{A}\rightarrow A$ and $\eta
_{B}:B\rightarrow D_{B}$, and define%
\begin{equation*}
\left( \mathrm{R}F\right) \left( A,B\right) :=F\left( C_{A},D_{B}\right) 
\text{.}
\end{equation*}

This yields a triangulated functor $\mathrm{R}F^{\bullet }:\mathrm{D}%
^{b}\left( \mathcal{A}\right) ^{\mathrm{op}}\times \mathrm{D}^{b}\left( 
\mathcal{B}\right) \rightarrow \mathrm{D}^{b}\left( \mathcal{R}\right) $.
Defining $\mu _{A}:=F\left( \eta _{A},\eta _{B}\right) :F^{\bullet }\left(
A,B\right) \rightarrow \left( \mathrm{R}F\right) \left( A\right) $ for each
bounded complex $A$ over $\mathcal{A}$ yields a morphism $\mu :Q_{\mathcal{R}%
}F^{\bullet }\Rightarrow \left( \mathrm{R}F\right) \left( Q_{\mathcal{A}%
}\times Q_{\mathcal{B}}\right) $ of triangulated functors. We have that $%
\left( \mathrm{R}F,\mu \right) $ is the total right derived functor of $%
F^{\bullet }$. Furthermore, if $\mathcal{R}$ is a quasi-abelian category,
then $\mathrm{H}^{0}\circ \mathrm{R}F^{\bullet }$ is a cohomological right
derived functor of $\mathrm{H}^{0}\circ F^{\bullet }$.
\end{proposition}

\subsection{Injective and projective objects}

Suppose that $\mathcal{A}$ is an exact category. An object $I$ of $\mathcal{A%
}$ is called \emph{injective }\cite[Definition 11.1]{buhler_exact_2010} if
the functor $\mathrm{Ext}\left( -,I\right) :\mathcal{A}^{\mathrm{op}%
}\rightarrow \mathbf{Ab}$ is exact; see also \cite[Proposition 11.3]%
{buhler_exact_2010}. The category $\mathcal{A}$ \emph{has enough injectives }%
\cite[Definition 11.9]{buhler_exact_2010} if for every object $A$ of $%
\mathcal{A}$ there exists an admissible monic $m:A\rightarrow I$ where $I$
is injective. Let $\mathcal{A}$ be an exact category with enough injectives,
and let $\mathcal{I}$ be the class of injective objects in $\mathcal{A}$.
Then for every left-bounded complex $A$ in $\mathcal{A}$ there exists a
left-bounded complex $I_{A}$ over $\mathcal{I}$ and a quasi-isomorphism $\mu
_{A}:A\rightarrow I_{A}$ \cite[Theorem 12.7]{buhler_exact_2010}. In
particular, when $A$ is an object of $\mathcal{A}$, regarded as a complex
concentrated in degree zero, then $I_{A}$ is called an \emph{injective
resolution} of $A$ \cite[Definition 12.1]{buhler_exact_2010}.

Furthermore, the inclusion $\mathcal{I}\rightarrow \emph{A}$ induces an
equivalence of triangulated categories $\mathrm{K}^{+}\left( \mathcal{I}%
\right) \rightarrow \mathrm{D}^{+}\left( \mathcal{A}\right) $. If $\mathcal{B%
}$ is an exact category, and $F:\mathcal{A}\rightarrow \mathcal{B}$ is a
functor, then $F$ has a total right derived functor $\mathrm{R}F:\mathrm{D}%
^{+}\left( \mathcal{A}\right) \rightarrow \mathrm{D}^{+}\left( \mathcal{B}%
\right) $. This is defined by setting $\left( \mathrm{R}F\right) \left(
A\right) :=F\left( I_{A}\right) $ and $\left( \mathrm{R}F\right) \left(
f\right) =Q_{\mathcal{B}}\left( F\left( g\right) \right) $ for left-bounded
complexes $A$ and $B$ over $\mathcal{A}$ and morphism $f:A\rightarrow B$ in 
\textrm{D}$^{+}\left( \mathcal{A}\right) $, where $g:I_{\mathcal{A}%
}\rightarrow I_{\mathcal{B}}$ is a morphism in $\mathrm{K}\left( \mathcal{A}%
\right) $ satisfying $Q_{\mathcal{A}}\left( g\right) =Q_{\mathcal{A}}\left(
\mu _{B}\right) \circ f\circ Q_{\mathcal{A}}\left( \mu _{A}\right) ^{-1}$.

For objects $A$ and $B$ of $\mathcal{A}$ and $n\in \mathbb{Z}$, one defines $%
\mathrm{Ext}^{n}\left( A,B\right) :=\mathrm{Hom}_{\mathrm{D}\left( \mathcal{A%
}\right) }\left( A,B[n]\right) $. One says that $\mathcal{A}$ has \emph{%
finite homological dimension }if there exists $d\in \mathbb{Z}$ such that $%
\mathrm{Ext}^{n}\left( A,B\right) =0$ for all $n>d$ and objects $A$ and $B$
of $\mathcal{A}$. In this case, the least such $d$ is called the \emph{%
homological dimension }$\mathrm{hd}\left( \mathcal{A}\right) $ of $\mathcal{A%
}$ \cite[Exercise 13.8]{kashiwara_categories_2006}. If $\mathcal{A}$ has
enough injectives and $d\geq 0$, then $\mathrm{hd}\left( \mathcal{A}\right)
\leq d$ if and only if for every exact sequence $X^{0}\rightarrow
X^{1}\rightarrow \cdots \rightarrow X^{d}\rightarrow 0$ in $\mathcal{A}$
with $X^{0},X^{1},\ldots ,X^{d-1}$ injective, we have that $X^{d}$ is also
injective \cite[Exercise 13.8]{kashiwara_categories_2006}. In this case,
every object of $\mathcal{A}$ has an injective resolution $I_{A}$ with $%
I_{A}^{k}=0$ for $k\in \mathbb{Z}\setminus \left\{ 0,1,\ldots ,d\right\} $.

If $\mathcal{A}$ has enough injectives and finite homological dimension,
then for every (bounded) complex $A$ over $\mathcal{A}$ there exists a
quasi-isomorphism $\mu _{A}:A\rightarrow I_{A}$ where $I_{A}$ is a (bounded)
complex over $\mathcal{I}$ \cite[Corollary I.7.7]{iversen_cohomology_1986}.
Furthermore, the inclusion $\mathcal{I}\rightarrow \mathcal{A}$ induces an
equivalence of categories $\mathrm{K}\left( \mathcal{I}\right) \rightarrow 
\mathrm{D}\left( \mathcal{A}\right) $, which restricts to an equivalence of
categories $\mathrm{K}^{b}\left( \mathcal{I}\right) \rightarrow \mathrm{D}%
^{b}\left( \mathcal{A}\right) $ \cite[Proposition IX.2.12]%
{iversen_cohomology_1986}; see also \cite[Proposition 13.2.2]%
{kashiwara_categories_2006}. As a consequence of Proposition \ref%
{Proposition:explicitly-right-derivable-bifunctor} one obtains the following:

\begin{proposition}
\label{Proposition:explicitly-right-derivable-injectives}Let $\mathcal{A}$
and $\mathcal{R}$ be exact categories. Let $F:\mathcal{A}^{\mathrm{op}%
}\times \mathcal{A}\rightarrow \mathcal{R}$ be a functor. Suppose that
either:

\begin{itemize}
\item $\mathcal{A}$ has enough injectives, and $F\left( -,I\right) :\mathcal{%
A}^{\mathrm{op}}\rightarrow \mathcal{R}$ is exact for every injective object 
$I$ of $\mathcal{A}$, or

\item $\mathcal{A}$ has enough projectives, and $F\left( P,-\right) :%
\mathcal{A}^{\mathrm{op}}\rightarrow \mathcal{R}$ is exact for every
projective object $P$ of $\mathcal{A}$.
\end{itemize}

Suppose furthermore that $\mathcal{A}$ has homological dimension at most $1$%
. Then $F^{\bullet }:\mathrm{K}^{b}\left( \mathcal{A}\right) ^{\mathrm{op}%
}\times \mathrm{K}^{b}\left( \mathcal{A}\right) \rightarrow \mathrm{K}%
^{b}\left( \mathcal{R}\right) $ has a total right derived functor $\mathrm{R}%
F:\mathrm{D}^{b}\left( \mathcal{A}\right) ^{\mathrm{op}}\times \mathrm{D}%
^{b}\left( \mathcal{A}\right) \rightarrow \mathrm{D}^{b}\left( \mathcal{R}%
\right) $. If $\mathcal{R}$ is quasi-abelian, then $\mathrm{H}^{0}\circ 
\mathrm{R}F$ is a cohomological derived functor of $\mathrm{H}^{0}\circ
F^{\bullet }$.
\end{proposition}

As a particular instance of Proposition \ref%
{Proposition:explicitly-right-derivable-bifunctor}, one has the following.

\begin{corollary}
\label{Corollary:explicitly-derivable-hom}Suppose that $\mathcal{M}$ is a
quasi-abelian category of abelian groups, and $\mathcal{A}$ is a
quasi-abelian $\mathcal{M}$-category. Assume that $\mathcal{A}$ has enough
injectives or enough projectives, and homological dimension at most $1$.
Then $\mathrm{Hom}^{\bullet }:\mathrm{K}^{b}\left( \mathcal{A}\right) ^{%
\mathrm{op}}\times \mathrm{K}^{b}\left( \mathcal{A}\right) \rightarrow 
\mathrm{K}^{b}\left( \mathcal{M}\right) $ has a total right derived functor $%
\mathrm{RHom}:\mathrm{D}^{b}\left( \mathcal{A}\right) ^{\mathrm{op}}\times 
\mathrm{D}^{b}\left( \mathcal{A}\right) \rightarrow \mathrm{D}^{b}\left( 
\mathcal{M}\right) $, and $\mathrm{Ext}^{0}:=\mathrm{H}^{0}\circ \mathrm{RHom%
}$ is a cohomological derived functor of $\mathrm{Hom}_{\mathrm{K}\left( 
\mathcal{A}\right) }:\mathrm{K}^{b}\left( \mathcal{A}\right) ^{\mathrm{op}%
}\times \mathrm{K}^{b}\left( \mathcal{A}\right) \rightarrow \mathrm{LH}(%
\mathcal{M)}$. Furthermore, for bounded complexes $A$ and $B$ over $\mathcal{%
A}$, $\mathrm{Ext}^{0}\left( A,B\right) $ is naturally isomorphic as an
abelian group to $\mathrm{Hom}_{\mathrm{D}\left( \mathcal{A}\right) }\left(
A,B\right) $. This isomorphism turns $\mathrm{D}^{b}\left( \mathcal{A}%
\right) $ into a triangulated category enriched over $\mathrm{LH}\left( 
\mathcal{M}\right) $.
\end{corollary}

\begin{proof}
The first assertion is a particular instance of Proposition \ref%
{Proposition:explicitly-right-derivable-bifunctor}.

For the second assertion, observe that by \cite[Proposition 11.7.3]%
{kashiwara_categories_2006}, we have that $\mathrm{Hom}_{\mathrm{K}\left( 
\mathcal{A}\right) }$ is naturally isomorphic to $\mathrm{H}^{0}\circ 
\mathrm{Hom}^{\bullet }$. Therefore, by the first assertion we have that $%
\mathrm{RHom}$ is a total right derived functor of $\mathrm{Hom}^{\bullet }$%
, and $\mathrm{Ext}^{0}:=\mathrm{H}^{0}\circ \mathrm{RHom}$ is a
cohomological right derived functor of $\mathrm{Hom}_{\mathrm{K}\left( 
\mathcal{A}\right) }\cong \mathrm{H}^{0}\circ \mathrm{Hom}^{\bullet }$.
Furthermore, we have that $\mathrm{Ext}^{0}$ is naturally isomorphic to $%
\mathrm{Hom}_{\mathrm{D}^{b}\left( \mathcal{A}\right) }$ by \cite[%
Proposition 11.7.3]{kashiwara_categories_2006}. This allows one to regard $%
\mathrm{Hom}_{\mathrm{D}^{b}\left( \mathcal{A}\right) }$ as a functor to $%
\mathrm{D}^{b}\left( \mathcal{A}\right) ^{\mathrm{op}}\times \mathrm{D}%
^{b}\left( \mathcal{A}\right) \rightarrow \mathrm{LH}\left( \mathcal{M}%
\right) $, which turns $\mathrm{D}^{b}\left( \mathcal{A}\right) $ into a
triangulated category enriched over $\mathrm{LH}\left( \mathcal{M}\right) $.
\end{proof}

Under the hypotheses of Corollary \ref{Corollary:explicitly-derivable-hom},
one sets $\mathrm{Ext}^{n}\left( A,B\right) :=\mathrm{Ext}^{0}\left(
A,B[n]\right) $ for complexes $A$ and $B$ over $\mathcal{A}$ and $n\in 
\mathbb{Z}$. Then we have that $\mathrm{Ext}^{n}=0$ for $n\geq 2$.
Furthermore, for objects $A,B$ of $\mathcal{A}$, $\mathrm{Ext}\left(
A,B\right) :=\mathrm{Ext}^{1}\left( A,B\right) $ is isomorphic to the group $%
\mathrm{Ext}_{\mathrm{Yon}}\left( A,B\right) $ of \emph{isomorphism classes
of extensions} of $A$ by $B$; see \cite[Section XI.4]%
{iversen_cohomology_1986}. Notice that if $\mathcal{B}$ is a thick
subcategory of $\mathcal{A}$, then for objects $A$ and $B$ of $\mathcal{B}$,
the group \textrm{Ext}$_{\mathrm{Yon}}\left( A,B\right) $ is unchanged
whether it is computed in $\mathcal{B}$ or in $\mathcal{A}$; see \cite[%
Exercise 13.17]{kashiwara_categories_2006}.

\subsection{The left heart of thick categories of Polish abelian groups}

Let $\mathbf{PAb}$ be the category of Polish abelian groups and continuous
group homomorphisms. This is a quasi-abelian category \cite[Lemma 6.1]%
{lupini_looking_2022}. An explicit description of \textrm{LH}$\left( \mathbf{%
PAb}\right) $ was given in \cite{lupini_looking_2022} in terms of \emph{%
groups with a Polish cover}. An abelian group with a Polish cover is a group 
$G$ explicitly presented as a quotient $\hat{G}/N$, where $\hat{G}$ is a
Polish abelian group and $N\subseteq \hat{G}$ is a Polishable subgroup of $%
\hat{G}$. This means that $N$ is a Polish group with respect to some Polish
topology such that the inclusion $N\rightarrow \hat{G}$ is continuous. The
Polish topology on $N$ is in general not the subspace topology inherited
from $\hat{G}$ (unless $N$ is closed in $\hat{G}$). Nonetheless, the
topology on $N$ is in some sense \emph{induced }by the topology on $\hat{G}$%
, as it is the unique Polish topology on $N$ whose open sets are Borel in $%
\hat{G}$. Every Polish abelian group $G$ can be identified with the group
with a Polish cover $\hat{G}/N$ where $G=\hat{G}$ and $N$ is the trivial
subgroup of $G$. A group homomorphism $\varphi :G\rightarrow H$ between
groups with a Polish cover $G=\hat{G}/N$ and $H=\hat{H}/M$ is \emph{%
Borel-definable }if has a \emph{lift }to a Borel function $f:\hat{G}%
\rightarrow \hat{H}$, such that $\varphi \left( x+N\right) =f\left( x\right)
+M$ for every $x\in \hat{G}$.

The category $\mathrm{LH}\left( \mathbf{PAb}\right) $ is (equivalent to) the
category of groups with a Polish cover and Borel-definable group
homomorphisms \cite[Theorem 6.2]{lupini_looking_2022}. More generally, if $%
\mathcal{B}$ is a thick subcategory of $\mathbf{PAb}$, then $\mathrm{LH}%
\left( \mathcal{B}\right) $ is (equivalent to) the full subcategory of the
category of abelian groups with a Polish cover spanned by \emph{groups with
a }$\mathcal{B}$-cover, which are the groups with a Polish cover of the form 
$\hat{G}/N$ where both $\hat{G}$ and $N$ belong to $\mathcal{B}$ \cite[%
Theorem 6.13 and Proposition 6.15]{lupini_looking_2022}. As it is remarked
therein, a Borel-definable homomorphism between groups with a $\mathcal{B}$%
-cover is an isomorphism in $\mathrm{LH}\left( \mathcal{B}\right) $ if and
only if it is a bijection.

An abelian Polish group is \emph{non-Archimedean }if it has a basis of zero
neighborhoods consisting of subgroups. This is equivalent to the assertion
that $A$ is isomorphic to an inverse limit of countable groups.
Non-Archimedean Polish abelian groups form a thick subcategory of the
category of Polish abelian groups; see \cite[Theorem 6.17]%
{lupini_looking_2022}.

\section{Homological algebra for pro-Lie Polish abelian groups\label%
{Section:pro-Lie}}

\subsection{Borel cocycles for Polish groups\label{Subsection:Borel-cocycles}%
}

Let $G,H$ be abelian Polish groups. A (symmetric)\emph{\ }$2$-\emph{cocycle}
on $G$ with values in $H$ is a function $c:G\times G\rightarrow H$
satisfying the following identities for all $x,y,z\in G$:

\begin{itemize}
\item $c\left( x+y,z\right) +c\left( x,y\right) =c\left( x,y+z\right)
+c\left( y,z\right) $;

\item $c\left( x,y\right) =c\left( y,x\right) $;

\item $c\left( 0,x\right) =0$.
\end{itemize}

We say that a $2$-cocycle is Borel (respectively, continuous) if it is Borel
(respectively, continuous) as a function $G\times G\rightarrow H$. Given a
function $t:G\rightarrow H$, we define $\delta t:G\times G\rightarrow H$ to
be the function $\left( x,y\right) \mapsto t\left( x\right) +t\left(
y\right) -t\left( x+y\right) $. A Borel $2$-cocycle $c$ on $G$ with values
in $H$ is a coboundary if there exists a Borel function $t:G\rightarrow H$
such that $c=\delta t$. Borel $2$-cocycles on $G$ with values in $H$ form a
group $\mathrm{Z}^{1}\left( G,H\right) $ with respect to pointwise addition.
Coboundaries form a subgroup $\mathrm{B}^{1}\left( G,H\right) $ of \textrm{Z}%
$^{1}\left( G,H\right) $. We define $\mathrm{Ext}_{\mathrm{c}}\left(
G,H\right) $ to be the quotient group $\mathrm{Z}^{1}\left( G,H\right) /%
\mathrm{B}^{1}\left( G,H\right) $. This defines a functor $\mathrm{Ext}_{%
\mathrm{c}}:\mathbf{PAb}^{\mathrm{op}}\times \mathbf{PAb}\rightarrow \mathbf{%
Ab}$.

We let $\mathrm{Ext}_{\mathrm{Yon}}\left( C,A\right) $ be the group whose
elements are the isomorphism classes of abelian Polish group extensions $%
A\rightarrow X\rightarrow C$, where the group operation is induced by Baer
sum of extensions and the trivial element is the class of split extensions.
A short exact sequence $A\rightarrow X\rightarrow C$ in $\mathbf{PAb}$
yields an element of $\mathrm{Ext}_{\mathrm{c}}\left( C,A\right) $ as
follows. Pick a Borel right inverse $t:C\rightarrow X$ for the map $%
X\rightarrow C$, which exists by \cite[Theorem 12.17]{kechris_classical_1995}%
. Identifying $A$ with a closed subgroup of $X$, we have that for every $%
x,y\in C$, $\kappa \left( x,y\right) :=t\left( x+y\right) -t\left( x\right)
-t\left( y\right) $ belongs to $A$. This defines a Borel $2$-cocycle on $C$
with values in $A$, whose corresponding element of $\mathrm{Ext}_{\mathrm{c}%
}\left( C,A\right) $ is independent on the choice of $t$. This assignment
defines an injective group homomorphism $\mathrm{Ext}_{\mathrm{Yon}}\left(
C,A\right) \rightarrow \mathrm{Ext}_{\mathrm{c}}\left( C,A\right) $.

\begin{lemma}
\label{Lemma:Ext-Yon-lc}Suppose that $C$ and $A$ are locally compact Polish
abelian groups. Then the group homomorphism $\mathrm{Ext}_{\mathrm{Yon}%
}\left( C,A\right) \rightarrow \mathrm{Ext}_{\mathrm{c}}\left( C,A\right) $
is an isomorphism.
\end{lemma}

\begin{proof}
It suffices to show that it is surjective. Given a Borel cocycle $\kappa
:C\times C\rightarrow A$ one can define a corresponding extension $%
A\rightarrow X\rightarrow C$ as follows. One let $X$ be the group that has $%
A\times C$ as set of objects, and group operation defined by%
\begin{equation*}
\left( a,c\right) +\left( a^{\prime },c^{\prime }\right) =\left( a+a^{\prime
}+\kappa \left( c,c^{\prime }\right) ,c+c^{\prime }\right) \text{.}
\end{equation*}%
We also endow $A\times C$ with the product Borel structure and the Borel
measure defined as the product of Haar measures on $A$ and $G$. Then we have
that this is a $\sigma $-finite invariant Borel measure, and so by \cite[%
Theorem 1]{mackey_les_1957} there exists a unique locally compact Polish
group topology on $X$ that is compatible with its Borel structure and has
the given measure as a Haar measure. When endowed with this topology, the
canonical maps $A\rightarrow X\rightarrow C$ give a short exact sequence in $%
\mathbf{PAb}$. The image of the corresponding element in $\mathrm{Ext}_{%
\mathrm{Yon}}\left( C,A\right) $ under the homomorphism $\mathrm{Ext}_{%
\mathrm{Yon}}\left( C,A\right) \rightarrow \mathrm{Ext}_{\mathrm{c}}\left(
C,A\right) $ is equal to the element of $\mathrm{Ext}_{\mathrm{c}}\left(
C,A\right) $ represented by $\kappa $.
\end{proof}

\begin{lemma}
\label{Lemma:Ext-Yon-nA}Suppose that $C$ and $A$ are Polish abelian groups,
where $C$ is non-Archimedean.\ Then the image of the group homomorphism $%
\mathrm{Ext}_{\mathrm{Yon}}\left( C,A\right) \rightarrow \mathrm{Ext}_{%
\mathrm{c}}\left( C,A\right) $ is equal to the subgroup of elements
represented by \emph{continuous }cocycles.
\end{lemma}

\begin{proof}
Suppose that $A\rightarrow X\rightarrow C$ is a short exact sequence in $%
\mathbf{PAb}$, where $C$ is non-Archimedean. Then by \cite[Proposition 4.6]%
{bergfalk_definable_2020} we have that the map $X\rightarrow C$ has a \emph{%
continuous }right inverse $t:C\rightarrow X$. This yields a cocycle $\delta
t $ that represents an element of $\mathrm{Ext}_{\mathrm{c}}\left(
C,A\right) $ that is in the image of the element of $\mathrm{Ext}_{\mathrm{%
Yon}}\left( C,A\right) $ represented by the given extension.

Conversely if $\kappa :C\times C\rightarrow A$ is a continuous cocycle, then
one can define the extension $A\rightarrow X\rightarrow C$ letting $X$ be
the group that has $A\times C$ as set of objects, endowed with the product
topology, and group operation defined by%
\begin{equation*}
\left( a,c\right) +\left( a^{\prime },c^{\prime }\right) =\left( a+a^{\prime
}+\kappa \left( c,c^{\prime }\right) ,c+c^{\prime }\right) \text{.}
\end{equation*}
\end{proof}

The following lemma is \cite[Lemma 4.7]{bergfalk_applications_2023}.

\begin{lemma}
\label{Lemma:coboundary}Suppose that $C$ is an abelian Polish group and $A$
is an abelian Polish group. Suppose that $c:C\times C\rightarrow A$ is a
continuous cocycle, and $t:C\rightarrow A$ is a function such that $c=\delta
t$. If $t$ is Borel, then $t$ is continuous, and $c$ is a coboundary.
\end{lemma}

\begin{proof}
Define the Polish group $X:=A\rtimes _{c}C$ obtained by endowing the product 
$A\times C$ with the product topology and the group operation defined by $%
\left( a,x\right) +\left( a^{\prime },x^{\prime }\right) :=\left(
a+a^{\prime }+c\left( x,x^{\prime }\right) ,x+x^{\prime }\right) $ for $%
a,a^{\prime }\in A$ and $x,x^{\prime }\in C$. Observe that the map $a\mapsto
\left( a,0\right) $ is an inclusion of $A$ into $X$ as a closed subgroup,
giving an extension $A\rightarrow X\rightarrow C$.

Let also $A\times C$ be the product Polish group. We can define a Borel
group homomorphism $\varphi :X\rightarrow A\times C$ by setting%
\begin{equation*}
\varphi \left( a,x\right) :=\left( a+t\left( x\right) ,x\right) \text{.}
\end{equation*}%
Being a\ Borel group homomorphism, we must have that $\varphi $ is
continuous. Hence, $t:C\rightarrow A$ is continuous.
\end{proof}

Given a short exact sequence $A\rightarrow B\rightarrow C$ with the
corresponding cocycle $\kappa $ as above, and abelian Polish groups $X$ and $%
Y$, we define the corresponding \emph{boundary homomorphisms }$\mathrm{Hom}%
\left( X,C\right) \rightarrow \mathrm{Ext}_{\mathrm{c}}\left( X,A\right) $, $%
\varphi \mapsto \kappa \circ \left( \varphi \times \varphi \right) +\mathrm{B%
}^{1}\left( X,A\right) $ and \textrm{Hom}$\left( A,Y\right) \rightarrow 
\mathrm{Ext}_{\mathrm{c}}\left( C,Y\right) $, $\psi \mapsto \psi \circ
\kappa +\mathrm{B}^{1}\left( C,Y\right) $. The same proof as in the case of
discrete groups gives the following lemma; see \cite[Chapter IX]%
{fuchs_infinite_1970}. When $X,C$ are non-Archimedean, one can choose $%
\kappa $ to be continuous, whence the boundary homomorphism takes values in $%
\mathrm{Ext}_{\mathrm{Yon}}\left( C,A\right) $.

\begin{lemma}
\label{Lemma:exact-Ext-Yon}Let $X,Y$ be abelian Polish groups, and let $%
A\rightarrow B\overset{\pi }{\rightarrow }C$ be a short exact sequence of
abelian Polish groups.

\begin{enumerate}
\item We have exact sequences of abelian groups:%
\begin{equation*}
0\rightarrow \mathrm{Hom}\left( X,A\right) \rightarrow \mathrm{Hom}\left(
X,B\right) \rightarrow \mathrm{Hom}\left( X,C\right) \rightarrow \mathrm{Ext}%
_{\mathrm{c}}\left( X,A\right) \rightarrow \mathrm{Ext}_{\mathrm{c}}\left(
X,B\right) \rightarrow \mathrm{Ext}_{\mathrm{c}}\left( X,C\right)
\end{equation*}

\item If $X,C$ are non-Archimedean, then we have an exact sequence of
abelian groups:%
\begin{equation*}
0\rightarrow \mathrm{Hom}\left( X,A\right) \rightarrow \mathrm{Hom}\left(
X,B\right) \rightarrow \mathrm{Hom}\left( X,C\right) \rightarrow \mathrm{Ext}%
_{\mathrm{Yon}}\left( X,A\right) \rightarrow \mathrm{Ext}_{\mathrm{Yon}%
}\left( X,B\right) \rightarrow \mathrm{Ext}_{\mathrm{Yon}}\left( X,C\right)
\end{equation*}

\item If $A,B,C$ are non-Archimedean, then we have an exact sequence of
abelian groups:%
\begin{equation*}
0\rightarrow \mathrm{Hom}\left( C,Y\right) \rightarrow \mathrm{Hom}\left(
B,Y\right) \rightarrow \mathrm{Hom}\left( A,Y\right) \rightarrow \mathrm{Ext}%
_{\mathrm{Yon}}\left( C,Y\right) \rightarrow \mathrm{Ext}_{\mathrm{Yon}%
}\left( B,Y\right) \rightarrow \mathrm{Ext}_{\mathrm{Yon}}\left( A,Y\right)
\end{equation*}
\end{enumerate}
\end{lemma}

\begin{proof}
Let us fix a Borel right inverse $\sigma :C\rightarrow B$ for $\pi $, and
let $\kappa =\delta \sigma :C\times C\rightarrow A$. We identify $A$ with a
closed subgroup of $B$. By \cite[Proposition 4.6]{bergfalk_definable_2020},
if $C$ is non-Archimedean, then one can choose $\sigma $ to be continuous.

(1) and (2): It is clear that the image of $\mathrm{Hom}\left( X,C\right)
\rightarrow \mathrm{Ext}_{\mathrm{c}}\left( X,A\right) $ is contained in the
kernel of $\mathrm{Ext}_{\mathrm{c}}\left( X,A\right) \rightarrow \mathrm{Ext%
}_{\mathrm{c}}\left( X,B\right) $. We prove the converse implication.
Suppose that $c:X\times X\rightarrow A$ is a Borel cocycle such that there
exists a function $t:X\rightarrow B$ such that $c=\delta t$. Then we have
that $\pi t:X\rightarrow C$ is a group homomorphism. Then $\kappa \circ
\left( \pi t\times \pi t\right) $ is a cocycle cohomologous to $c$, as we
are about to show. Define $g:X\rightarrow A$, $x\mapsto \sigma \pi t\left(
x\right) -t\left( x\right) $. For $x,y\in X$,%
\begin{eqnarray*}
\left( \kappa \circ \left( \pi t\times \pi t\right) \right) \left(
x,y\right) &=&\kappa \left( \pi t\left( x\right) ,\pi t\left( y\right)
\right) \\
&=&\sigma \left( \pi t\left( x\right) \right) +\sigma \left( \pi t\left(
y\right) \right) -\sigma \left( \pi t\left( x+y\right) \right) \\
&=&t\left( x\right) +t\left( y\right) -t\left( x+y\right) +\delta g\left(
x,y\right) \\
&=&c\left( x,y\right) +\delta g\left( x,y\right) \text{.}
\end{eqnarray*}%
This concludes the proof that the image of $\mathrm{Hom}\left( X,C\right)
\rightarrow \mathrm{Ext}_{\mathrm{c}}\left( X,A\right) $ is equal to the
kernel of $\mathrm{Ext}_{\mathrm{c}}\left( X,A\right) \rightarrow \mathrm{Ext%
}_{\mathrm{c}}\left( X,B\right) $.

We now prove that the kernel of $\mathrm{Ext}_{\mathrm{c}}\left( X,B\right)
\rightarrow \mathrm{Ext}_{\mathrm{c}}\left( X,C\right) $ is contained in the
image of $\mathrm{Ext}_{\mathrm{c}}\left( X,A\right) \rightarrow \mathrm{Ext}%
_{\mathrm{c}}\left( X,B\right) $. Suppose that $c:X\times X\rightarrow B$ is
a Borel cocycle such that there is a Borel function $f:X\rightarrow C$ such
that $\pi c=\delta f$. Then we have that $c_{0}:=c-\delta \left( \sigma
f\right) $ is a Borel cocycle $X\times X\rightarrow A$ such that $%
c=c_{0}+\delta \left( \sigma f\right) $ and hence $c,c_{0}$ are cohomologous
as\ Borel cocycles $X\times X\rightarrow B$. This shows that the kernel of $%
\mathrm{Ext}_{\mathrm{c}}\left( X,B\right) \rightarrow \mathrm{Ext}_{\mathrm{%
c}}\left( X,C\right) $ is contained in the image of $\mathrm{Ext}_{\mathrm{c}%
}\left( X,A\right) \rightarrow \mathrm{Ext}_{\mathrm{c}}\left( X,B\right) $.

When $X,C$ are non-Archimedean, $\sigma $ is continuous. If $c$ is
continuous, then $f$ is continuous, and hence $c_{0}$ is continuous as well.
This shows that the kernel of $\mathrm{Ext}_{\mathrm{Yon}}\left( X,B\right)
\rightarrow \mathrm{Ext}_{\mathrm{Yon}}\left( X,C\right) $ is contained in
the image of $\mathrm{Ext}_{\mathrm{Yon}}\left( X,A\right) \rightarrow 
\mathrm{Ext}_{\mathrm{Yon}}\left( X,B\right) $ when $X,C$ are
non-Archimedean.

(3) Suppose now that $c:C\times C\rightarrow Y$ is a Borel cocycle that
represents an element of the kernel of $\mathrm{Ext}_{\mathrm{c}}\left(
C,Y\right) \rightarrow \mathrm{Ext}_{\mathrm{c}}\left( B,Y\right) $. Thus,
we have that $c\circ \left( \pi \times \pi \right) $ is a coboundary, and
there exists a Borel function $f:B\rightarrow Y$ such that $\delta f=c\circ
\left( \pi \times \pi \right) $. Hence,%
\begin{equation*}
f\left( b+b^{\prime }\right) =f\left( b\right) +f\left( b^{\prime }\right)
+c\left( \pi \left( b\right) ,\pi \left( b^{\prime }\right) \right)
\end{equation*}%
for $b,b^{\prime }\in B$. This implies that%
\begin{equation*}
f\left( 0\right) =0
\end{equation*}%
and%
\begin{equation*}
f\left( b\right) +f\left( -b\right) =-c\left( \pi \left( b\right) ,-\pi
\left( b\right) \right) \text{.}
\end{equation*}%
Then we have that $\varphi :=f|_{A}:A\rightarrow Y$ is a continuous group
homomorphism. We have that $\varphi \circ \kappa $ is a Borel cocycle
cohomologous to $c$. Indeed, define $g:=f\sigma :C\rightarrow Y$. For $%
x,y\in C$ we have that 
\begin{eqnarray*}
\left( \varphi \circ \kappa \right) \left( x,y\right) &=&\varphi \left(
\sigma \left( x\right) +\sigma \left( y\right) -\sigma \left( x+y\right)
\right) \\
&=&f\left( \sigma \left( x\right) +\sigma \left( y\right) -\sigma \left(
x+y\right) \right) \\
&=&f\left( \sigma \left( x\right) +\sigma \left( y\right) \right) +f\left(
-\sigma \left( x+y\right) \right) +c\left( x+y,-\left( x+y\right) \right) \\
&=&f\left( \sigma \left( x\right) \right) +f\left( \sigma \left( y\right)
\right) +c\left( x,y\right) -f\left( \sigma \left( x+y\right) \right)
-c\left( x+y,-\left( x+y\right) \right) +c\left( x+y,-\left( x+y\right)
\right) \\
&=&\left( c+\delta g\right) \left( x,y\right) \text{.}
\end{eqnarray*}%
This shows that the kernel of $\mathrm{Ext}_{\mathrm{c}}\left( C,Y\right)
\rightarrow \mathrm{Ext}_{\mathrm{c}}\left( B,Y\right) $ is contained in the
image of $\mathrm{Hom}\left( A,Y\right) \rightarrow \mathrm{Ext}_{\mathrm{c}%
}\left( C,Y\right) $. In particular, we have that the kernel of $\mathrm{Ext}%
_{\mathrm{Yon}}\left( C,Y\right) \rightarrow \mathrm{Ext}_{\mathrm{Yon}%
}\left( B,Y\right) $ is contained in the image of $\mathrm{Hom}\left(
A,Y\right) \rightarrow \mathrm{Ext}_{\mathrm{Yon}}\left( C,Y\right) $.

We now prove that the kernel of $\mathrm{Ext}_{\mathrm{Yon}}\left(
B,Y\right) \rightarrow \mathrm{Ext}_{\mathrm{Yon}}\left( A,Y\right) $ is
contained in the image of $\mathrm{Ext}_{\mathrm{Yon}}\left( C,Y\right)
\rightarrow \mathrm{Ext}_{\mathrm{Yon}}\left( B,Y\right) $. Suppose that $%
Y\rightarrow H\overset{p}{\rightarrow }B$ is an extension that represents an
element of the kernel of $\mathrm{Ext}_{\mathrm{Yon}}\left( B,Y\right)
\rightarrow \mathrm{Ext}_{\mathrm{Yon}}\left( A,Y\right) $. This means that
the induced extension $Y\rightarrow p^{-1}\left( A\right) \rightarrow A$
splits. Thus, there exists a continuous group homomorphism $\xi
:A\rightarrow H$ with closed image such that $p\xi $ is the inclusion of $A$
in $B$. In particular, we have that $\pi p\xi =0$. Thus, $\pi p$ induces a
continuous group homomorphism $\lambda :H/\xi \left( A\right) \rightarrow C$
such that $\lambda \left( h+\xi \left( A\right) \right) =\pi \left( p\left(
h\right) \right) $ for $h\in H$. This yields a commuting diagram%
\begin{equation*}
\begin{array}{ccccccccc}
&  &  &  & A & = & A &  &  \\ 
&  &  &  & \downarrow &  & \downarrow &  &  \\ 
0 & \rightarrow & Y & \rightarrow & H & \rightarrow & B & \rightarrow & 0 \\ 
&  &  &  & \downarrow &  & \downarrow &  &  \\ 
0 & \rightarrow & Y & \rightarrow & H/\xi \left( A\right) & \rightarrow & C
& \rightarrow & 0%
\end{array}%
\end{equation*}%
where the diagram%
\begin{equation*}
\begin{array}{ccc}
H & \rightarrow & B \\ 
\downarrow &  & \downarrow \\ 
H/\xi \left( A\right) & \rightarrow & C%
\end{array}%
\end{equation*}%
is a push-out. This yields a short exact sequence $Y\rightarrow H/\xi \left(
A\right) \rightarrow C$ that represents an element of $\mathrm{Ext}_{\mathrm{%
Yon}}\left( C,Y\right) $ whose image in $\mathrm{Ext}_{\mathrm{Yon}}\left(
B,Y\right) $ is the element represented by $Y\rightarrow H\rightarrow B$.
\end{proof}

Suppose that $G,H$ are abelian Polish groups, where $H=\prod_{n\in \alpha
}H_{n}$ for some $\alpha \leq \omega $. It is clear from the definition that 
\begin{equation*}
\mathrm{Ext}_{\mathrm{c}}\left( G,H\right) \cong \prod_{n\in \alpha }\mathrm{%
Ext}_{\mathrm{c}}\left( G,H_{n}\right) \text{.}
\end{equation*}%
Furthermore, if $\alpha <\omega $ then, since $\mathrm{Ext}_{\mathrm{c}%
}\left( -,G\right) $ is an additive functor, we also have%
\begin{equation*}
\mathrm{Ext}_{\mathrm{c}}(H,G)\cong \prod_{n\in \alpha }\mathrm{Ext}_{%
\mathrm{c}}\left( H_{n},G\right)
\end{equation*}

\subsection{Preliminaries on pro-Lie groups\label{Subsection:pro-Lie}}

In this section, we recall fundamental facts about the theory of pro-Lie
groups as can be found in \cite{hofmann_lie_2007}. We will only consider
abelian Polish groups. The category $\mathbf{PAb}$ of abelian Polish groups
and continuous group homomorphisms is a quasi-abelian category \cite[Theorem
6.2]{lupini_looking_2022}. Locally compact Polish abelian groups form a
thick subcategory $\mathbf{LCPAb}$ of $\mathbf{PAb}$ \cite[Theorem 6.17]%
{lupini_looking_2022}.

An abelian Polish group is a \emph{Lie group} if and only if it is of the
form $V\oplus T\oplus D$ where $D$ is countable discrete, $V$ is a
finite-dimensional vector group (isomorphic to $\mathbb{R}^{n}$ for some $%
n\in \omega $), and $T$ is a finite-dimensional torus (isomorphic to $%
\mathbb{T}^{d}$ for some $d\in \omega $); see \cite[Exercise E5.18]%
{hofmann_structure_2013}. Lie groups form a thick subcategory $\mathbf{LiePAb%
}$ of the category of locally compact abelian Polish groups; see \cite[%
Theorem 6.17]{lupini_looking_2022}.

A Polish abelian group $G$ has \emph{no small subgroups} if and only if it
has a zero neighborhood $U$ such that for every subgroup $N$ of $G$
contained in $U$, one has that $N=\left\{ 0\right\} $ \cite%
{moskowitz_homological_1967}. When $G$ is locally compact, this is
equivalent to the assertion that the Pontryagin dual $G^{\vee }$ is
compactly generated, as well as to the assertion that $G$ is a Lie group 
\cite[Theorem 2.4 and Corollary 1]{moskowitz_homological_1967}.

Let $G$ be an abelian Polish group. A closed subgroup $N$ of $G$ is co-Lie
if $G/N$ is a Lie group. Following \cite{hofmann_lie_2007}, we let $\mathcal{%
N}\left( G\right) $ be the collection of co-Lie closed subgroups of $G$.

\begin{definition}
An abelian Polish group is \emph{pro-Lie }if every zero neighborhood in $G$
contains an element of $\mathcal{N}\left( G\right) $.
\end{definition}

An abelian Polish group $G$ is a \emph{pro-Lie group} if every zero
neighborhood of $G$ contains a closed subgroup $N$ such that $G/N$ is a Lie
group. This implies that $\mathcal{N}\left( G\right) $ is closed under
intersections, and hence a filter basis \cite[page 148]{hofmann_lie_2007}.
It also implies that $G\cong \mathrm{\mathrm{lim}}_{N\in \mathcal{N}\left(
G\right) }G/N$. For a decreasing chain $\left( N_{k}\right) _{k\in \omega }$
in $\mathcal{N}\left( G\right) $, we say that $N_{k}\rightarrow 0$ if for
every zero neighborhood $U$ in $G$, $N_{k}$ is contained in $U$ eventually.
Notice that such a sequence exists for a pro-Lie Polish abelian group (as a
Polish abelian group has a countable basis of zero neighborhoods).

\begin{lemma}
\label{Lemma:countable-cofinality}Suppose that $G$ is a pro-Lie abelian
Polish group. Let $\left( N_{k}\right) _{k\in \omega }$ be a decreasing
chain in $\mathcal{N}\left( G\right) $.\ Then we have that $N_{k}\rightarrow
0$ if and only if $\left( N_{k}\right) _{k\in \omega }$ is cofinal in $%
\mathcal{N}\left( G\right) $. If this holds, then $G\cong \mathrm{\mathrm{lim%
}}_{k}G/N_{k}$.
\end{lemma}

\begin{proof}
Suppose that $N_{k}\rightarrow 0$.\ For $M\in \mathcal{N}\left( G\right) $,
we have that $G/M$ is a Lie group. Thus, it has a zero neighborhood $U$ that
does not contain any subgroup. Since $N_{k}\rightarrow 0$, we have that
there exists $k\in \omega $ such that $\pi \left( N_{k}\right) \subseteq U$,
where $\pi :G\rightarrow G/M$ is the quotient map. This implies that $\pi
\left( N_{k}\right) =\left\{ 0\right\} $ and $N_{k}\subseteq M$. The
converse implication is obvious. The second assertion follows from the fact
that $G\cong \mathrm{\mathrm{lim}}_{N\in \mathcal{N}\left( G\right) }G/N$.
\end{proof}

It is proved in \cite[Theorem 3.39]{hofmann_lie_2007} that a Polish abelian
group is pro-Lie if and only if it is the inverse limit of an inverse system
of Polish abelian Lie groups.

A Polish group $G$ is \emph{almost connected }if $G/c\left( G\right) $ is
compact, where $c\left( G\right) $ is the connected component of the trivial
element of $G$ (which is a closed subgroup of $G$) \cite[Definition 5.6]%
{hofmann_lie_2007}. If $G$ is an abelian pro-Lie group, then $G/c\left(
G\right) $ is non-Archimedean \cite[Theorem 5.20(iii)]{hofmann_lie_2007}.

The class of abelian Polish pro-Lie groups contains all locally compact
abelian Polish groups \cite[Example 5.1]{hofmann_lie_2007}, all
non-Archimedean abelian Polish groups, and all almost connected abelian
Polish groups \cite{yamabe_generalization_1953, yamabe_conjecture_1953}, and
is closed within the category of Polish groups under the following
operations \cite[Chapter 3]{hofmann_lie_2007}:

\begin{itemize}
\item taking closed subgroups;

\item taking quotients by closed subgroups;

\item taking countable limits.
\end{itemize}

An abelian pro-Lie group is called:

\begin{itemize}
\item a \emph{vector group} if it is isomorphic to $\mathbb{R}^{\alpha }$
for some $\alpha \leq \omega $;

\item a \emph{torus group }if it is isomorphic to $\mathbb{T}^{\beta }$ for
some $\beta \leq \omega $.
\end{itemize}

We also say that an abelian Polish pro-Lie group is \emph{vector-free }if it
has no nonzero closed subgroups that are vector groups. By \cite[Theorem 5.19%
]{hofmann_lie_2007} we have the following:

\begin{lemma}
\label{Lemma:injective-pro-Lie}Let $G$ be an abelian Polish pro-Lie group
and $H$ a closed subgroup of $G$. If $H$ is isomorphic to an abelian Polish
pro-Lie group of the form $T\oplus V$ where $T$ is a torus group and $V$ is
a vector group, then the short exact sequence $H\rightarrow G\rightarrow G/H$
splits.
\end{lemma}

An abelian Polish pro-Lie group $G$ admits a closed subgroup $V$, called 
\emph{maximal vector subgroup }(or vector group complement), that is a
vector group, and such that $H:=G/V$ is vector-free; see \cite[Theorem 5.20]%
{hofmann_lie_2007}. In this case, by Lemma \ref{Lemma:injective-pro-Lie} we
have that $G\cong V\oplus H$. Weil's Lemma for pro-Lie groups asserts the
following; see \cite[Theorem 5.3]{hofmann_lie_2007}.

\begin{lemma}
\label{Lemma:Weil}Let $G$ be an abelian Polish pro-Lie group and let $E$ be
either $\mathbb{Z}$ or $\mathbb{R}$. If $f:E\rightarrow G$ is a continuous
homomorphism, then exactly one of the following alternatives holds: either
the image of $f$ has compact closure, or $f$ is injective with closed image.
\end{lemma}

Let $G$ be an abelian Polish group. Define $\mathrm{comp}\left( G\right) $
to be the subgroup of $g\in G$ such that the subgroup generated by $g$ has
compact closure. If $G$ is an abelian Polish pro-Lie group, then $\mathrm{%
comp}\left( G\right) $ is a closed subgroup of $G$ \cite[Theorem 5.5]%
{hofmann_lie_2007}. One says that $G$ is \emph{elementwise compact }if $G=%
\mathrm{comp}\left( G\right) $ and \emph{compact-free }if $\mathrm{comp}%
\left( G\right) =\left\{ 0\right\} $; see \cite[Definition 5.4]%
{hofmann_lie_2007}. We have that $G/\mathrm{comp}\left( G\right) \cong
V\oplus S$ where $V$ is a maximal vector subgroup and $S$ is non-Archimedean
and compact-free \cite[Theorem 5.20(iv)]{hofmann_lie_2007}. By \cite[%
Proposition 5.43]{hofmann_lie_2007}, we have the following:

\begin{lemma}
\label{Lemma:structure-pro-Lie}If $G$ is an abelian Polish group, then there
exists a non-Archimedean closed subgroup $D$ of $G$ such that $G/D$ is a
torus group.
\end{lemma}

The following result is established in \cite[Theorem 4.1]{hofmann_lie_2007}.

\begin{lemma}
\label{Lemma:quotient}Suppose that $G$ is a pro-Lie group, and $H$ is a
closed subgroup of $G$. Suppose that $\left( N_{k}\right) $ is a cofinal
sequence in $\mathcal{N}\left( G\right) $. Then $G/H$ is pro-Lie, and $((%
\overline{N_{k}+H})/H)_{k\in \omega }$ is a cofinal sequence in $\mathcal{N}%
\left( G/H\right) $.
\end{lemma}

We can describe the groups of morphisms between abelian pro-Lie groups as
follows.

\begin{lemma}
\label{Lemma:Hom-Lie}Suppose that $G$ is a pro-Lie Polish abelian group and $%
H$ is a Lie abelian group.\ Then $\mathrm{Hom}\left( G,H\right) \cong 
\mathrm{co\mathrm{lim}}_{N\in \mathcal{N}\left( G\right) }\mathrm{Hom}\left(
G/N,H\right) $.
\end{lemma}

\begin{proof}
Since $H$ has no small subgroups, for every continuous group homomorphism $%
\varphi :G\rightarrow H$ there exists $N\in \mathcal{N}\left( G\right) $
such that $N\subseteq \mathrm{\mathrm{Ker}}\left( \varphi \right) $ and
hence $\varphi $ factors through a homomorphism $G/N\rightarrow H$.
\end{proof}

\begin{proposition}
\label{Proposition:Hom-pro-Lie}Suppose that $G$ and $H$ are pro-Lie Polish
abelian groups. Then%
\begin{equation*}
\mathrm{Hom}\left( G,H\right) \cong \mathrm{\mathrm{lim}}_{M\in \mathcal{N}%
\left( H\right) }\mathrm{co\mathrm{lim}}_{N\in \mathcal{N}\left( G\right) }%
\mathrm{Hom}\left( G/N,H/M\right) \text{.}
\end{equation*}
\end{proposition}

\begin{proof}
Since $H\cong \mathrm{\mathrm{lim}}_{M\in \mathcal{N}\left( H\right) }H/M$,
by the universal property of the limit we have%
\begin{equation*}
\mathrm{Hom}\left( G,H\right) \cong \mathrm{\mathrm{lim}}_{M\in \mathcal{N}%
\left( H\right) }\mathrm{Hom}\left( G,H/M\right) \text{.}
\end{equation*}%
By definition, for $M\in \mathcal{N}\left( H\right) $, $H/M$ is a Polish Lie
abelian group. Thus, the conclusion follows from Lemma \ref{Lemma:Hom-Lie}.
\end{proof}

We regard pro-Lie Polish abelian groups as a full subcategory $\mathbf{%
proLiePAb}$ of the quasi-abelian category $\mathbf{PAb}$ of Polish abelian
groups. We will soon prove that $\mathbf{proLiePAb}$ is in fact a thick
subcategory of $\mathbf{PAb}$.

When $G$ is a locally compact Polish abelian group and $H$ is a pro-Lie
Polish abelian group, we have that $\mathrm{Hom}\left( G,H\right) $ is a 
\emph{Polish} abelian group endowed with the compact-open topology \cite[%
Exercise 1.1.6]{gao_invariant_2009}, and the isomorphism%
\begin{equation*}
\mathrm{Hom}\left( G,H\right) \cong \mathrm{\mathrm{lim}}_{M\in \mathcal{N}%
\left( H\right) }\mathrm{Hom}\left( G,H/M\right)
\end{equation*}%
is as topological groups.

\begin{lemma}
\label{Lemma:vector-group}If $V$ is a Polish $\mathbb{R}$-vector space that
is a pro-Lie abelian Polish group, then $V\cong \mathbb{R}^{\alpha }$ for
some $\alpha \leq \omega $.
\end{lemma}

\begin{proof}
Let $\left( N_{k}\right) $ be a cofinal sequence in $\mathcal{N}\left(
V\right) $. Since $V$ is a Polish $\mathbb{R}$-vector space, the function $%
\mathrm{Hom}\left( \mathbb{R},V\right) \rightarrow V$, $\varphi \mapsto
\varphi \left( 1\right) $ is a topological isomorphism. We have that%
\begin{equation*}
\mathrm{Hom}\left( \mathbb{R},V\right) \cong \mathrm{\mathrm{lim}}_{k}%
\mathrm{Hom}\left( \mathbb{R},V/N_{k}\right)
\end{equation*}%
Since $V/N_{k}$ is a connected Polish Lie abelian group, we have that $%
V/N_{k}\cong \mathbb{R}^{d_{k}}\oplus \mathbb{T}^{m_{k}}$, and $W_{k}:=%
\mathrm{Hom}\left( \mathbb{R},V/N_{k}\right) $ is a finite-dimensional $%
\mathbb{R}$-vector space. Since $\mathrm{\mathrm{Ker}}\left(
W_{k+1}\rightarrow W_{k}\right) $ is a closed $\mathbb{R}$-subspace of $%
W_{k+1}$, it is a finite-dimensional $\mathbb{R}$-vector space and a direct
summand of $W_{k+1}$. Thus, we have that 
\begin{equation*}
V\cong \mathrm{Hom}\left( \mathbb{R},V\right) \cong \mathrm{\mathrm{lim}}%
_{k}W_{k}
\end{equation*}%
is isomorphic to a product of finite-dimensional $\mathbb{R}$-vector spaces,
and hence to $\mathbb{R}^{\alpha }$ for some $\alpha \leq \omega $.
\end{proof}

By Lemma \ref{Lemma:vector-group}, the vector groups are precisely the
pro-Lie Polish abelian groups that are Polish $\mathbb{R}$-vector spaces.
These are also called \emph{weakly complete topological Polish }$\mathbb{R}$-%
\emph{vector spaces in }\cite[Proposition 5.43]{hofmann_lie_2007}. We have
that if $V,W$ are vector groups, and $\varphi :V\rightarrow W$ is a
continuous homomorphism, then it is $\mathbb{R}$-linear and its image is a
closed subspace of $W$ which is a direct summand \cite[Theorem A2.12]%
{hofmann_lie_2007}.

\subsection{Extensions of abelian Polish pro-Lie groups\label%
{Subsection:extensions}}

The goal of this section is to prove the following result.

\begin{theorem}
\label{Theorem:pro-Lie-thick}Pro-Lie Polish abelian groups form a thick
subcategory of the category of Polish abelian groups.
\end{theorem}

In view of the results of \cite[Chapter 3]{hofmann_lie_2007}, in order to
establish Theorem \ref{Theorem:pro-Lie-thick} it suffices to prove that the
class of pro-Lie Polish abelian groups is closed under extensions.

\begin{lemma}
\label{Lemma:Ext-product}Suppose that $A\rightarrow B\rightarrow C$ is a
short exact sequence of abelian Polish groups, where $A=\prod_{n\in \omega
}A_{n}$ for Polish groups $A_{n}$ for $n\in \omega $. Then we have an
injective continuous homomorphism with closed image $\eta :B\rightarrow
\prod_{n\in \omega }B_{n}$ where, for every $n\in \omega $, there is a short
exact sequence of abelian Polish groups $A_{n}\rightarrow B_{n}\rightarrow C$%
. Furthermore, we have a commuting diagram%
\begin{equation*}
\begin{array}{ccc}
A & = & \prod_{n}A_{n} \\ 
\downarrow &  & \downarrow \\ 
B & \overset{\eta }{\rightarrow } & \prod_{n}B_{n}%
\end{array}%
\end{equation*}%
where the map%
\begin{equation*}
\prod_{n\in \omega }A_{n}\rightarrow \prod_{n\in \omega }B_{n}
\end{equation*}%
is induced by the maps $A_{n}\rightarrow B_{n}$ for $n\in \omega $.
\end{lemma}

\begin{proof}
We identify $A$ with a closed subgroup of $B$. Let $\pi :B\rightarrow C$ be
the quotient map. For $n\in \omega $, consider the pushout%
\begin{equation*}
\begin{array}{ccc}
A & \rightarrow  & B \\ 
p_{n}\downarrow  &  & \downarrow \eta _{n} \\ 
A_{n} & \overset{\sigma _{n}}{\rightarrow } & B_{n}%
\end{array}%
\end{equation*}%
where $p_{n}:A\rightarrow A_{n}$ is the canonical projection. Explicitly, we
have that%
\begin{equation*}
B_{n}:=A_{n}\oplus _{A}B
\end{equation*}%
is the quotient of $A_{n}\oplus B$ by the closed subgroup%
\begin{equation*}
\Xi _{n}:=\left\{ \left( -p_{n}(a),a\right) :a\in A\right\} \text{.}
\end{equation*}%
The map $\eta _{n}:B\rightarrow B_{n}$ is given by $b\mapsto \left(
0,b\right) +\Xi _{n}$. Then we have a short exact sequence 
\begin{equation*}
A_{n}\rightarrow B_{n}\overset{\pi _{n}}{\rightarrow }C
\end{equation*}%
where%
\begin{equation*}
\pi _{n}\left( \left( t,b\right) +\Xi _{n}\right) =\pi \left( b\right) .
\end{equation*}%
Define the continuous group homomorphism $\eta :B\rightarrow \prod_{n}B_{n}$%
, $b\mapsto \left( \eta _{n}\left( b\right) \right) _{n\in \omega }$. We
claim that $\eta $ is injective and has closed image. Indeed, suppose that $%
b\in B$ is such that $\eta \left( b\right) =0$. This gives that $b\in A$ and 
$p_{n}\left( b\right) =0$ for every $n\in \omega $, and hence $b=0$. This
shows that $\eta $ is injective.

Consider the continuous homomorphism $\tau :\prod_{n\in \omega
}B_{n}\rightarrow C^{\omega }$, $\left( x_{n}\right) \mapsto \left( \pi
_{n}\left( x_{n}\right) \right) $, and let%
\begin{equation*}
\Delta _{C}=\left\{ \left( c_{n}\right) _{n\in \omega }\in C^{\omega
}:\forall n\in \omega ,c_{n}=c_{0}\right\} \subseteq C^{\omega }\text{.}
\end{equation*}%
Then we have that the image of $\eta $ is equal to the preimage of $\Delta
_{C}$ under $\tau $. Indeed, it is clear that $\tau \circ \eta $ has image
contained in $\Delta _{C}$. Conversely, suppose that 
\begin{equation*}
\left( \left( t_{n},b_{n}\right) +\Xi _{n}\right) _{n\in \omega }\in
\prod_{n\in \omega }B_{n}
\end{equation*}%
is mapped to $\Delta _{C}$ by $\tau $. This implies that $\pi \left(
b_{n}\right) =\pi \left( b_{0}\right) $ for every $n\in \omega $. Thus, for
every $n\in \omega $ there exists $a_{n}\in A$ such that 
\begin{equation*}
b_{n}=b_{0}+a_{n}\text{.}
\end{equation*}%
Define now $s_{n}=p_{n}\left( a_{n}\right) $ for $n\in \omega $, $s:=\left(
s_{n}\right) \in A$, and $t=\left( t_{n}\right) \in A$. We have%
\begin{equation*}
\left( \left( t_{n},b_{n}\right) +\Xi _{n}\right) _{n\in \omega }=\left(
\left( t_{n},b_{0}+a_{n}\right) +\Xi _{n}\right) _{n\in \omega }=\left(
\left( t_{n}+s_{n},b_{0}\right) +\Xi _{n}\right) _{n\in \omega }=\left(
\left( 0,b_{0}+t+s\right) +\Xi _{n}\right) _{n\in \omega }=\eta \left(
b_{0}+t+s\right)
\end{equation*}%
This concludes the proof.
\end{proof}

\begin{lemma}
\label{Lemma:Ext-subgroup}Suppose that $A\rightarrow B\rightarrow C$ is a
short exact sequence of abelian Polish groups. If $A$ is isomorphic to a
closed subgroup of an abelian Polish group $A^{\prime }$, then there exists
a short exact sequence of abelian Polish groups $A^{\prime }\rightarrow
B^{\prime }\rightarrow C$ such that $B$ is isomorphic to a closed subgroup
of $B^{\prime }$.
\end{lemma}

\begin{proof}
It suffices to consider the pushout diagram%
\begin{equation*}
\begin{array}{ccc}
A & \rightarrow & B \\ 
\downarrow &  & \downarrow \\ 
A^{\prime } & \rightarrow & B^{\prime }%
\end{array}%
\end{equation*}%
where the vertical map $A\rightarrow A^{\prime }$ is the inclusion of $A$
into $A^{\prime }$ as a closed subgroup.
\end{proof}

\begin{lemma}
\label{Lemma:lc-by-nA}Suppose that $A\rightarrow B\rightarrow C$ is a short
exact sequence of abelian Polish groups. If $A$ is non-Archimedean and $C$
is locally compact, then $B$ is pro-Lie.
\end{lemma}

\begin{proof}
If $A$ is countable, then $B$ is locally compact, and hence pro-Lie. If $%
A\cong \prod_{n\in \omega }A_{n}$ where, for every $n\in \omega $, $A_{n}$
is countable, then the conclusion follows from Lemma \ref{Lemma:Ext-product}%
. Finally, if $A$ is isomorphic to a closed subgroup of $\prod_{n\in \omega
}A_{n}$ where, for every $n\in \omega $, $A_{n}$ is countable, then the
conclusion follows from Lemma \ref{Lemma:Ext-subgroup}.
\end{proof}

\begin{lemma}
\label{Lemma:lc-by-proL}Suppose that $A\rightarrow B\rightarrow C$ is a
short exact sequence of abelian Polish groups. If $A$ is pro-Lie and $C$ is
locally compact, then $B$ is pro-Lie.
\end{lemma}

\begin{proof}
By Lemma \ref{Lemma:structure-pro-Lie}, we have a short exact sequence $%
D\rightarrow A\rightarrow T$ where $D$ is non-Archimedean and $T$ is a torus
group. Considering the pushout diagram%
\begin{equation*}
\begin{array}{ccc}
A & \rightarrow & B \\ 
\downarrow &  & \downarrow \\ 
T & \rightarrow & B_{T}%
\end{array}%
\end{equation*}%
we obtain a commuting diagram%
\begin{equation*}
\begin{array}{ccccc}
D & \rightarrow & D & \rightarrow & 0 \\ 
\downarrow &  & \downarrow &  & \downarrow \\ 
A & \rightarrow & B & \rightarrow & C \\ 
\downarrow &  & \downarrow &  & \downarrow \\ 
T & \rightarrow & B_{T} & \rightarrow & C%
\end{array}%
\end{equation*}%
whose rows and columns are short exact sequences. Since $C$ and $T$ are
locally compact, we have that $B_{T}$ is locally compact. Since $D$ is
non-Archimedean, we have that $B$ is pro-Lie by Lemma \ref{Lemma:lc-by-nA}.
\end{proof}

If $X$ is a countable set, then for a bounded function $f:X\rightarrow 
\mathbb{R}$ we set%
\begin{equation*}
\left\Vert f\right\Vert _{\infty }:=\mathrm{\mathrm{sup}}\left\{ \left\vert
f\left( x\right) \right\vert :x\in X\right\} \text{.}
\end{equation*}

\begin{lemma}
\label{Lemma:vanishing-bounded-cohomology}For every $\varepsilon >0$,
countable discrete group $A$, and bounded $2$-cocycle $c:A\times
A\rightarrow \mathbb{R}$ with $\left\Vert c\right\Vert _{\infty }\leq
\varepsilon $, there exists a bounded $t:A\rightarrow \mathbb{R}$ with $%
\left\Vert t\right\Vert _{\infty }\leq \varepsilon $ and $\delta t=c$.
\end{lemma}

\begin{proof}
Since $A$ is abelian and, in particular, amenable, we have that the bounded
cohomology group $\mathrm{H}_{\mathrm{b}}^{2}\left( A,\mathbb{R}\right) $ is
trivial by \cite[Theorem 3.6]{frigerio_bounded_2017}. Furthermore, $A$ has $%
n $-th vanishing modulus of $1$; see \cite[Definition 4.10 and Example 4.11]%
{fournier-facio_bounded_2022}. The statement of the lemma is just a
reformulation of this fact.
\end{proof}

\begin{lemma}
\label{Lemma:u-nA-projective-reals}Suppose that $V\rightarrow X\overset{\pi }%
{\rightarrow }C$ is a short exact sequence of abelian Polish groups. If $%
V\cong \mathbb{R}$ and $C\cong \prod_{n\in \omega }A_{n}$ where $A_{n}$ is a
countable abelian group for every $n\in \omega $, then the sequence $%
V\rightarrow X\rightarrow C$ splits.
\end{lemma}

\begin{proof}
We identify $C$ with $\prod_{n\in \omega }A_{n}$, and we identify $V$ with $%
\mathbb{R}$ and with a closed subgroup of $X$. For $n\in \omega $, we write $%
C_{>k}=\left\{ x\in C:\forall i\leq k,x_{i}=0\right\} $. Fix a compatible
complete invariant metric $d$ on $X$ \cite[Exercise 2.2.9]%
{gao_invariant_2009}. For $x\in V$ we let $\left\vert x\right\vert $ be its
absolute value in $\mathbb{R}$.

By \cite[Proposition 4.6]{bergfalk_definable_2020}, there exists a
continuous right inverse $\varphi :C\rightarrow X$ for $\pi $. Define by
recursion a strictly increasing sequence $\left( k_{n}\right) _{n\in \omega
} $ in $\omega $ such that, for all $n\in \omega $ and $x,y\in C_{>k_{n}}$, $%
d\left( \varphi \left( x\right) ,0\right) \leq 2^{-n}$ and $\left\vert
\delta \varphi \left( x,y\right) \right\vert \leq 2^{-n}$.

After replacing $A_{n}$ with $\prod_{i=k_{n}+1}^{k_{n+1}}A_{i}$ for $n\in
\omega $, we can assume without loss of generality that $k_{n}=n$. We have
that $\left\Vert \delta \varphi |_{A_{n}\times A_{n}}\right\Vert _{\infty
}\leq 2^{-n}$. By Lemma \ref{Lemma:vanishing-bounded-cohomology}, we have
that exists a bounded function $t_{n}:A_{n}\rightarrow V\cong \mathbb{R}$
such that $\delta t_{n}=\delta \varphi |_{A_{n}\times A_{n}}$ and $%
\left\Vert t_{n}\right\Vert _{\infty }\leq 2^{-n}$.

We define now $t:C\rightarrow V$ by setting 
\begin{equation*}
t\left( x\right) =\sum_{n\in \omega }t_{n}\left( x_{n}\right) \text{,}
\end{equation*}%
where we identify $x_{n}$ as an element of $A_{n}$. Define also the
continuous function $\psi :C\rightarrow X$ by setting%
\begin{equation*}
\psi \left( x\right) =\sum_{n\in \omega }\varphi \left( x_{n}\right)
\end{equation*}%
We notice that $\psi $ is a right inverse for $\pi $. Furthermore, we have
that, for $x,y\in C$,%
\begin{equation*}
\delta \psi \left( x,y\right) =\sum_{n\in \omega }\delta \varphi \left(
x_{n},y_{n}\right) =\sum_{n\in \omega }\delta t_{n}\left( x_{n},y_{n}\right)
=\delta t\left( x,y\right) \text{.}
\end{equation*}%
Therefore, we have that $\eta :=\psi -t:C\rightarrow X$ is a continuous
homomorphism that is a right inverse for $\pi $. This shows that the short
exact sequence $V\rightarrow X\rightarrow C$ splits.
\end{proof}

\begin{lemma}
\label{Lemma:u-nA-projective-torus}Suppose that $\mathbb{T}\rightarrow X%
\overset{\pi }{\rightarrow }C$ is a short exact sequence of abelian Polish
groups. If $C\cong \prod_{n\in \omega }A_{n}$ where $A_{n}$ is a countable
abelian group for every $n\in \omega $, then the sequence $\mathbb{T}%
\rightarrow X\rightarrow C$ splits.
\end{lemma}

\begin{proof}
We identify $\mathbb{T}$ with a closed subgroup of $X$ and with $\mathbb{R}/%
\mathbb{Z}$. We let $d_{\mathbb{T}}$ be a compatible complete invariant
metric on $\mathbb{T}$. For $\varepsilon >0$ define $\mathbb{T}_{\varepsilon
}=\left\{ x\in \mathbb{T}:d_{\mathbb{T}}\left( x,0\right) <\varepsilon
\right\} $. Fix $\varepsilon >0$ such that there exists a continuous
function $\rho :\mathbb{T}_{3\varepsilon }\rightarrow \mathbb{R}$ such that $%
\rho \left( z\right) +\mathbb{Z}=z$ for every $z\in \mathbb{T}_{3\varepsilon
}$, and $\rho \left( x+y\right) =\rho \left( x\right) +\rho \left( y\right) $
for all $x,y\in \mathbb{T}_{\varepsilon }$.

By \cite[Proposition 4.6]{bergfalk_definable_2020}, we have a continuous
right inverse $\varphi :C\rightarrow X$ for $\pi $. Define the continuous $%
\kappa :C\times C\rightarrow \mathbb{T}$ by $\kappa \left( x,y\right)
=\varphi \left( x+y\right) -\varphi \left( x\right) -\varphi \left( y\right) 
$. Since $\kappa $ is continuous, there exists $n_{0}\in \omega $ such that $%
d_{\mathbb{T}}\left( \kappa (x,y),0\right) <\varepsilon $ for every $x\in
C_{\geq n_{0}}:=\left\{ x\in C:\forall i<n_{0},x_{i}=0\right\} $.

Since $\mathrm{Ext}_{\mathrm{Yon}}\left( C_{<n_{0}},\mathbb{T}\right) =0$ by
injectivity of $\mathbb{T}$ in the category of locally compact abelian
Polish groups, where $C_{<n_{0}}=\left\{ x\in C:\forall i\geq
n_{0},x_{i}=0\right\} $, it suffices to show that the function $\pi |_{\pi
^{-1}\left( C_{\geq n}\right) }:\pi ^{-1}\left( C_{\geq n}\right)
\rightarrow C_{\geq n}$ has a right inverse that is a continuous group
homomorphism. Thus, we can assume without loss of generality that $n=0$ and $%
d_{\mathbb{T}}\left( \kappa \left( x,y\right) ,0\right) <\varepsilon $ for
every $x\in C$.

This implies that, setting $c:=\rho \circ \kappa $, one obtains a continuous 
$2$-cocycle $c:C\times C\rightarrow \mathbb{R}$ such that $c\left(
x,y\right) +\mathbb{Z}=\kappa \left( x,y\right) $ for every $x,y\in C$. By
Lemma \ref{Lemma:u-nA-projective-reals} and Lemma \ref{Lemma:Ext-Yon-nA}, we
have that $\mathrm{Ext}_{\mathrm{Yon}}\left( C,\mathbb{R}\right) =0$.
Therefore, $c$ is a coboundary, and hence $\kappa $ is a coboundary as well.
This implies that $\varphi $ has a right inverse that is a continuous
homomorphism.
\end{proof}

\begin{lemma}
\label{Lemma:u-nA}Suppose that $C$ is a non-Archimedean abelian Polish
group. Then there exists a continuous surjective homomorphism $\left( 
\mathbb{Z}^{\left( \omega \right) }\right) ^{\omega }\rightarrow C$.
\end{lemma}

\begin{proof}
We have that $C$ is isomorphic to a closed subgroup of $\prod_{n\in \omega
}A_{n}$ where, for every $n\in \omega $, $A_{n}$ is countable. For every $%
n\in \omega $, there exists a surjective homomorphism $\mathbb{Z}^{\left(
\omega \right) }\rightarrow A_{n}$. Hence, $C$ is isomorphic to a quotient
of a closed subgroup $\hat{C}$ of $\left( \mathbb{Z}^{\left( \omega \right)
}\right) ^{\omega }$. Let $S$ be a pruned tree on $\mathbb{Z}^{\left( \omega
\right) }$ such that $\hat{C}$ is equal to closed subgroup $[S]\subseteq
\left( \mathbb{Z}^{\left( \omega \right) }\right) ^{\omega }$ consisting of
the branches of $S$. For every $n\in \omega $, let $B_{n}:=S\cap \left( 
\mathbb{Z}^{\left( \omega \right) }\right) ^{n}$, which is a subgroup of $%
\left( \mathbb{Z}^{\left( \omega \right) }\right) ^{n}$. Since $S$ is
pruned, we have that the projection map $B_{n+1}\rightarrow B_{n}$ is onto.
Furthermore, we have that $\hat{C}\cong \lbrack S]\cong \mathrm{\mathrm{lim}}%
_{n}B_{n}$. Since $B_{n}\cong \mathbb{Z}^{\left( \omega \right) }$ and $%
\mathbb{Z}^{\left( \omega \right) }$ is projective for countable abelian
groups, we have that $B_{n+1}\cong B_{n}\oplus \mathrm{\mathrm{Ker}}\left(
\pi _{n+1}\right) $. Thus, we have that $\hat{C}\cong B_{0}\oplus
\prod_{n\geq 1}\mathrm{\mathrm{Ker}}\left( \pi _{n}\right) $. Since for
every countable group $B$ there exists a surjective homomorphism $\mathbb{Z}%
^{\left( \omega \right) }\rightarrow B$, the conclusion follows.
\end{proof}

\begin{lemma}
\label{Lemma:nA-by-torus}Suppose that $L\rightarrow X\overset{\pi }{%
\rightarrow }C$ is a short exact sequence of abelian Polish groups. If $%
L\cong \mathbb{R}^{\alpha }\oplus \mathbb{T}^{\beta }$ for $\alpha ,\beta
\leq \omega $ and $C$ is non-Archimedean, then the sequence $L\rightarrow
X\rightarrow C$ splits.
\end{lemma}

\begin{proof}
By Lemma \ref{Lemma:Ext-product} it suffices to consider the case when $L=%
\mathbb{R}$ or $L=\mathbb{T}$. In this case, Lemma \ref%
{Lemma:u-nA-projective-reals} and Lemma \ref{Lemma:u-nA-projective-torus}
prove the result when $C$ is product of countable groups. If $C$ is an
arbitrary non-Archimedean Polish abelian group, then by Lemma \ref%
{Lemma:u-nA} there exists a surjective continuous homomorphism $g:\left( 
\mathbb{Z}^{\left( \omega \right) }\right) ^{\omega }\rightarrow C$.
Considering the pullback diagram%
\begin{equation*}
\begin{array}{ccc}
Y & \rightarrow & \left( \mathbb{Z}^{\left( \omega \right) }\right) ^{\omega
} \\ 
\downarrow &  & \downarrow g \\ 
X & \rightarrow & C%
\end{array}%
\end{equation*}%
we obtain a commuting diagram%
\begin{equation*}
\begin{array}{ccccc}
0 & \rightarrow & \mathrm{ker}\left( g\right) & \rightarrow & \mathrm{ker}%
\left( g\right) \\ 
\downarrow &  & \downarrow &  & \downarrow \\ 
L & \rightarrow & Y & \rightarrow & \left( \mathbb{Z}^{\left( \omega \right)
}\right) ^{\omega } \\ 
\downarrow &  & \downarrow &  & \downarrow \\ 
L & \rightarrow & X & \rightarrow & C%
\end{array}%
\end{equation*}%
whose rows and columns are short exact sequences. Then the exact sequence $%
L\rightarrow Y\rightarrow \left( \mathbb{Z}^{\left( \omega \right) }\right)
^{\omega }$ splits as explained at the beginning of the proof of this lemma,
hence $Y\cong L\oplus \left( \mathbb{Z}^{\left( \omega \right) }\right)
^{\omega }$ is pro-Lie.\ Therefore, $X$ is pro-Lie, being quotient of a
pro-Lie group by a closed subgroup. (Notice that a surjective continuous
homomorphism between Polish groups is a quotient mapping by the Open Mapping
Theorem for Polish groups; see \cite[Corollary 2.3.4]{gao_invariant_2009}.)
By Lemma \ref{Lemma:injective-pro-Lie}, we conclude that the short exact
sequence $L\rightarrow X\rightarrow C$ splits.
\end{proof}

\begin{lemma}
\label{Lemma:nA-by-proL}Suppose that $A\rightarrow B\rightarrow C$ is a
short exact sequence of abelian Polish groups. If $A$ is pro-Lie and $C$ is
non-Archimedean, then $B$ is pro-Lie.
\end{lemma}

\begin{proof}
As in the proof of Lemma \ref{Lemma:lc-by-proL}, by Lemma \ref%
{Lemma:structure-pro-Lie} we have a short exact sequence $D\rightarrow
A\rightarrow T$ where $D$ is non-Archimedean and $T$ is a torus group.
Considering the pushout diagram%
\begin{equation*}
\begin{array}{ccc}
A & \rightarrow & B \\ 
\downarrow &  & \downarrow \\ 
T & \rightarrow & B_{T}%
\end{array}%
\end{equation*}%
we obtain a commuting diagram%
\begin{equation*}
\begin{array}{ccccc}
D & \rightarrow & D & \rightarrow & 0 \\ 
\downarrow &  & \downarrow &  & \downarrow \\ 
A & \rightarrow & B & \rightarrow & C \\ 
\downarrow &  & \downarrow &  & \downarrow \\ 
T & \rightarrow & B_{T} & \rightarrow & C%
\end{array}%
\end{equation*}%
whose rows and columns are short exact sequences. Since $C$ is
non-Archimedean and $T$ is a torus group, we have that the short exact
sequence $T\rightarrow B_{T}\rightarrow C$ splits by Lemma \ref%
{Lemma:nA-by-torus}. Consider the commuting diagram%
\begin{equation*}
\begin{array}{ccccc}
D & \rightarrow & B_{0} & \rightarrow & C \\ 
\downarrow &  & \downarrow &  & \downarrow \\ 
D & \rightarrow & B & \rightarrow & B_{T}\cong T\oplus C \\ 
\downarrow &  & \downarrow &  & \downarrow \\ 
0 & \rightarrow & T & \rightarrow & T%
\end{array}%
\end{equation*}%
with short exact rows and columns. We have that $B_{0}$ is non-Archimedean
since $D$ and $C$ are non-Archimedean. Hence, $B$ is pro-Lie by Lemma \ref%
{Lemma:lc-by-nA}.
\end{proof}

\begin{theorem}
\label{Theorem:extension-proL}Suppose that $A\rightarrow B\rightarrow C$ is
a short exact sequence of abelian Polish groups. Then we have that $B$ is
pro-Lie if and only if $A$ and $C$ are pro-Lie.
\end{theorem}

\begin{proof}
We just need to prove that if $A$ and $C$ are pro-Lie, then $B$ is pro-Lie,
as the other implication is established in \cite[Chapter 3]{hofmann_lie_2007}%
. By Lemma \ref{Lemma:structure-pro-Lie}, we have a short exact sequence $%
D\rightarrow C\rightarrow T$ where $D$ is non-Archimedean and $T$ is a torus
group. Considering the pullback diagram%
\begin{equation*}
\begin{array}{ccc}
B_{D} & \rightarrow & D \\ 
\downarrow &  & \downarrow \\ 
B & \rightarrow & C%
\end{array}%
\end{equation*}%
we obtain a commuting diagram%
\begin{equation*}
\begin{array}{ccccc}
A & \rightarrow & B_{D} & \rightarrow & D \\ 
\downarrow &  & \downarrow &  & \downarrow \\ 
A & \rightarrow & B & \rightarrow & C \\ 
\downarrow &  & \downarrow &  & \downarrow \\ 
0 & \rightarrow & T & \rightarrow & T%
\end{array}%
\end{equation*}%
with exact rows and columns. Since $D$ is non-Archimedean, we have that $%
B_{D}$ is pro-Lie by Lemma \ref{Lemma:nA-by-proL}. Since $T$ is locally
compact, we have that $B$ is pro-Lie by Lemma \ref{Lemma:lc-by-proL}.
\end{proof}

Theorem \ref{Theorem:pro-Lie-thick} is an immediate consequence of Theorem %
\ref{Theorem:extension-proL}.

\subsection{Type decomposition for pro-Lie Polish abelian groups}

We recall the type decomposition for locally compact Polish abelian groups
as described by Hoffmann and Spitzweck in \cite[Section 2]%
{hoffmann_homological_2007}.

\begin{definition}
\label{Definition:lc-types}Let $A$ be a locally compact Polish abelian
group. Then we say that:

\begin{itemize}
\item $A$ is a topological $p$-group if, for every $x\in A$, the sequence $%
\left( p^{n}x\right) _{n\in \omega }$ is vanishing or, equivalently, $A$ has
a basis of zero neighborhoods consisting of open subgroups $U$ such that $%
A/U $ is a $p$-group \cite[Chapter 2]{armacost_structure_1981};

\item $A$ is a topological torsion group if, for every $x\in A$, the
sequence $\left( n!x\right) _{n\in \omega }$ is vanishing or, equivalently, $%
A$ has a basis of zero neighborhoods consisting of open subgroups $U$ such
that $A/U$ is a torsion group \cite[Chapter 3]{armacost_structure_1981};

\item $A$ is type $\mathbb{Z}$ if it is discrete and torsion-free;

\item $A$ is type $\mathbb{S}^{1}$ if it is compact and connected or,
equivalently, its Pontryagin dual $A^{\vee }$ is type $\mathbb{Z}$;

\item $A$ is type $\mathbb{R}$ if it is a vector group;

\item $A$ is type $\mathbb{A}$ if $A\cong V\oplus B$ where $V$ is a vector
group and $B$ is a topological torsion group.
\end{itemize}
\end{definition}

For a locally compact Polish abelian group $A$, it is proved in \cite[%
Proposition 2.2]{hoffmann_homological_2007} that there exist canonical short
exact sequences $F_{\mathbb{Z}}A\rightarrow A\rightarrow A_{\mathbb{Z}}$ and 
$A_{\mathbb{S}^{1}}\rightarrow F_{\mathbb{Z}}A\rightarrow A_{\mathbb{A}}$
where $A_{\mathbb{Z}}$, $A_{\mathbb{S}^{1}}$, and $A_{\mathbb{A}}$ have type 
$\mathbb{Z}$, $\mathbb{S}^{1}$, and $\mathbb{A}$, respectively. Furthermore,
we have $A_{\mathbb{A}}=A_{\mathbb{R}}\oplus A_{\mathrm{t}}$ where $A_{%
\mathbb{R}}$ is a finite-dimensional vector group and $A_{\mathrm{t}}$ is a
topological torsion group.

If $P$ is any of the properties from Definition \ref{Definition:lc-types},
we say a pro-Lie Polish abelian group $A$ satisfies $P$ if and only if $%
\left\{ N\in \mathcal{N}\left( A\right) :A/N\text{ satisfies }P\right\} $ is
cofinal in $\mathcal{N}\left( A\right) $ (ordered by reverse inclusion). We
say that $A$ has sub-type $\mathbb{Z}$ if it is a Polishable subgroup of a
type $\mathbb{Z}$ pro-Lie Polish abelian group.

\begin{lemma}
\label{Definition:pro-Lie-types}Suppose that $A$ is an abelian Polish
pro-Lie group.

\begin{enumerate}
\item $A$ is a topological $p$-group if and only if $A$ is isomorphic to a
closed subgroup of $\mathbb{Z}\left( p^{\infty }\right) ^{\omega }$ if and
only if the sequence $\left( p^{n}x\right) _{n\in \omega }$ is vanishing for
every $x\in A$;

\item $A$ is a topological torsion group if and only if $A$ is isomorphic to
a closed subgroup of $\prod_{p}\mathbb{Z}\left( p^{\infty }\right) ^{\omega
} $ if and only if the sequence $\left( n!x\right) _{n\in \omega }$ is
vanishing for every $x\in A$;

\item $A$ is type $\mathbb{Z}$ if and only if $A$ is isomorphic to a closed
subgroup of $(\mathbb{Q}^{\left( \omega \right) })^{\omega }$;

\item $A$ is type $\mathbb{S}^{1}$ if and only if $A$ is compact connected;

\item $A$ is type $\mathbb{R}$ if and only if it is a vector group;

\item $A$ is type $\mathbb{A}$ if and only if $A\cong V\oplus B$ where $B$
is a vector group and $B$ is a topological torsion group.
\end{enumerate}
\end{lemma}

\begin{proof}
We give details for (6), the other assertions being easy to see.

Suppose that $A$ is type $\mathbb{A}$. Then we have that $A\cong V\oplus B$
where $V$ is a vector group and $B$ is vector-free. We claim that $B$ is a
topological torsion group. If $N\in \mathcal{N}\left( B\right) $ then there
exist $L\in \mathcal{N}\left( V\right) $ and $M\in \mathcal{N}\left(
B\right) $ contained in $N$ such that 
\begin{equation*}
\frac{V\oplus B}{L\oplus M}\cong \left( V/L\right) \oplus \left( B/M\right)
\end{equation*}%
is type $\mathbb{A}$, and hence $B/M$ is type $\mathbb{A}$. We can write $%
B/M=W\oplus S$ where $W$ is a vector group and $S$ is a topological torsion
group. Since $B$ is vector-free, we must have that the composition $%
B\rightarrow B/M\rightarrow W$ where $B\rightarrow B/M$ is the quotient map
and $B/M\rightarrow W$ is the coordinate projection, is the zero
homomorphism. This shows that $W=0$ and $B/M$ is a topological torsion group.

The converse implication follows from (1) and (5).
\end{proof}

For locally compact abelian Polish groups $A$ and $B$, we have that $\mathrm{%
Hom}\left( A,B\right) =0$ whenever:

\begin{itemize}
\item $A$ has type $\mathbb{S}^{1}$ and $B$ has type $\mathbb{A}$ or $%
\mathbb{Z}$;

\item $A$ has type $\mathbb{A}$ and $B$ has type $\mathbb{Z}$.
\end{itemize}

It follows from this and Proposition \ref{Proposition:Hom-pro-Lie} that the
same conclusions hold when $A$ and $B$ are pro-Lie Polish abelian groups.

If $A$ is a pro-Lie Polish abelian group, then by \cite[Theorem 5.20]%
{hofmann_lie_2007} we have that $A$ has a largest closed vector subgroup,
which we denote by $A_{\mathbb{R}}$. Then we have that $A=A_{\mathbb{R}%
}\oplus B$ where $B$ is vector-free. We also let $A_{\mathbb{S}^{1}}$ be the
connected component of the trivial element in $B$, which is \emph{compact }%
\cite[Theorem 24(ii)]{hofmann_lie_2007}. We let $c\left( A\right) $ be the
connected component of the trivial element in $A$.

\begin{lemma}
\label{Lemma:connected-compact}Suppose that $\left( A^{\left( k\right)
}\right) $ is a tower of compact connected Lie groups. Then $\mathrm{\mathrm{%
lim}}_{k}A^{\left( k\right) }$ is connected.
\end{lemma}

\begin{proof}
Define $B^{\left( k\right) }=\bigcap_{n>k}\mathrm{Ran}\left( A^{\left(
n\right) }\rightarrow A^{\left( k\right) }\right) \subseteq A^{\left(
k\right) }$ for $k\in \omega $. Then, by compactness, $\left( B^{\left(
k\right) }\right) $ is an epimorphic tower of compact connected Lie groups
with $\mathrm{\mathrm{lim}}_{k}B^{\left( k\right) }=\mathrm{\mathrm{lim}}%
_{k}A^{\left( k\right) }$. Thus, without loss of generality we can assume
that $\left( A^{\left( k\right) }\right) $ is an epimorphic tower. If $%
U\subseteq \mathrm{\mathrm{lim}}_{k}A^{\left( k\right) }$ is an open
subgroup, then for every $k\in \omega $, $\pi ^{\left( k\right) }\left(
U\right) \subseteq A^{\left( k\right) }$ is an open subgroup. Thus, $\pi
^{\left( k\right) }\left( U\right) =A^{\left( k\right) }$ and hence $U$ is
dense in $\mathrm{\mathrm{lim}}_{k}A^{\left( k\right) }$. This implies that $%
U=\mathrm{\mathrm{lim}}_{k}A^{\left( k\right) }$, showing that $\mathrm{%
\mathrm{lim}}_{k}A^{\left( k\right) }$, being a compact group, is connected.
\end{proof}

\begin{lemma}
If $V$ is a\emph{\ locally compact} vector group, then $V$ is projective in
the category of pro-Lie Polish abelian groups.x
\end{lemma}

\begin{proof}
Suppose that $Y$ is a pro-Lie Polish abelian group and $\varphi
:Y\rightarrow V$ is a surjective continuous group homomorphism. Since $V$ is
a Lie group, we can find a tower $\left( Y_{n},\eta _{n}:Y_{n+1}\rightarrow
Y_{n}\right) $ of abelian Lie groups and a continuous group homomorphism $%
\varphi _{0}:Y_{0}\rightarrow V$ such that $Y=\mathrm{\mathrm{lim}}_{n}Y_{n}$
and $\varphi =\varphi _{0}\circ \pi _{0}$ where $\pi _{0}:Y\rightarrow Y_{0}$
is the canonical map. By projectivity of $V$ in the category of locally
compact Polish abelian groups, we can find a continuous homomorphism $\psi
_{0}:V\rightarrow Y_{0}$ that is a right inverse for $\varphi _{0}$. Since
the image of $\psi _{0}$ is a closed subgroup of $Y_{0}$ isomorphic to $V$,
we can find a continuous homomorphism $\psi _{1}:\psi _{0}\left( V\right)
\rightarrow Y_{1}$ that is a right inverse for $\eta _{0}|_{\eta
_{0}^{-1}(\psi _{0}(V))}$, whence $\psi _{1}\circ \psi _{0}$ is a continuous
homomorphism that is a right inverse for $\varphi _{0}\circ \eta _{0}$.
Proceeding in this fashion, we can define recursively a continuous
homomorphisms $\psi _{n}:\left( \psi _{n-1}\circ \cdots \circ \psi
_{0}\right) \left( V\right) \rightarrow Y_{n}$ that is a right inverse for $%
\eta _{n}|_{\eta _{n}^{-1}((\psi _{n}\circ \cdots \circ \psi _{0})(V))}$,
whence $\psi _{n}\circ \cdots \circ \psi _{0}$ is a continuous homomorphism
that is a right inverse for $\varphi _{0}\circ \eta _{0}\circ \cdots \circ
\eta _{n-1}$.

One can then define the continuous group homomorphism $\psi :V\rightarrow Y$%
, $x\mapsto \left( \left( \psi _{n}\circ \cdots \circ \psi _{0}\right)
(x)\right) _{n\in \omega }$, which is a right inverse for $\varphi $.
\end{proof}

\begin{corollary}
If $Y$ is a vector-free pro-Lie Polish abelian group, then for every $N\in 
\mathcal{N}\left( Y\right) $, $Y/N$ is a vector-free abelian\ Lie group.
\end{corollary}

\begin{lemma}
\label{Lemma:connected-component-surjective}Suppose that $\pi :B\rightarrow
A $ is a continuous surjective homomorphism between Lie Polish abelian
groups. Then its restriction $B_{0}\rightarrow A_{0}$ is surjective.
\end{lemma}

\begin{proof}
We can write $B=B_{0}\oplus C$ where $C$ is countable. Thus, we have that $%
\pi \left( B_{0}\right) $ is an analytic subgroup of $A_{0}$ of countable
index, whence it is open. Since $A_{0}$ is connected, this implies that $\pi
\left( B_{0}\right) =A_{0}$.
\end{proof}

\begin{lemma}
\label{Lemma:connected-component-tower}Suppose that $A$ is a pro-Lie Polish
abelian group. Let $\left( A^{\left( k\right) }\right) $ be an epimorphic
tower with $A=\mathrm{\mathrm{lim}}_{k}A^{\left( k\right) }$. Then $%
(c(A^{\left( k\right) }))$ is an epimorphic tower.
\end{lemma}

\begin{proof}
This follows from Lemma \ref{Lemma:connected-component-surjective}.
\end{proof}

\begin{lemma}
\label{Lemma:connected-component vf}Let $A$ be a vector-free pro-Lie Polish
abelian group. Suppose that $\left( A^{\left( k\right) }\right) $ is an
epimorphic tower of Lie groups with $A=\mathrm{\mathrm{lim}}_{k}A^{\left(
k\right) }$. Then $c\left( A\right) =\mathrm{\mathrm{lim}}_{k}c\left(
A^{\left( k\right) }\right) $ is a compact connected subgroups of $A$.
\end{lemma}

\begin{proof}
Notice that, for every $k\in \omega $, $A^{\left( k\right) }$ is
vector-free. We have that $\mathrm{\mathrm{lim}}_{k}c(A^{\left( k\right) })$
is compact, and connected by Lemma \ref{Lemma:connected-compact}. This shows
that $\mathrm{\mathrm{lim}}_{k}c(A^{\left( k\right) })\subseteq c\left(
A\right) $. The converse inclusion follow from functoriality of $c\left(
A\right) $.
\end{proof}

\begin{lemma}
\label{Lemma:non-Archimedean-quotient vf}Suppose that $A$ is a vector-free
pro-Lie Polish abelian group. Then $A\left/ c\left( A\right) \right. $ is a
non-Archimedean Polish abelian group.
\end{lemma}

\begin{proof}
Let $\left( A^{\left( k\right) }\right) $ be an epimorphic tower of
(necessarily vector-free) Lie Polish abelian groups with $A=\mathrm{\mathrm{%
lim}}_{k}A^{\left( k\right) }$. For every $k\in \omega $, we have a short
exact sequence%
\begin{equation*}
c(A^{\left( k\right) })\rightarrow A^{\left( k\right) }\rightarrow C^{\left(
k\right) }
\end{equation*}%
where $C^{\left( k\right) }$ is countable. Since $(c\left( A\right) ^{\left(
k\right) })$ is an epimorphic tower by Lemma \ref%
{Lemma:connected-component-tower}, this implies that we have a short exact
sequence%
\begin{equation*}
\mathrm{\mathrm{lim}}_{k}c(A^{\left( k\right) })\rightarrow \mathrm{\mathrm{%
lim}}_{k}A^{\left( k\right) }\rightarrow \mathrm{\mathrm{lim}}_{k}C^{\left(
k\right) }
\end{equation*}%
Since $c\left( A\right) =\mathrm{\mathrm{lim}}_{k}c(A^{\left( k\right) })$
by Lemma \ref{Lemma:connected-component-tower}, $A=\mathrm{\mathrm{lim}}%
_{k}A^{\left( k\right) }$, and $\mathrm{\mathrm{lim}}_{k}C^{\left( k\right)
} $ is non-Archimedean, the conclusion follows.
\end{proof}

\begin{lemma}
Suppose that $A$ is a non-Archimedean Polish group. Then $A$ has a largest
topological torsion closed subgroup $A_{\mathrm{t}}$. Furthermore, $A\left/
A_{\mathrm{t}}\right. $ has sub-type $\mathbb{Z}$.
\end{lemma}

\begin{proof}
Let $\left( A^{\left( k\right) }\right) $ be an epimorphic tower of
countable groups with $A=\mathrm{\mathrm{lim}}_{k}A^{\left( k\right) }$. For 
$k\in \omega $, let $T\left( A^{\left( k\right) }\right) $ be the torsion
subgroup of $A^{\left( k\right) }$, and set $C^{\left( k\right) }:=A^{\left(
k\right) }\left/ T\left( A^{\left( k\right) }\right) \right. $. Then we have
an exact sequence%
\begin{equation*}
0\rightarrow \mathrm{\mathrm{lim}}_{k}T(A^{\left( k\right) })\rightarrow 
\mathrm{\mathrm{lim}}_{k}A^{\left( k\right) }\rightarrow \mathrm{\mathrm{lim}%
}_{k}C^{\left( k\right) }\text{.}
\end{equation*}%
Setting $A_{\mathrm{t}}:=\mathrm{\mathrm{lim}}_{k}T\left( A^{\left( k\right)
}\right) $ and $C:=\mathrm{\mathrm{lim}}_{k}C^{\left( k\right) }$, we have
that $A_{\mathrm{t}}$ is a topological torsion group and $C$ has type $%
\mathbb{Z}$. The above argument shows that $A\left/ A_{\mathrm{t}}\right. $
is a Polishable subgroup of $C$, and hence of sub-type $\mathbb{Z}$.

If $B$ is a topological torsion closed subgroup of $A$, then $\mathrm{Hom}%
\left( B,C\right) =0$ and hence $B\subseteq A_{\mathrm{t}}$.
\end{proof}

\begin{lemma}
\label{Lemma:connected-tower}Suppose that $\left( A^{\left( k\right)
}\right) $ is a tower of connected Lie groups. Then $\mathrm{\mathrm{lim}}%
_{k}A^{\left( k\right) }$ is connected.
\end{lemma}

\begin{proof}
As in the proof of Lemma \ref{Lemma:connected-compact}, we can assume
without loss of generality that $\left( A^{\left( k\right) }\right) $ is an
epimorphic tower. We can write $A^{\left( k\right) }=V^{\left( k\right)
}\oplus T^{\left( k\right) }$ where $V^{\left( k\right) }$ is a
finite-dimensional vector group and $T^{\left( k\right) }$ is a
finite-dimensional torus group. Furthermore, since $\mathrm{Hom}\left(
T^{\left( k+1\right) },V^{\left( k\right) }\right) =0$ for every $k\in
\omega $, $\left( V^{\left( k\right) }\right) $ is an epimorphic tower.
Therefore, we have an exact sequence%
\begin{equation*}
0\rightarrow \mathrm{\mathrm{lim}}_{k}T^{\left( k\right) }\rightarrow 
\mathrm{lim}_{k}A^{\left( k\right) }\rightarrow \mathrm{\mathrm{lim}}%
_{k}V^{\left( k\right) }\rightarrow \mathrm{\mathrm{lim}}_{k}^{1}T^{\left(
k\right) }\text{.}
\end{equation*}%
Notice that, for every $k\in \omega $, if $d_{k}$ is the dimension of $%
T^{\left( k\right) }$, then $d_{k}+1$ is the maximum length of a chain of
compact connected subgroups of $T^{\left( k\right) }$. It follows that the
chain of subgroups $\left( \mathrm{Ran}\left( T^{\left( n\right)
}\rightarrow T^{\left( k\right) }\right) \right) _{n>k}$ stabilizes. As this
holds for every $k\in \omega $, $\mathrm{\mathrm{lim}}_{k}^{1}T^{\left(
k\right) }=0$. By projectivity of vector groups, and since $\mathrm{Ran}%
\left( V^{\left( k+1\right) }\rightarrow V^{\left( k\right) }\right) $ is a
closed vector subgroup of $V^{\left( k\right) }$ for every $k\in \omega $ by 
\cite[Theorem A2.12]{hofmann_lie_2007}, we have that $V:=\mathrm{\mathrm{lim}%
}_{k}V^{\left( k\right) }$ is a vector group. Hence, by projectivity of
vector groups, we have that $\mathrm{\mathrm{lim}}_{k}A^{\left( k\right)
}\cong \mathrm{\mathrm{lim}}_{k}T^{\left( k\right) }\oplus V$, where $%
\mathrm{\mathrm{lim}}_{k}T^{\left( k\right) }$ is connected by the previous
lemma.
\end{proof}

\begin{corollary}
\label{Corollary:connected-component}Let $A$ be a pro-Lie Polish abelian
group.

\begin{enumerate}
\item Suppose that $\left( A^{\left( k\right) }\right) $ is an epimorphic
tower of Lie groups with $A=\mathrm{\mathrm{lim}}_{k}A^{\left( k\right) }$.
Then $c\left( A\right) =\mathrm{\mathrm{lim}}_{k}c\left( A^{\left( k\right)
}\right) $.

\item $A/c\left( A\right) $ is a non-Archimedean Polish abelian group.
\end{enumerate}
\end{corollary}

\begin{proof}
The proofs of Lemma \ref{Lemma:connected-component vf} and Lemma \ref%
{Lemma:non-Archimedean-quotient vf} apply verbatim with Lemma \ref%
{Lemma:connected-compact} replaced with Lemma \ref{Lemma:connected-tower}.
\end{proof}

\begin{lemma}
\label{Lemma:non-Archimedean-quotient}Suppose that $A$ is a pro-Lie Polish
abelian group. Then $A\left/ c\left( A\right) \right. $ is a non-Archimedean
Polish abelian group.
\end{lemma}

\begin{proof}
Let $\left( A^{\left( k\right) }\right) $ be an epimorphic tower of Lie
Polish abelian groups with $A=\mathrm{\mathrm{lim}}_{k}A^{\left( k\right) }$%
. For every $k\in \omega $, we have a short exact sequence%
\begin{equation*}
c\left( A^{\left( k\right) }\right) \rightarrow A^{\left( k\right)
}\rightarrow C^{\left( k\right) }
\end{equation*}%
where $C^{\left( k\right) }$ is countable. Since $(c\left( A\right) ^{\left(
k\right) })$ is an epimorphic tower by Lemma \ref%
{Lemma:connected-component-tower}, we have a short exact sequence%
\begin{equation*}
\mathrm{\mathrm{lim}}_{k}c(A^{\left( k\right) })\rightarrow \mathrm{\mathrm{%
lim}}_{k}A^{\left( k\right) }\rightarrow \mathrm{\mathrm{lim}}_{k}C^{\left(
k\right) }
\end{equation*}%
By Corollary \ref{Corollary:connected-component}, $c\left( A\right) =\mathrm{%
\mathrm{lim}}_{k}c(A^{\left( k\right) })$. Hence, $A/c\left( A\right) \cong 
\mathrm{\mathrm{lim}}_{k}C^{\left( k\right) }$ is non-Archimedean.
\end{proof}

\begin{lemma}
Suppose that $A$ is a non-Archimedean Polish group. Then $A$ has a largest
topological torsion closed subgroup $A_{\mathrm{t}}$. Furthermore, $A\left/
A_{\mathrm{t}}\right. $ has sub-type $\mathbb{Z}$.
\end{lemma}

\begin{proof}
Let $\left( A^{\left( k\right) }\right) $ be an epimorphic tower of
countable groups with $A=\mathrm{\mathrm{lim}}_{k}A^{\left( k\right) }$. For 
$k\in \omega $, let $T\left( A^{\left( k\right) }\right) $ be the torsion
subgroup of $A^{\left( k\right) }$, and set $C^{\left( k\right) }:=A^{\left(
k\right) }\left/ T\left( A^{\left( k\right) }\right) \right. $. Then we have
an exact sequence%
\begin{equation*}
0\rightarrow \mathrm{\mathrm{lim}}_{k}T(A^{\left( k\right) })\rightarrow 
\mathrm{\mathrm{lim}}_{k}A^{\left( k\right) }\rightarrow \mathrm{\mathrm{lim}%
}_{k}C^{\left( k\right) }\text{.}
\end{equation*}%
Setting $A_{\mathrm{t}}:=\mathrm{\mathrm{lim}}_{k}T\left( A^{\left( k\right)
}\right) $ and $C:=\mathrm{\mathrm{lim}}_{k}C^{\left( k\right) }$, we have
that $A_{\mathrm{t}}$ is a topological torsion group and $C$ has type $%
\mathbb{Z}$. The above argument shows that $A\left/ A_{\mathrm{t}}\right. $
is a Polishable subgroup of $C$, and hence of sub-type $\mathbb{Z}$.

If $B$ is a topological torsion closed subgroup of $A$, then $\mathrm{Hom}%
\left( B,C\right) =0$ and hence $B\subseteq A_{\mathrm{t}}$.
\end{proof}

\begin{lemma}
We have that $\mathbb{R}^{\omega }$ is projective in $\mathbf{proLiePAb}$.
\end{lemma}

\begin{proof}
Suppose that $X$ is a pro-Lie Polish abelian group and $\varphi
:X\rightarrow \mathbb{R}^{\omega }$ is a continuous surjective
homomorphism.\ We prove that it has a right inverse that is a continuous
group homomorphism. We can write $X=V\oplus Y$ where $V$ is a vector group
and $Y$ has no nonzero closed vector subgroups. We have short exact
sequences $Y_{\mathbb{S}^{1}}\rightarrow F_{\mathbb{Z}}Y\rightarrow Y_{%
\mathrm{t}}$ and $F_{\mathbb{Z}}Y\rightarrow Y\rightarrow Y_{\mathbb{Z}}$,
where $\varphi |_{Y_{\mathbb{S}^{1}}}=0$, and the induced function $Y_{%
\mathrm{t}}\rightarrow \mathbb{R}^{\omega }$ is also trivial. Thus, if $\psi
:Y_{\mathbb{Z}}\rightarrow \mathbb{R}^{\omega }$ is the continuous
homomorphism induced by $\varphi |_{Y}$, then $\mathrm{Ran}\left( \psi
\right) =\mathrm{Ran}\left( \varphi |_{Y}\right) $.

By \cite[Theorem A2.12]{hofmann_lie_2007}, $\mathrm{Ran}\left( \varphi
|_{V}\right) $ is a closed subgroup and a topological direct summand of $%
\mathbb{R}^{\omega }$.

Let $\left( N_{k}^{V}\right) $ be a cofinal sequence in $\mathcal{N}(V)$, $%
\left( N_{k}^{Y}\right) $ be a cofinal sequence in $\mathcal{N}(Y)$, and set 
$N_{k}:=N_{k}^{V}\oplus N_{k}^{Y}$ for $k\in \omega $. Thus, $\left(
N_{k}\right) $ is a cofinal sequence in $\mathcal{N}\left( X\right) $. For $%
n\in \omega $, set 
\begin{equation*}
M_{n}=\left\{ x\in \mathbb{R}^{\omega }:\forall i\leq n\text{, }%
x_{i}=0\right\} \subseteq \mathbb{R}^{\omega }\text{.}
\end{equation*}%
We claim that $\mathrm{Ran}\left( \varphi |_{V}\right) =\mathbb{R}^{\omega }$%
. It suffices to prove that, for every $n\in \omega $, $\pi _{M_{n}}\circ
\varphi |_{V}:V\rightarrow \mathbb{R}^{n}$ is surjective. We have that $\pi
_{M_{n}}\circ \varphi :X\rightarrow \mathbb{R}^{n}$ factors through $X/N_{k}$
for some $k\in \omega $.\ Set $Z_{k}:=X/N_{k}$ and let $\pi
_{N_{k}}:X\rightarrow Z_{k}$ be the quotient map. Then we have that there
exists a continuous homomorphism $\psi :Z_{k}\rightarrow \mathbb{R}^{n}$
such that $\psi \pi _{N_{k}}=\pi _{M_{n}}\varphi $. We have $Z_{k}\cong
V/N_{k}^{V}\oplus Y/N_{k}^{Y}$. Since $Y$ has no nonzero vector subgroups,
it is easily seen considering the type decompositions of $Y$ and $Y/N_{k}$
that the image of $\psi |_{Y/N_{k}^{Y}}:Y/N_{k}^{Y}\rightarrow \mathbb{R}%
^{n} $ is countable. Since $\mathbb{R}^{n}$ has no nontrivial analytic
subgroups of countable index (as any such a subgroup must be open), it
follows that $\psi |_{V/N_{k}^{V}}$ is surjective, and hence $\pi
_{M_{n}}\varphi |_{V}$ is surjective.

We have therefore shown that $\varphi |_{V}:V\rightarrow \mathbb{R}^{\omega
} $ is surjective. Since $\mathrm{\mathrm{Ker}}\left( \varphi |_{V}\right) $
is a closed $\mathbb{R}$-subspace of $V$, it is a vector group. Therefore,
the short exact sequence $\mathrm{\mathrm{Ker}}\left( \varphi |_{V}\right)
\rightarrow V\rightarrow \mathbb{R}^{\omega }$ splits, and there exists a
continuous homomorphism $\psi :\mathbb{R}^{\omega }\rightarrow V\subseteq X$
that is a right inverse for $\varphi |_{V}$. Thus, $\psi $ is also a right
inverse for $\varphi $, if regarded as a continuous group homomorphism with
codomain $X$.
\end{proof}

If $A$ is a pro-Lie Polish abelian group, we define $A_{\mathrm{t}}$ to be
the largest topological torsion closed subgroup of $A\left/ c\left( A\right)
\right. $, and $F_{\mathbb{Z}}A$ to be the preimage of $A_{\mathrm{t}}$
under the quotient map $A\rightarrow A\left/ c\left( A\right) \right. $.
Define also $A_{\mathbb{Z}}:=A\left/ F_{\mathbb{Z}}A\right. $ and $A_{%
\mathbb{S}^{1}}=\mathrm{comp}\left( c\left( A\right) \right) $, which is the
set of elements of $c\left( A\right) $ that generate a subgroup with compact
closure.

By Lemma \ref{Lemma:R-projective}, we have that $c\left( A\right) =A_{%
\mathbb{R}}\oplus A_{\mathbb{S}^{1}}$ where $A_{\mathbb{R}}$ is the largest
closed vector subgroup of $A$. We also set $A_{\mathbb{A}}=A_{\mathbb{R}%
}\oplus A_{\mathrm{t}}$. Then we have canonical exact sequences%
\begin{equation*}
A_{\mathbb{S}^{1}}\rightarrow F_{\mathbb{Z}}A\rightarrow A_{\mathbb{A}}
\end{equation*}%
and%
\begin{equation*}
F_{\mathbb{Z}}A\rightarrow A\rightarrow A_{\mathbb{Z}}
\end{equation*}%
where $A_{\mathbb{S}^{1}}$, $A_{\mathbb{A}}$ have type $\mathbb{S}^{1}$ and $%
\mathbb{A}$, respectively, and $A_{\mathbb{Z}}$ has sub-type $\mathbb{Z}$.

\begin{remark}
In general it is not the case that $A_{\mathbb{Z}}$ has type $\mathbb{Z}$.
For example, consider the reduced product 
\begin{equation*}
G:=\prod_{n}\left( \mathbb{Z}:2\mathbb{Z}\right)
\end{equation*}%
consisting of sequences of integers that are eventually even. Then $G$ is a
Polishable subgroup of $\mathbb{Z}^{\omega }$ such that every open subgroup $%
U$ of $G$ is such that $G/U$ has nontrivial torsion subgroup. Indeed, there
exists $n\in \omega $ such that the set 
\begin{equation*}
U_{n}:=\left\{ x\in G:\forall i<\omega \text{, }x_{i}\in 2\mathbb{Z}\text{
and }\forall i<n\text{, }x_{i}=0\right\}
\end{equation*}%
is contained $U$, and hence $G/U$ is a quotient of $G/U_{n}\cong \mathbb{Z}%
^{n}\oplus \left( \mathbb{Z}/2\mathbb{Z}\right) ^{\left( \omega \right) }$.
\end{remark}

It easily follows that, if $A$ is a pro-Lie Polish abelian groups, then
setting $A_{\mathbb{S}^{1}}:=\mathrm{\mathrm{lim}}_{N\in \mathcal{N}\left(
G\right) }\left( A/N\right) _{\mathbb{S}^{1}}$ etcetera, we obtain short
exact sequences $A_{\mathbb{S}^{1}}\rightarrow F_{\mathbb{Z}}A\rightarrow A_{%
\mathbb{A}}$ and $F_{\mathbb{Z}}A\rightarrow A\rightarrow A_{\mathbb{Z}}$.
Furthermore, we have that $A_{\mathbb{S}^{1}}$ is the largest closed
subgroup of $A$ of type $\mathbb{S}^{1}$ and the smallest closed subgroup of 
$F_{\mathbb{Z}}A$ such that $F_{\mathbb{Z}}A/A_{\mathbb{S}^{1}}$ is of type $%
\mathbb{A}$, while $F_{\mathbb{Z}}A$ is the smallest closed subgroup of $A$
such that $A/F_{\mathbb{Z}}A$ is of sub-type $\mathbb{Z}$.

This decomposition for pro-Lie Polish abelian groups can be reformulated in
terms of \emph{torsion pars}, showing that pro-Lie Polish abelian groups of
type $\mathbb{A}$, type $\mathbb{S}^{1}$, and subtype $\mathbb{Z}$ form
fully exact subcategories of the category $\mathbf{proLiePAb}$. This also
follows from Lemma \ref{Lemma:thick-pro-category}, together with the fact
that locally compact pro-Lie Polish abelian groups of type $\mathbb{A}$, $%
\mathbb{S}^{1}$, and type $\mathbb{Z}$, respectively, form a fully exact
subcategories of $\mathbf{LCPAb}$.

\subsection{Injective and projective pro-Lie Polish abelian groups}

In this section, we completely characterize the projective objects in $%
\mathbf{proLiePAb}$, and obtain the following.

\begin{theorem}
\label{Theorem:proLiePAb}The quasi-abelian category $\mathbf{proLiePAb}$ has
enough projectives but not enough injectives, and homological dimension $1$.
\end{theorem}

Recall that an abelian Lie group $A$ is of the form $V\oplus T\oplus D$
where $V$ is a finite-dimensional vector group, $T$ is a finite-dimensional
torus, and $D$ is discrete. (This is equivalent to the assertion that $A$ is
a locally compact Polish abelian group \emph{with no small subgroups}; see 
\cite[Theorem 2.4]{moskowitz_homological_1967}.) It easily follows that
there exists a surjective homomorphism $\mathbb{R}^{n}\oplus \mathbb{Z}%
^{\left( \omega \right) }\rightarrow A$ for some $n\in \omega $. Recall that
the projective objects in $\mathbf{LiePAb}$ are precisely those of the form $%
V\oplus F$ where $V$ is a finite-dimensional vector group and $F$ is a
countable free abelian group \cite[Theorem 3.3]{moskowitz_homological_1967}.
The injective objects in $\mathbf{LiePAb}$ are precisely those of the form $%
V\oplus T$ where $V$ is a finite-dimensional vector group and $T$ is a
finite-dimensional torus \cite[Theorem 3.2]{moskowitz_homological_1967}.

\begin{lemma}
\label{Lemma:tower-quasi-abelian}Suppose that $\mathcal{A}$ be a
quasi-abelian subcategory of $\mathbf{PAb}$ closed under countable products,
and let $\mathcal{D}$ be a class of projective objects in $\mathcal{A}$
closed under direct sums. Suppose that $\left( G_{n}\right) $ is an inverse
sequence of objects of $\mathcal{A}$ with surjective continuous
homomorphisms $\pi _{n+1}:G_{n+1}\rightarrow G_{n}$ as bonding maps. We also
set $G_{-1}=0$ and $\pi _{0}=0$. Suppose that for every $n\in \omega $ there
exists a surjective continuous homomorphism $P\rightarrow \mathrm{\mathrm{Ker%
}}\left( \pi _{n}:G_{n}\rightarrow G_{n-1}\right) $ for some $P\in \mathcal{D%
}$. Then there exists a surjective continuous homomorphism $\mathrm{lim}%
_{n}D_{n}\rightarrow \mathrm{lim}_{n}G_{n}$ where $D=\left( D_{n}\right) $
is an inverse sequence with $D_{n}\in \mathcal{D}$ and surjective continuous
homomorphisms $p_{n+1}:D_{n+1}\rightarrow D_{n}$ as bonding maps.
\end{lemma}

\begin{proof}
We define by recursion objects $D_{n}$ of $\mathcal{D}$, continuous
homomorphisms $p_{n+1}:D_{n+1}\rightarrow D_{n}$, and $\varphi
_{n}:D_{n}\rightarrow G_{n}$ such that $\pi _{n+1}\varphi _{n+1}=\varphi
_{n}p_{n+1}$. We have that $\varphi _{0}$ exists by hypothesis. Suppose that 
$\varphi _{i}$ and $p_{i}$ have been defined for $i\leq n$. Then we consider
a pushout diagram%
\begin{equation*}
\begin{array}{ccc}
Y_{n+1} & \rightarrow & G_{n+1} \\ 
\downarrow &  & \downarrow \\ 
D_{n} & \rightarrow & G_{n}%
\end{array}%
\end{equation*}%
Since $D_{n}\rightarrow G_{n}$ and $G_{n+1}\rightarrow G_{n}$ are continuous
surjective homomorphisms, the same holds for $Y_{n+1}\rightarrow D_{n}$ and $%
Y_{n+1}\rightarrow G_{n}$. Since $D_{n}$ is projective, we have that $%
Y_{n+1}\cong D_{n}\oplus \mathrm{\mathrm{Ker}}\left( \pi _{n+1}\right) $. By
the inductive hypothesis, and since $\mathcal{D}$ is closed under direct
sums, there exists a surjective continuous homomorphism $D_{n+1}\rightarrow
Y_{n+1}$. This concludes the recursive construction.

One then has that $\mathrm{lim}_{n}\varphi _{n}:\mathrm{lim}%
_{n}D_{n}\rightarrow \mathrm{lim}_{n}G_{n}$ is a continuous surjective
homomorphism.
\end{proof}

\begin{lemma}
\label{Lemma:enough-projective-proLie}Let $G$ be a pro-Lie Polish abelian
group. Then there exists a surjective continuous homomorphism $\mathbb{R}%
^{\omega }\oplus \left( \mathbb{Z}^{\left( \omega \right) }\right) ^{\omega
}\rightarrow G$.
\end{lemma}

\begin{proof}
By Lemma \ref{Lemma:tower-quasi-abelian} it suffices to prove that the
conclusion holds when $G$ is an abelian Lie group, in which case the
conclusion follows from the above remarks.
\end{proof}

\begin{lemma}
\label{Lemma:free-subgroups}If $G$ is a closed subgroup of $\left( \mathbb{Z}%
^{\left( \omega \right) }\right) ^{\omega }$, then $G\cong \mathbb{Z}%
^{\alpha }\oplus \left( \mathbb{Z}^{\left( \omega \right) }\right) ^{\beta }$
for some $\alpha ,\beta \in \omega +1$.
\end{lemma}

\begin{proof}
Define $V_{n}=\left\{ x\in \left( \mathbb{Z}^{\left( \omega \right) }\right)
^{\omega }:\forall i<n,x_{i}=0\right\} $. Then $\left( V_{n}\right) $ is a
basis of zero neighborhoods for $\left( \mathbb{Z}^{\left( \omega \right)
}\right) ^{\omega }$, and $\left( G\cap V_{n}\right) $ is a basis of zero
neighborhoods of $G$. For $n\in \omega $, $G/(G\cap V_{n})$ is a subgroup of 
$\left( \mathbb{Z}^{\left( \omega \right) }\right) ^{n}$, and hence free
abelian. Since countable free abelian groups are projective, we have that $%
G\cong \mathrm{\mathrm{lim}}_{n}G/(G\cap V_{n})$ is isomorphic to a product
of free abelian groups. The conclusion easily follows.
\end{proof}

\begin{lemma}
\label{Lemma:projective-proLiePAb}If $G$ is a projective object in $\mathbf{%
proLiePAb}$, then $G\cong \mathbb{Z}^{\alpha }\oplus \left( \mathbb{Z}%
^{\left( \omega \right) }\right) ^{\beta }\oplus \mathbb{R}^{\gamma }$ for
some $\alpha ,\beta ,\gamma \in \omega +1$.
\end{lemma}

\begin{proof}
By Lemma \ref{Lemma:enough-projective-proLie}, we have a surjective
epimorphism $\left( \mathbb{Z}^{\left( \omega \right) }\right) ^{\omega
}\oplus \mathbb{R}^{\omega }\rightarrow G$. If $G$ is projective, then we
have that $G$ is a topological direct summand of $\left( \mathbb{Z}^{\left(
\omega \right) }\right) ^{\omega }\oplus \mathbb{R}^{\omega }$. This implies
that $G_{\mathbb{S}^{1}}=0$, $G_{\mathrm{t}}=0$, and hence $G\cong V\oplus
G_{0}$ where $V$ is a vector group and $G_{0}$ is subtype $\mathbb{Z}$.
Since $V$ is injective in $\mathbf{proLiePAb}$ and $\mathrm{Hom}\left(
V,\left( \mathbb{Z}^{\left( \omega \right) }\right) ^{\omega }\right) =0$,
we can assume without loss of generality that $V=0$ and $G=G_{0}$ is subtype 
$\mathbb{Z}$. In this case, we have that $\mathrm{Hom}\left( \mathbb{R}%
^{\omega },G\right) =0$ and $G\subseteq \left( \mathbb{Z}^{\left( \omega
\right) }\right) ^{\omega }$. The conclusion thus follows from Lemma \ref%
{Lemma:free-subgroups}.
\end{proof}

\begin{lemma}
\label{Lemma:vanishing-Ext-product0}Suppose that $\left( C_{i}\right) _{i\in
\omega }$ is a sequence of countable abelian groups, $C=\prod_{i\in \omega
}C_{i}$, and $A$ is a countable abelian group. If $\mathrm{Ext}_{_{\mathrm{%
Yon}}}\left( C_{i},A\right) =0$ for every $i\in \omega $, then $\mathrm{Ext}%
_{\mathrm{Yon}}\left( C,A\right) =0$.
\end{lemma}

\begin{proof}
By Lemma \ref{Lemma:Ext-Yon-nA} we can identify $\mathrm{Ext}_{\mathrm{Yon}%
}\left( C,A\right) $ with the subgroup of $\mathrm{Ext}_{\mathrm{c}}\left(
C,A\right) $ corresponding to continuous cocycles. For $n\in \omega $, define%
\begin{equation*}
C^{>n}=\left\{ x\in C:\forall i\leq n,x_{i}=0\right\}
\end{equation*}%
and%
\begin{equation*}
C^{\leq n}=\left\{ x\in C:\forall i>n,x_{i}=0\right\} \text{.}
\end{equation*}%
Suppose that $c:C\times C\rightarrow A$ is a continuous cocycle. Since $c$
is continuous, there exists $n_{0}\in \omega $ such that $c\left( x,y\right)
=0$ whenever $x,y\in C^{>n}$. Since 
\begin{equation*}
\mathrm{Ext}_{\mathrm{Yon}}\left( C^{\leq n},A\right) \cong \mathrm{Ext}%
\left( C_{0},A\right) \oplus \cdots \oplus \mathrm{Ext}\left( C_{n},A\right)
\cong 0\text{,}
\end{equation*}%
it follows that the inclusion $C^{>n}\rightarrow C$ induces an isomorphism%
\begin{equation*}
\mathrm{Ext}_{\mathrm{Yon}}\left( C,A\right) \rightarrow \mathrm{Ext}_{%
\mathrm{Yon}}\left( C^{>n},A\right) \text{.}
\end{equation*}%
Thus, we have that $c$ is a coboundary.
\end{proof}

\begin{lemma}
\label{Lemma:vanishing-Ext-product1}Suppose that $C$ and $A$ are
non-Archimedean Polish abelian groups. Suppose that $\left( V_{k}\right) $
is a basis of zero neighborhoods of $A$ with $V_{0}=A$. If $\mathrm{Ext}_{_{%
\mathrm{Yon}}}\left( C,V_{k}/V_{k+1}\right) =0$ for every $k\in \omega $,
then $\mathrm{Ext}_{\mathrm{Yon}}\left( C,A\right) =0$.
\end{lemma}

\begin{proof}
Suppose that $c:C\times C\rightarrow A$ is a continuous cocycle. By
hypothesis, there exists a continuous function $\varphi _{0}:C\rightarrow A$
such that $c_{0}:=\delta \varphi _{0}+c$ defines an element of $\mathrm{Ext}%
_{\mathrm{c}}\left( C,V_{1}\right) $. Proceeding in this fashion, we can
define a sequence of continuous $2$-cocycles $c_{n}:C\times C\rightarrow
V_{n}$ and continuous functions $\varphi _{n}:C^{\omega }\rightarrow V_{n}$
with $c_{0}=c$ such that $c_{n+1}=\delta \varphi _{n}+c_{n}$ for every $n\in
\omega $.

Setting $\varphi :=\sum_{n\in \omega }\varphi _{n}$ we obtain a continuous
function $C\rightarrow A$ such that $\delta \varphi +c=0$, concluding the
proof.
\end{proof}

\begin{lemma}
\label{Lemma:vanishing-Ext-product}Suppose that $\left( C_{i}\right) _{i\in
\omega }$ is a sequence of countable abelian groups, and $A$ is a
non-Archimedean abelian Polish group. Suppose that $\left( V_{k}\right) $ is
a basis of zero neighborhoods of $A$ with $V_{0}=A$. Set also $%
C:=\prod_{i\in \omega }C_{i}$. If $\mathrm{Ext}_{_{\mathrm{Yon}}}\left(
C_{i},V_{k}/V_{k+1}\right) =0$ for every $k,i\in \omega $, then $\mathrm{Ext}%
_{\mathrm{Yon}}\left( C,A\right) =0$.
\end{lemma}

\begin{proof}
This is obtained combining Lemma \ref{Lemma:vanishing-Ext-product0} and
Lemma \ref{Lemma:vanishing-Ext-product1}.
\end{proof}

\begin{lemma}
\label{Lemma:vanishing-Ext-product-2}Suppose that $\left( C_{i}\right)
_{i\in \omega }$ is a sequence of countable free abelian groups. Then,
setting $C:=\prod_{k\in \omega }C_{k}$, we have that $C$ is projective for
pro-Lie Polish abelian groups.
\end{lemma}

\begin{proof}
By Lemma \ref{Lemma:Ext-Yon-nA}, if $A$ is a pro-Lie Polish abelian group,
then we can identify $\mathrm{Ext}_{\mathrm{Yon}}\left( C,A\right) $ with
the subgroup of $\mathrm{Ext}_{\mathrm{c}}\left( C,A\right) $ corresponding
to continuous cocycles. By Lemma \ref{Lemma:injective-pro-Lie}, since
pro-Lie Polish abelian groups form a thick subcategory of $\mathbf{PAb}$, we
have that $\mathrm{Ext}_{\mathrm{Yon}}\left( C,T\right) =0$ for every torus $%
T$. Furthermore, by Lemma \ref{Lemma:vanishing-Ext-product} we have that $%
\mathrm{Ext}_{\mathrm{Yon}}\left( C,A\right) =0$ for every non-Archimedean
Polish abelian group $A$. If $B$ is an arbitrary pro-Lie abelian Polish
group, then by Lemma \ref{Lemma:structure-pro-Lie} we have a short exact
sequence $A\rightarrow B\rightarrow T$ where $A$ is non-Archimedean and $T$
is a torus. By Lemma \ref{Lemma:exact-Ext-Yon}(2), this induces an exact
sequence $0=\mathrm{Ext}_{\mathrm{Yon}}\left( C,A\right) \rightarrow \mathrm{%
Ext}_{\mathrm{Yon}}\left( C,B\right) \rightarrow \mathrm{Ext}_{\mathrm{Yon}%
}\left( C,T\right) =0$, showing that \textrm{Ext}$_{\mathrm{Yon}}\left(
C,B\right) =0$.
\end{proof}

\begin{theorem}
\label{Theorem:characterize-projectives}Let $G$ be an abelian pro-Lie Polish
group. The following assertions are equivalent:

\begin{enumerate}
\item $G$ is projective in $\mathbf{proLiePAb}$;

\item $G$ is isomorphic to $\mathbb{Z}^{\alpha }\oplus \left( \mathbb{Z}%
^{\left( \omega \right) }\right) ^{\beta }\oplus \mathbb{R}^{\gamma }$ for
some $\alpha ,\beta ,\gamma \in \omega +1$;

\item $\left\{ N\in \mathcal{N}\left( G\right) :G/N\text{ is projective in }%
\mathbf{LieAb}\right\} $ is cofinal in $\mathcal{N}\left( G\right) $.
\end{enumerate}
\end{theorem}

\begin{proof}
(2)$\Rightarrow $(1) We have that $\mathbb{R}^{\omega }$ is projective by
Lemma \ref{Lemma:R-projective}. We have that $\left( \mathbb{Z}^{\left(
\omega \right) }\right) ^{\omega }$ is projective by Lemma \ref%
{Lemma:vanishing-Ext-product-2}. Any group of the form $\mathbb{Z}^{\alpha
}\oplus \left( \mathbb{Z}^{\left( \omega \right) }\right) ^{\beta }\oplus 
\mathbb{R}^{\gamma }$ for $\alpha ,\beta ,\gamma \in \omega +1$ is a direct
summand of $\mathbb{R}^{\omega }\oplus \left( \mathbb{Z}^{\left( \omega
\right) }\right) ^{\omega }$, and hence projective.

(1)$\Rightarrow $(2) Any projective object is isomorphic to one of this form
by Lemma \ref{Lemma:projective-proLiePAb}.

(3)$\Rightarrow $(1) Let $\left( N_{k}\right) $ be a cofinal sequence in $%
\mathcal{N}\left( G\right) $ such that $G^{\left( k\right) }:=G/N_{k}$ is
projective in $\mathbf{LieAb}$ for every $k\in \omega $. Then by Lemma \ref%
{Lemma:connected-component} we have that $c\left( G\right) =G_{\mathbb{R}}$
and $G\left/ c\left( G\right) \right. \cong \mathrm{\mathrm{lim}}_{k}G_{%
\mathbb{Z}}^{\left( k\right) }$ where $(G_{\mathbb{Z}}^{\left( k\right)
})_{k\in \omega }$ is an epimorphic tower of free abelian groups. Whence, $%
G_{\mathbb{Z}}$ is isomorphic to a product of free abelian groups and hence
projective in $\mathbf{proLiePAb}$. Hence, $G\cong G_{\mathbb{R}}\oplus G_{%
\mathbb{Z}}$ is projective in $\mathbf{proLiePAb}$.

(2)$\Rightarrow $(3) This follows from the fact that the projective objects
in $\mathbf{LieAb}$ are of the form $\mathbb{R}^{n}\oplus \mathbb{Z}^{\left(
\alpha \right) }$ for $n<\omega $ and $\alpha \leq \omega $.
\end{proof}

\begin{theorem}
\label{Theorem:hd-proLiePAb}The category $\mathbf{proLiePAb}$ has enough
projectives and homological dimension $1$.
\end{theorem}

\begin{proof}
We have that $\mathbf{proLiePAb}$ has enough projectives by Theorem \ref%
{Theorem:characterize-projectives} and Lemma \ref%
{Lemma:enough-projective-proLie}. We now prove that closed subgroups of
projectives are projective. Suppose that $G$ is a closed subgroup of $\left(
F\oplus \mathbb{R}\right) ^{\omega }$ where $F$ is a countable free abelian
group of infinite rank. Then we have that $G_{\mathbb{S}^{1}}=0$ and $G_{%
\mathrm{t}}=0$. Hence, we have that $G\cong G_{\mathbb{R}}\oplus G_{\mathbb{Z%
}}$. Since $G_{\mathbb{R}}$ is injective, we can assume without loss of
generality that $G=G_{\mathbb{Z}}$.

Set 
\begin{equation*}
N_{k}:=\left\{ x\in \left( F\oplus \mathbb{R}\right) ^{\omega }:\forall i<k%
\text{, }x_{i}=0\right\} \text{.}
\end{equation*}%
Then $\left( N_{k}\right) $ is a vanishing sequence in $\mathcal{N}\left(
\left( F\oplus \mathbb{R}\right) ^{\omega }\right) $, and $\left(
M_{k}\right) _{k\in \omega }$ is a vanishing sequence in $\mathcal{N}\left(
G\right) $, where $M_{k}:=N_{k}\cap G$. We have that $G/M_{k}$ is isomorphic
to a closed subgroup $L$ of 
\begin{equation*}
\left( F\oplus \mathbb{R}\right) ^{\omega }/N_{k}\cong \left( F\oplus 
\mathbb{R}\right) ^{k}\cong F^{k}\oplus \mathbb{R}^{k}\text{.}
\end{equation*}%
We consider the cohomological derived functor of $\mathrm{H}^{0}\circ 
\mathrm{Hom}^{\bullet }$ for locally compact Polish abelian groups
introduced in \cite{bergfalk_applications_2023}; see also \cite[Section 4]%
{hoffmann_homological_2007}. For every locally compact Polish abelian group $%
A$, we have an exact sequence%
\begin{equation*}
0=\mathrm{Ext}\left( F^{k}\oplus \mathbb{R}^{k},A\right) \rightarrow \mathrm{%
Ext}\left( L,A\right) \rightarrow \mathrm{Ext}^{2}\left( \left( F^{k}\oplus 
\mathbb{R}^{k}\right) /L,A\right) =0\text{.}
\end{equation*}%
Therefore, we have that $L$ is a projective locally compact Polish abelian
group. Since $L$ is also countable, we have that is a free abelian group.

Thus, we have that $G\cong \mathrm{\mathrm{lim}}_{k}G/M_{k}$ where, for
every $k\in \omega $, $G/M_{k}$ is a countable free abelian group. This
shows that $G$ is isomorphic to a product of countable free abelian groups,
whence projective in $\mathbf{proLiePAb}$.
\end{proof}

\begin{corollary}
The functor $\mathrm{Hom}^{\bullet }:\mathrm{K}^{b}\left( \mathbf{proLiePAb}%
\right) ^{\mathrm{op}}\times \mathrm{K}^{b}\left( \mathbf{proLiePAb}\right)
\rightarrow \mathrm{K}^{b}\left( \mathbf{Ab}\right) $ admits a total right
derived functor $\mathrm{RHom}:\mathrm{D}^{b}\left( \mathbf{proLiePAb}%
\right) ^{\mathrm{op}}\times \mathrm{D}^{b}\left( \mathbf{proLiePAb}\right)
\rightarrow \mathrm{D}^{b}\left( \mathbf{Ab}\right) $.
\end{corollary}

\begin{corollary}
A pro-Lie Polish abelian group $G$ is projective if and only if $\mathrm{Ext}%
\left( G,\mathbb{Z}\right) =0$ and $\mathrm{Ext}\left( G,\mathbb{Z}^{\left(
\omega \right) }\right) =0$.
\end{corollary}

\begin{proof}
We prove sufficiency, as necessity is obvious. By Theorem \ref%
{Theorem:hd-proLiePAb}, we have a short exact sequence $P^{0}\rightarrow
P^{1}\rightarrow G$ where $P^{0}$ and $P^{1}$ are projectives. By Theorem %
\ref{Theorem:characterize-projectives}, we have that $P^{0}\cong \mathbb{Z}%
^{\alpha }\oplus \left( \mathbb{Z}^{\left( \omega \right) }\right) ^{\beta
}\oplus \mathbb{R}^{\gamma }$ for some $\alpha ,\beta ,\gamma \leq \omega $.
We have that $\mathrm{Ext}\left( G,\mathbb{R}^{\gamma }\right) =\mathrm{Ext}%
\left( G,\mathbb{R}\right) ^{\gamma }=0$ since $\mathbb{R}$ is injective. By
assumption, we have $\mathrm{Ext}\left( G,\mathbb{Z}^{\alpha }\right) \cong 
\mathrm{Ext}\left( G,\mathbb{Z}\right) ^{\alpha }=0$ and $\mathrm{Ext}%
(G,\left( \mathbb{Z}^{\left( \omega \right) }\right) ^{\beta })\cong \mathrm{%
Ext}\left( G,\mathbb{Z}^{\left( \omega \right) }\right) ^{\beta }=0$ by
hypothesis. This implies that $\mathrm{Ext}\left( G,P^{0}\right) =0$. Thus, $%
G$ is a topological direct summand of $P^{1}$, and hence projective.
\end{proof}

\begin{lemma}
\label{Lemma:countable-divisible-injective}Every countable divisible group
is injective in the category of non-Archimedean Polish abelian groups.
\end{lemma}

\begin{proof}
Let $D$ be a countable divisible group, and $A$ be a non-Archimedean Polish
abelian group. Then there exists a sequence $\left( C_{i}\right) $ of
countable abelian groups such that, setting 
\begin{equation*}
C:=\prod_{n}C_{n}\text{,}
\end{equation*}%
there exists an injective continuous homomorphism $A\rightarrow C$. This
induces a surjective homomorphism%
\begin{equation*}
\mathrm{Ext}\left( C,D\right) \rightarrow \mathrm{Ext}\left( A,D\right) 
\text{.}
\end{equation*}%
As $D$ is injective in the category of countable abelian groups, for every $%
i\in \omega $ we have that $\mathrm{Ext}\left( C_{i},D\right) =0$. By Lemma %
\ref{Lemma:vanishing-Ext-product0} this implies that $\mathrm{Ext}\left(
C,D\right) =0$ and hence $\mathrm{Ext}\left( A,D\right) =0$.
\end{proof}

For pro-Lie Polish abelian groups $A$ and $B$, we set $\mathrm{Ext}%
^{n}\left( A,B\right) =\mathrm{H}^{n}\left( \mathrm{RHom}\left( A,B\right)
\right) $. We thus have that $\mathrm{Ext}^{n}\left( A,B\right) =0$ for $%
n\geq 2$, while $\mathrm{Ext}^{1}\left( A,B\right) $ (which we simply denote
by $\mathrm{Ext}\left( A,B\right) $) is isomorphic to the group $\mathrm{Ext}%
_{\mathrm{Yon}}^{1}\left( A,B\right) $ (which we have been denoting by $%
\mathrm{Ext}_{\mathrm{Yon}}\left( A,B\right) $) of isomorphism classes of ($%
1 $-fold) extensions of $A$ by $B$.

Notice that, if $C$ is a countable discrete group and $A$ is a pro-Lie
Polish abelian group, then any extension $A\rightarrow X\rightarrow C$ in $%
\mathbf{Ab}$ can be turned uniquely into an extension in $\mathbf{proLiePAb}$%
. Whence, we have that $\mathrm{Ext}\left( C,A\right) \cong \mathrm{Ext}%
\left( C,A_{\mathrm{disc}}\right) $, where $A_{\mathrm{disc}}$ is the group $%
A$ endowed with the discrete topology.

We now characterize the injective objects in the category of pro-Lie Polish
abelian groups.

\begin{lemma}
\label{Lemma:embed-pro-Lie}Every abelian Polish pro-Lie group is isomorphic
to a closed subgroup of $\mathbb{R}^{\omega }\oplus \mathbb{T}^{\omega
}\oplus U^{\omega }$ where $U=\mathbb{Q}^{\left( \omega \right) }\oplus
\bigoplus_{p}\mathbb{Z}\left( p^{\infty }\right) ^{\left( \omega \right) }$.
\end{lemma}

\begin{proof}
Let $A$ be a pro-Lie abelian Polish group. Then $A$ is isomorphic to a
closed subgroup of $\prod_{k\in \omega }A_{k}$ where, for every $k\in \omega 
$, $A_{k}$ is an abelian Polish Lie group. Thus, $A_{k}$ is isomorphic to $%
\mathbb{R}^{n_{k}}\oplus \mathbb{T}^{m_{k}}\oplus D_{k}$ for some $%
n_{k},m_{k}\in \omega $ and countable discrete group $D_{k}$. Thus, $A_{k}$
is isomorphic to a closed subgroup of $\mathbb{R}^{\omega }\oplus \mathbb{T}%
^{\omega }\oplus U$. The conclusion about $A$ follows.
\end{proof}

\begin{theorem}
\label{Theorem:injective-pro-Lie}Let $G$ be an abelian Polish pro-Lie group.
The following assertions are equivalent:

\begin{enumerate}
\item $G$ is an \emph{injective object }in $\mathbf{proLiePAb}$;

\item $G$ is isomorphic to $\mathbb{R}^{\alpha }\oplus \mathbb{T}^{\beta }$
for some $\alpha ,\beta \in \omega +1$;

\item $G$ is path-connected;

\item $\mathrm{Ext}\left( \mathbb{T},G\right) =0$.
\end{enumerate}
\end{theorem}

\begin{proof}
(2)$\Rightarrow $(1) This follows from Lemma \ref{Lemma:injective-pro-Lie}
considering that $\mathbf{proLiePAb}$ is a thick subcategory of $\mathbf{PAb}
$ by Theorem \ref{Theorem:pro-Lie-thick}.

(4)$\Rightarrow $(3) Suppose that $\mathrm{Ext}\left( \mathbb{T},G\right) =0$%
. Then we have an exact sequence%
\begin{equation*}
\mathrm{Hom}\left( \mathbb{R},G\right) \rightarrow \mathrm{Hom}\left( 
\mathbb{Z},G\right) \rightarrow \mathrm{Ext}\left( \mathbb{T},G\right) =0
\end{equation*}%
This shows that $G$ is path-connected.

(3)$\Rightarrow $(2) Since $G$ is connected, we have that $G\cong G_{\mathbb{%
R}}\oplus G_{\mathbb{S}^{1}}$. Furthermore, $G_{\mathbb{S}^{1}}$ is
path-connected, and hence a torus by \cite[Theorem 8.27]%
{armacost_structure_1981}.
\end{proof}

\begin{corollary}
\label{Corollary:quotient-injective-pro-Lie}The class of injective objects
in $\mathbf{proLiePAb}$ is closed under quotients.
\end{corollary}

\begin{corollary}
If $D$ is a nonzero countable discrete divisible group, then there exists no
injective object $I$ in $\mathbf{proLiePAb}$ such that $D$ is isomorphic to
a closed subgroup of $I$. Hence, $\mathbf{proLiePAb}$ does not have enough
injectives.
\end{corollary}

\begin{proof}
Suppose that $D$ is a countable discrete divisible group that is a closed
subgroup of $\mathbb{R}^{\omega }\oplus \mathbb{T}^{\omega }$. Then we have
that $\mathbb{T}^{\omega }\cap D$ is a divisible subgroup of $D$. Hence, we
have that $D\cong H\oplus \left( \mathbb{T}^{\omega }\cap D\right) $ for
some (necessarily divisible) subgroup $H$ of $D$. The composition $%
H\rightarrow \mathbb{R}^{\omega }\oplus \mathbb{T}^{\omega }\rightarrow 
\mathbb{R}^{\omega }$ is injective. Since $\mathbb{T}^{\omega }$ is compact
and $H$ is closed in $\mathbb{T}^{\omega }\oplus \mathbb{R}^{\omega }$, we
have that $\mathbb{T}^{\omega }+H$ is also closed in $\mathbb{T}^{\omega
}\oplus \mathbb{R}^{\omega }$. This implies that the image of the function $%
H\rightarrow \mathbb{R}^{\omega }\oplus \mathbb{T}^{\omega }\rightarrow 
\mathbb{R}^{\omega }$ is a closed subgroup of $\mathbb{R}^{\omega }$. Every
closed divisible subgroup of $\mathbb{R}^{\omega }$ is a closed $\mathbb{R}$%
-subspace. Since $H$ is countable, this implies that $H=0$.

Thus, we have that $D=D\cap \mathbb{T}^{\omega }\subseteq \mathbb{T}^{\omega
}$. Thus, we have that the inclusion $D\rightarrow \mathbb{T}^{\omega }$
induces a surjective homomorphism $\mathbb{Z}^{\left( \omega \right)
}\rightarrow D^{\vee }$. Since $D^{\vee }$ is compact and connected, it is
uncountable whenever it is nonzero. Hence, $D^{\vee }=0$ and $D=0$.
\end{proof}

\section{Thick categories of pro-Lie Polish abelian groups\label%
{Section:thick}}

\subsection{Pro-$p$ Polish abelian groups\label{Subsection:pro-p}}

Let us say that a \emph{pro}-$p$ Polish group is a pro-Lie Polish abelian
topological $p$-group. We let $\mathbf{PAb}\left( p\right) $ be the category
of pro-$p$ Polish abelian groups, which is easily seen to be a thick
subcategory of $\mathbf{proLiePAb}$. A locally compact Polish abelian group $%
G$ is pro-$p$ if and only if it has a basis of zero neighborhoods consisting
of subgroups $U$ such that $G/U$ is a $p$-group. A pro-Lie Polish abelian
group is a pro-$p$ group if and only if $G/N$ is pro-$p$ for every $N\in 
\mathcal{N}\left( G\right) $.

For an abelian group $G$, we define $pG=\left\{ px:x\in G\right\} $. An
abelian group $G$ is $p$-divisible if $pG=G$. We also define recursively for
every ordinal $\alpha $:%
\begin{equation*}
p^{0}G=G
\end{equation*}%
and, for $\alpha >0$,%
\begin{equation*}
p^{\alpha }G=\bigcap_{\beta <\alpha }p\left( p^{\beta }G\right) \text{.}
\end{equation*}%
If $\sigma $ is the least ordinal such that $p^{\sigma }G=p^{\sigma +1}G$,
then $p^{\sigma }G$ is the largest $p$-divisible subgroup of $G$. Recall
that a group $G$ is $p$-local if it is $q$-divisible for every prime number $%
q$ other than $p$. When $G$ is $p$-local, $\sigma $ is the Ulm rank of $G$
and $p^{\sigma }G$ is the largest divisible subgroup $d\left( G\right) $ of $%
G$. When $G$ is a Polish group, the least ordinal $\sigma $ such that $%
p^{\sigma }G=p^{\sigma +1}G$ is countable. We define, for an ordinal $\alpha 
$,%
\begin{equation*}
L_{\alpha }^{p}\left( G\right) =\mathrm{\mathrm{lim}}_{\beta <\alpha
}G/p^{\beta }G
\end{equation*}%
and $E_{\alpha }^{p}\left( G\right) $ to be the quotient of $L_{\alpha
}^{p}\left( G\right) $ by the image of $G$ under the canonical homomorphism $%
G\rightarrow L_{\alpha }^{p}\left( G\right) $.

For an ordinal $\alpha $, as a particular instance of \cite[Theorem 1.5]%
{nunke_homology_1967} in the case of the cotorsion functor with enough
projectives $G\mapsto p^{\alpha }G$, we have the following:

\begin{lemma}
\label{Lemma:nunke}Let $A,G$ be groups. The short exact sequences%
\begin{equation*}
p^{\alpha }A\rightarrow A\rightarrow A/p^{\alpha }A
\end{equation*}%
and%
\begin{equation*}
p^{\alpha }G\rightarrow G\rightarrow G/p^{\alpha }G
\end{equation*}%
induce homomorphisms%
\begin{equation*}
\mathrm{Ext}\left( A,p^{\alpha }G\right) \rightarrow \mathrm{Ext}\left(
p^{\alpha }A,p^{\alpha }G\right)
\end{equation*}%
and%
\begin{equation*}
\mathrm{Hom}\left( A,G/p^{\alpha }G\right) \rightarrow \mathrm{Ext}\left(
A,p^{\alpha }G\right) \text{.}
\end{equation*}%
These induce an isomorphism%
\begin{equation*}
\eta :\frac{\mathrm{Ext}\left( A,p^{\alpha }G\right) }{\mathrm{Ran}\left( 
\mathrm{Hom}\left( A,G/p^{\alpha }G\right) \rightarrow \mathrm{Ext}\left(
A,p^{\alpha }G\right) \right) }\rightarrow \mathrm{Ext}\left( p^{\alpha
}A,p^{\alpha }G\right)
\end{equation*}%
Furthermore, the exact sequence $p^{\alpha }G\rightarrow G\rightarrow
G/p^{\alpha }G$ induces an exact sequence%
\begin{equation*}
\mathrm{Hom}\left( A,G/p^{\alpha }G\right) \rightarrow \mathrm{Ext}\left(
A,p^{\alpha }G\right) \rightarrow p^{\alpha }\mathrm{Ext}\left( A,G\right)
\rightarrow p^{\alpha }\mathrm{Ext}\left( A,G/p^{\alpha }G\right)
\rightarrow 0
\end{equation*}%
which induces an exact sequence%
\begin{equation*}
0\rightarrow \frac{\mathrm{Ext}\left( A,p^{\alpha }G\right) }{\mathrm{Ran}%
\left( \mathrm{Hom}\left( A,G/p^{\alpha }G\right) \rightarrow \mathrm{Ext}%
\left( A,p^{\alpha }G\right) \right) }\overset{\rho }{\rightarrow }p^{\alpha
}\mathrm{Ext}\left( A,G\right) \rightarrow p^{\alpha }\mathrm{Ext}\left(
A,G/p^{\alpha }G\right) \rightarrow 0\text{.}
\end{equation*}%
Consider the homomorphism%
\begin{equation*}
r:=\rho \circ \eta ^{-1}:\mathrm{Ext}\left( p^{\alpha }A,p^{\alpha }G\right)
\rightarrow p^{\alpha }\mathrm{Ext}\left( A,G\right) \text{.}
\end{equation*}%
We have an exact sequence%
\begin{equation*}
0\rightarrow \mathrm{Ext}\left( p^{\alpha }A,p^{\alpha }G\right) \overset{r}{%
\rightarrow }p^{\alpha }\mathrm{Ext}\left( A,G\right) \rightarrow p^{\alpha }%
\mathrm{Ext}\left( A,G/p^{\alpha }G\right) \rightarrow 0
\end{equation*}%
and $r$ restricts to an isomorphism%
\begin{equation*}
\gamma :\mathrm{PExt}\left( p^{\alpha }A,p^{\alpha }G\right) \rightarrow
u_{1}\left( p^{\alpha }\mathrm{Ext}\left( A,G\right) \right)
\end{equation*}%
where $u_{1}$ is the first Ulm subgroup and $\mathrm{PExt}\left( p^{\alpha
}A,p^{\alpha }G\right) =u_{1}\left( \mathrm{Ext}\left( p^{\alpha
}A,p^{\alpha }G\right) \right) $.
\end{lemma}

\begin{theorem}
\label{Theorem:injective-pro-p}Suppose that $G$ is a pro-$p$ Polish abelian
group. The following assertions are equivalent:

\begin{enumerate}
\item $G$ is injective in $\mathbf{PAb}\left( p\right) $;

\item $G$ is divisible and it has a basis of zero neighborhoods consisting
of divisible subgroups;

\item $G\cong \mathbb{Z}\left( p^{\infty }\right) ^{\alpha }\oplus (\mathbb{Z%
}\left( p^{\infty }\right) ^{\left( \omega \right) })^{\beta }$ for $\alpha
,\beta \leq \omega $;

\item $\mathrm{Ext}(\mathbb{Z}\left( p^{\infty }\right) ^{\left( \omega
\right) },G)=0$.
\end{enumerate}
\end{theorem}

\begin{proof}
(2)$\Rightarrow $(3) Let $\left( V_{k}\right) _{k\in \omega }$ be a
decreasing basis of zero neighborhoods of $G$ consisting of divisible
subgroups such that $V_{0}=G$. For $n\in \omega $ define $%
H_{n}:=V_{n}/V_{n+1}$. Since $V_{n+1}$ is divisible, the short exact
sequence $V_{n+1}\rightarrow V_{n}\rightarrow V_{n}/V_{n+1}$ splits. Let $%
\pi _{n}:V_{n}\rightarrow V_{n}/V_{n+1}$ be the quotient map and $\rho
_{n}:V_{n}/V_{n+1}\rightarrow V_{n}$ be a group homomorphism that is a right
inverse for $\pi _{n}$. Given $g\in G$ one defines $g_{n}\in V_{n}$ for $%
n\in \omega $ recursively by setting $g_{0}=g$ and $g_{n+1}=g_{n}-\rho
_{n}\pi _{n}\left( g_{n}\right) $. The continuous group homomorphism $%
G\rightarrow \prod_{n\in \omega }H_{n}$, $g\mapsto \left(
g_{n}+V_{n+1}\right) _{n\in \omega }$ is a group isomorphism. The conclusion
now follows from the structure theorem for countable divisible $p$-groups.

(3)$\Rightarrow $(1) Suppose that, $G=\prod_{n\in \omega }D_{n}$ where, for
every $n\in \omega $, $D_{n}$ is a countable divisible group. If $H$ is a
pro-$p$ Polish abelian group, we have that%
\begin{equation*}
\mathrm{Ext}\left( H,G\right) \cong \prod_{n\in \omega }\mathrm{Ext}\left(
H,D_{n}\right) \text{.}
\end{equation*}%
For $n\in \omega $, we have that $\mathrm{Ext}\left( H,D_{n}\right) =0$ by
Lemma \ref{Lemma:countable-divisible-injective}.

(1)$\Rightarrow $(2) If $G$ is not divisible, then $\mathrm{Ext}\left( 
\mathbb{Z}/p\mathbb{Z},G\right) \neq 0$. Since $\mathbb{Z}/p\mathbb{Z}$ is a
closed subgroup of $\mathbb{Z}\left( p^{\infty }\right) ^{\left( \omega
\right) }$, it follows that $\mathrm{Ext}(\mathbb{Z}\left( p^{\infty
}\right) ^{\left( \omega \right) },G)\neq 0$.

Suppose that $G$ does not have a basis of zero neighborhoods consisting of
divisible subgroups. Then it has a basis $\left( V_{k}\right) _{k\in \omega
} $ of open subgroups that are not divisible and such that $d\left(
V_{k}\right) $ is not open. For $k\in \omega $, let $\rho _{k}\geq 1$ be the
least countable ordinal such that $p^{\rho _{k}}V_{k}$ is \emph{not }open.
Set $H:=V_{0}$ and $\alpha :=\rho _{0}$.

Suppose initially that $\alpha =\beta +1$ is a successor ordinal. Then after
replacing $H$ with $p^{\beta }H$ we can assume that $\alpha =1$. Thus, we
have that $H$ is an open subgroup such that $pH$ is not open. Then we have
that, for every $k\geq 1$, $V_{k}$ is not contained in $pH$. Thus, for every 
$k\geq 1$,%
\begin{equation*}
\mathrm{Ran}\left( V_{k}/pV_{k}\rightarrow H/pH\right) \cong \mathrm{Ran}%
\left( \mathrm{Ext}\left( \mathbb{Z}/p\mathbb{Z},V_{k}\right) \rightarrow 
\mathrm{Ext}\left( \mathbb{Z}/p\mathbb{Z},H\right) \right) \neq 0\text{.}
\end{equation*}%
Thus, for every $k\geq 1$ there exists a cocycle $c_{k}:\mathbb{Z}/p\mathbb{Z%
}\rightarrow V_{k}$ that is not a coboundary even when regarded as a cocycle
with values in $H$. Define the continuous cocycle $c:\left( \mathbb{Z}/p%
\mathbb{Z}\right) ^{\mathbb{N}}\rightarrow H$ by setting%
\begin{equation*}
c\left( \left( x_{i}\right) ,\left( y_{i}\right) \right)
=\sum_{i}c_{i}\left( x_{i},y_{i}\right) \text{.}
\end{equation*}%
We claim that $c$ is not a coboundary even when regarded as a cocycle with
values in $G$. We have an exact sequence%
\begin{equation*}
\mathrm{Hom}(\left( \mathbb{Z}/p\mathbb{Z}\right) ^{\mathbb{N}%
},G/H)\rightarrow \mathrm{Ext}(\left( \mathbb{Z}/p\mathbb{Z}\right) ^{%
\mathbb{N}},H)\rightarrow \mathrm{Ext}(\left( \mathbb{Z}/p\mathbb{Z}\right)
^{\mathbb{N}},G)
\end{equation*}%
Thus, it suffices to prove that $c$ defines an element $[c]$ of $\mathrm{Ext}%
(\left( \mathbb{Z}/p\mathbb{Z}\right) ^{\mathbb{N}},H)$ that does not belong
to%
\begin{equation*}
\mathrm{Ran}\left( \mathrm{Hom}(\left( \mathbb{Z}/p\mathbb{Z}\right) ^{%
\mathbb{N}},G/H)\rightarrow \mathrm{Ext}(\left( \mathbb{Z}/p\mathbb{Z}%
\right) ^{\mathbb{N}},H)\right) \text{.}
\end{equation*}%
Considering the isomorphism%
\begin{equation*}
\mathrm{Hom}(\left( \mathbb{Z}/p\mathbb{Z}\right) ^{\mathbb{N}},G/H)\cong 
\mathrm{Hom}\left( \mathbb{Z}/p\mathbb{Z},G/H\right) ^{\left( \mathbb{N}%
\right) }
\end{equation*}%
due to the fact that $G/H$ is countable, it suffices to prove that, for
every $k\in \omega $, the image of $[c]$ in $\mathrm{Ext}(\mathbb{Z}/p%
\mathbb{Z},H)$ induced by the inclusion $\mathbb{Z}/p\mathbb{Z}\rightarrow
\left( \mathbb{Z}/p\mathbb{Z}\right) ^{\mathbb{N}}$ in the $k$-th
coordinate, is nonzero. As this image is the element $[c_{k}]$ represented
by the cocycle $c_{k}$, we have that it is nontrivial by the choice of $%
c_{k} $. This shows that $\mathrm{Ext}(\left( \mathbb{Z}/p\mathbb{Z}\right)
^{\mathbb{N}},G)\neq 0$.

Suppose now that $\alpha $ is a limit ordinal. For $k\geq 1$ let $\alpha
_{k} $ be the least countable ordinal such that $V_{k}$ is \emph{not}
contained in $p^{\alpha _{k}}H$. Notice that, since $\alpha $ is a limit
ordinal, $\alpha _{k}<\alpha $. After passing to a subsequence of $\left(
V_{k}\right) _{k\geq 1}$, we can assume that $\left( \alpha _{k}\right)
_{k\geq 1}$ is nondecreasing. We also have that $\alpha =\mathrm{\mathrm{sup}%
}_{k}\alpha _{k}$. We notice that, for $k\geq 1$, $\alpha _{k}$ must be a
successor ordinal, and we let $\beta _{k}$ be its predecessor.

Fix $k\geq 1$ and let $T_{k}$ be a countable $p$-group with $p^{\beta
_{k}}T_{k}=\mathbb{Z}/p\mathbb{Z}$. (Notice that this exists by Ulm's
classification of countable $p$-groups; see \cite[Chapter XII]%
{fuchs_infinite_1973}.) Consider the exact sequence%
\begin{equation*}
0\rightarrow \mathrm{Ext}\left( p^{\beta _{k}}T_{k},p^{\beta _{k}}H\right) 
\overset{r}{\rightarrow }p^{\beta _{k}}\mathrm{Ext}\left( T_{k},H\right)
\rightarrow p^{\beta _{k}}\mathrm{Ext}\left( T_{k},H/p^{\beta _{k}}H\right)
\rightarrow 0\text{.}
\end{equation*}%
The inclusion $V_{k}\subseteq p^{\beta _{k}}H$ induces an exact sequence%
\begin{equation*}
\mathrm{Ext}\left( p^{\beta _{k}}T_{k},V_{k}\right) \rightarrow \mathrm{Ext}%
\left( p^{\beta _{k}}T_{k},p^{\beta _{k}}H\right) \rightarrow \mathrm{Ext}%
\left( p^{\beta _{k}}T_{k},p^{\beta _{k}}H/V_{k}\right) \rightarrow 0
\end{equation*}%
that is isomorphic to%
\begin{equation*}
V_{k}/pV_{k}\rightarrow p^{\beta _{k}}H/p^{\alpha _{k}}H\rightarrow p^{\beta
_{k}}H/(p^{\alpha _{k}}H+V_{k})\rightarrow 0
\end{equation*}%
Since $V_{k}$ is not contained in $p^{\alpha _{k}}H$, we have that the
homomorphism%
\begin{equation*}
p^{\beta _{k}}H/p^{\alpha _{k}}H\rightarrow p^{\beta _{k}}H/(p^{\alpha
_{k}}H+V_{k})
\end{equation*}%
is not injective, and hence the homomorphism%
\begin{equation*}
V_{k}/pV_{k}\rightarrow p^{\beta _{k}}H/p^{\alpha _{k}}H
\end{equation*}%
is nonzero. Thus, the homomorphism%
\begin{equation*}
\mathrm{Ext}\left( p^{\beta _{k}}T_{k},V_{k}\right) \rightarrow \mathrm{Ext}%
\left( p^{\beta _{k}}T_{k},p^{\beta _{k}}H\right)
\end{equation*}%
is nonzero. The inclusion $p^{\beta _{k}}T_{k}\rightarrow T_{k}$ induces a
surjective homomorphism%
\begin{equation*}
\mathrm{Ext}\left( T_{k},V_{k}\right) \rightarrow \mathrm{Ext}\left(
p^{\beta _{k}}T_{k},V_{k}\right) \text{.}
\end{equation*}%
In view of the above remarks, there exists a cocycle $c_{k}:T_{k}\times
T_{k}\rightarrow V_{k}$ that represents an element of $\mathrm{Ext}\left(
T_{k},V_{k}\right) $ whose image under the composition%
\begin{equation*}
\mathrm{Ext}\left( T_{k},V_{k}\right) \rightarrow \mathrm{Ext}\left(
p^{\beta _{k}}T_{k},V_{k}\right) \rightarrow \mathrm{Ext}\left( p^{\beta
_{k}}T_{k},p^{\beta _{k}}H\right) \overset{r}{\rightarrow }p^{\beta _{k}}%
\mathrm{Ext}\left( T_{k},H\right) \subseteq \mathrm{Ext}\left( T_{k},H\right)
\end{equation*}%
is nonzero. Thus, $c_{k}$ is not a coboundary when regarded as a cocycle
with values in $H$.

Proceeding as above, we can conclude that, setting $T:=\prod_{k\geq 1}T_{k}$%
, $\mathrm{Ext}\left( T,H\right) $ is nonzero.

(4)$\Leftrightarrow $(1)\ This follows from the fact that every object of $%
\mathbf{PAb}\left( p\right) $ is isomorphic to a closed subgroup of $(%
\mathbb{Z}\left( p^{\infty }\right) ^{\left( \omega \right) })^{\omega }$.
\end{proof}

\begin{remark}
\label{Remark:Q_p-not-injective}The proof of Theorem \ref%
{Theorem:injective-pro-p} shows that $\mathrm{Ext}(\left( \mathbb{Z}/p%
\mathbb{Z}\right) ^{\omega },\mathbb{Q}_{p})$ is nonzero.
\end{remark}

Let us denote by $\mathbf{LCPAb}\left( p\right) $ the category of \emph{%
locally compact }pro-$p$ Polish abelian groups.

\begin{theorem}
\label{Theorem:injective-LCPAb(p)}Suppose that $G$ is a pro-$p$ locally
compact Polish abelian group. The following assertions are equivalent:

\begin{enumerate}
\item $G$ is injective in $\mathbf{LCPAb}\left( p\right) $;

\item $G$ is injective in $\mathbf{PAb}\left( p\right) $;

\item $G$ is a countable divisible $p$-group.
\end{enumerate}
\end{theorem}

\begin{proof}
(1)$\Rightarrow $(2)\ By Theorem \ref{Theorem:injective-pro-p}, it suffices
to prove that $G$ is divisible and has a basis of zero neighborhoods
consisting of divisible subgroups. As in the proof of (1)$\Rightarrow $(2)
of Theorem \ref{Theorem:injective-pro-p}, we have that $G$ is divisible.
Suppose that $G$ does not have a basis of zero neighborhoods consisting of
divisible subgroups. Then it has a basis $\left( V_{k}\right) _{k\in \omega
} $ of \emph{compact }open subgroups that are not divisible and such that $%
d\left( V_{k}\right) $ is not open. For $k\in \omega $, let $\rho _{k}\geq 1$
be the least countable ordinal such that $p^{\rho _{k}}V_{k}$ is \emph{not }%
open. By compactness of $V_{k}$, we have that $\rho _{k}\leq \omega $.
Indeed, if $\rho _{k}>\omega $, then we have that $\left\{ p^{n}V\setminus
p^{n+1}V:n\in \omega \right\} $ is an infinite clopen partition of the
compact set $V\setminus \rho ^{\omega }V$, which is impossible. Thus,
proceeding as in the proof of (1)$\Rightarrow $(2) of Theorem \ref%
{Theorem:injective-pro-p}, one obtains that there exists a sequence $\left(
T_{k}\right) $ of \emph{finite }$p$-groups such that, setting $%
T:=\prod_{k\in \omega }T_{k}$, one has that $\mathrm{Ext}\left( T,G\right)
\neq 0$, which is a contradiction.

(2)$\Rightarrow $(3) This follows from (1)$\Rightarrow $(3) in Theorem \ref%
{Theorem:injective-pro-p}, since $G$ is by hypothesis locally compact.
\end{proof}

\begin{corollary}
Suppose that $G$ is a pro-$p$ locally compact Polish abelian group. The
following assertions are equivalent:

\begin{enumerate}
\item $G$ is projective in $\mathbf{LCPAb}\left( p\right) $;

\item $G$ is a compact torsion-free $p$-group.
\end{enumerate}
\end{corollary}

\begin{proof}
We have that Pontryagin duality establishes an equivalence between $\mathbf{%
LCPAb}\left( p\right) $ and its opposite. Thus, it maps injective objects to
projective objects and vice versa. Furthermore, by \cite[Theorem 4.15]%
{armacost_structure_1981}, a locally compact Polish abelian group $G$ is
compact and torsion-free if and only if its dual is countable and divisible.
\end{proof}

\begin{corollary}
\label{Corollary:enough-injectives-pro-p}The quasi-abelian category $\mathbf{%
PAb}\left( p\right) $ has enough injectives and homological dimension $1$.
\end{corollary}

\begin{proof}
Let $G$ be a pro-$p$ Polish abelian group. Then $G$ is isomorphic to a
closed subgroup of $\prod_{n\in \omega }G_{n}$ for some sequence $\left(
G_{n}\right) $ of countable $p$-groups. In turn, for every $n\in \omega $, $%
G_{n}$ is a subgroup of the countable divisible $p$-group $D:=\mathbb{Z}%
\left( p^{\infty }\right) ^{\left( \omega \right) }$. Thus, $G$ is
isomorphic to a closed subgroup of $D^{\omega }$, which is injective by
Theorem \ref{Theorem:injective-pro-p}. This shows that $\mathbf{PAb}\left(
p\right) $ has enough injectives.

Suppose that $A$ is an injective pro-$p$ Polish abelian group, $B$ is a
Polish abelian group, and $\pi :A\rightarrow B$ is a surjective continuous
homomorphism. Since $A$ is divisible, $B$ is also divisible. If $\left(
V_{k}\right) $ is a basis of zero neighborhoods of $A$ consisting of
divisible subgroups, then $\left( \pi \left( V_{k}\right) \right) $ is a
basis of zero neighborhoods of $B$ consisting of divisible subgroups by the
Open Mapping Theorem for Polish groups. Thus, $B$ is an injective pro-$p$
Polish abelian group. This shows that quotients of injectives are injective
in $\mathbf{PAb}\left( p\right) $, and hence $\mathbf{PAb}\left( p\right) $
has homological dimension at most $1$.
\end{proof}

\begin{problem}
Characterize the projective objects in $\mathbf{PAb}\left( p\right) $. Does $%
\mathbf{PAb}\left( p\right) $ have enough projectives?
\end{problem}

\subsection{Topological torsion groups}

Recall that a pro-Lie Polish abelian group $G$ is topological torsion if and
only if it is an inverse limit of countable torsion groups. We thus say that
a Polish abelian group is topological torsion if it is an inverse limit of
countable torsion groups. Topological torsion Polish abelian groups form a
thick subcategory $\mathbf{proTorPAb}$ of $\mathbf{proLiePAb}$. We denote by 
$\mathbf{proTorLCPAb}$ the category of locally compact topological torsion
Polish abelian groups. These are precisely the locally compact Polish
abelian groups that are totally disconnected and whose Pontryagin dual is
also totally disconnected \cite[Theorem 3.5]{armacost_structure_1981}.

Given a topological torsion Polish abelian group, we let $G_{p}$ be the
subgroup of $G$ consisting of elements $x$ such that $\mathrm{\mathrm{lim}}%
_{n\rightarrow \infty }p^{n}x=0$ \cite[Definition 2.1]%
{armacost_structure_1981}, called the $p$-\emph{component }of $G$. Then we
have that $G_{p}$ is a closed subgroup of $G$ \cite[Lemma 3.8]%
{armacost_structure_1981}. We can write $G\cong G_{p}\oplus G_{p}^{\#}$
where $G_{p}^{\#}$ has trivial $p$-component, and 
\begin{equation*}
\mathrm{Hom}\left( G_{p}^{\#},T\right) =\mathrm{Hom}\left(
T,G_{p}^{\#}\right) =0
\end{equation*}%
for any topological $p$-group $T$. It follows that, when $G$ and $H$ are
topological torsion groups, that 
\begin{equation*}
\mathrm{Hom}\left( G,H\right) \cong \mathrm{Hom}\left( G_{p},H_{p}\right)
\oplus \mathrm{Hom}\left( G_{p}^{\#},H_{p}^{\#}\right)
\end{equation*}%
and 
\begin{equation*}
\mathrm{Ext}\left( G,H\right) \cong \mathrm{Ext}\left( G_{p},H_{p}\right)
\oplus \mathrm{Ext}\left( G_{p}^{\#},H_{p}^{\#}\right) \text{.}
\end{equation*}%
The same proofs as the ones in Section \ref{Subsection:pro-p} give the
following results.

\begin{theorem}
\label{Theorem:injective-pro-torsion}Suppose that $G$ is a topological
torsion Polish abelian group. The following assertions are equivalent:

\begin{enumerate}
\item $G$ is injective in $\mathbf{proTorPAb}$;

\item $G$ is divisible and it has a basis of zero neighborhoods consisting
of divisible subgroups;

\item $G$ is a product of countable divisible abelian groups;

\item $\mathrm{Ext}((\mathbb{Z}\left( p^{\infty }\right) ^{\left( \omega
\right) })^{\omega },G)=0$ for every prime $p$;

\item $G_{p}$ is injective in $\mathbf{PAb}\left( p\right) $ for every prime 
$p$.
\end{enumerate}
\end{theorem}

\begin{theorem}
Suppose that $G$ is a topological torsion locally compact Polish abelian
group. The following assertions are equivalent:

\begin{enumerate}
\item $G$ is injective in $\mathbf{proTorLCPAb}$;

\item $G$ is injective in $\mathbf{proTorPAb}$;

\item $G$ is a countable and divisible.
\end{enumerate}
\end{theorem}

\begin{corollary}
Suppose that $G$ is a topological torsion locally compact Polish abelian
group. The following assertions are equivalent:

\begin{enumerate}
\item $G$ is projective in $\mathbf{proTorLCPAb}$;

\item $G$ is compact and torsion-free.
\end{enumerate}
\end{corollary}

\begin{corollary}
\label{Corollary:enough-injectives-pro-torsion}The category $\mathbf{%
proTorPAb}$ has enough injectives and homological dimension $1$.
\end{corollary}

\begin{problem}
Characterize the projective objects in $\mathbf{proTorPAb}$. Does $\mathbf{%
proTorPAb}$ have enough projectives?
\end{problem}

\subsection{Injective non-Archimedean Polish abelian groups}

In this section, we characterize injective and projective objects in the
category of non-Archimedean Polish abelian groups. Recall that a Polish
abelian group $G$ is non-Archimedean if it has a basis of zero neighborhoods
consisting of subgroups. This is equivalent to the assertion that $G$ is
isomorphic to a closed subgroup of $U^{\omega }$, where $U$ is the universal
countable discrete group $\mathbb{Q}^{\left( \omega \right) }\oplus
\bigoplus_{p}\mathbb{Z}\left( p^{\infty }\right) ^{\left( \omega \right) }$.
Non-Archimedean Polish abelian groups form a thick subcategory $\mathbf{PAb}%
_{\mathrm{nA}}$ of the quasi-abelian category of Polish abelian groups.

Recall that a locally compact Polish abelian group has type $\mathbb{Z}$ if
and only if it is discrete and torsion-free. A pro-Lie Polish abelian group $%
G$ has type $\mathbb{Z}$ if $\left\{ N\in \mathcal{N}\left( G\right) :G/N%
\text{ has type }\mathbb{Z}\right\} $ is cofinal in $\mathcal{N}\left(
G\right) $. Notice that every type $\mathbb{Z}$ Polish abelian group is
isomorphic to a subgroup of $(\mathbb{Q}^{\left( \omega \right) })^{\omega }$%
.We let $\mathbf{proLiePAb}_{\mathbb{Z}}$ be the fully exact subcategory of $%
\mathbf{proLiePAb}$ consisting of Polish abelian groups of type $\mathbb{Z}$.

\begin{theorem}
\label{Theorem:projective-type-Z}Let $A$ be a non-Archimedean Polish abelian
group. The following assertions are equivalent:

\begin{enumerate}
\item $A$ is of type $\mathbb{Z}$, and projective in $\mathbf{proLiePAb}$;

\item $A$ is projective in $\mathbf{proLiePAb}_{\mathbb{Z}}$;

\item $A$ is isomorphic to $\mathbb{Z}^{\alpha }\oplus \left( \mathbb{Z}%
^{\left( \omega \right) }\right) ^{\beta }$ for some $\alpha ,\beta \leq
\omega $;

\item $\mathrm{Ext}\left( A,\mathbb{Z}\right) =0$ and $\mathrm{Ext}\left( A,%
\mathbb{Z}^{\left( \omega \right) }\right) =0$.
\end{enumerate}
\end{theorem}

\begin{proof}
The equivalence (3)$\Leftrightarrow $(1) follows from Theorem \ref%
{Theorem:characterize-projectives}, while the implications (1)$\Rightarrow $%
(4) and (1)$\Rightarrow $(2) is obvious.

(2)$\Rightarrow $(1) Suppose that $A$ is projective in $\mathbf{proLiePAb}_{%
\mathbb{Z}}$. Then by Theorem \ref{Theorem:hd-proLiePAb} there exists a
surjective homomorphism $P\oplus V\rightarrow A$ where $P$ is isomorphic to $%
\mathbb{Z}^{\alpha }\oplus \left( \mathbb{Z}^{\left( \omega \right) }\right)
^{\beta }$ for some $\alpha ,\beta \leq \omega $ and $V$ is a vector group.
Since $A$ is type $\mathbb{Z}$, $\mathrm{Hom}\left( V,A\right) =0$. Thus,
there exists a surjective homomorphism $P\rightarrow A$. Since $A$ is
projective in $\mathbf{proLiePAb}_{\mathbb{Z}}$, we have that $A$ is a
direct summand of $P$. Since $P$ is projective in $\mathbf{proLiePAb}$ by
Theorem \ref{Theorem:characterize-projectives}, also $A$ is projective in $%
\mathbf{proLiePAb}$, concluding the proof.

(4)$\Rightarrow $(1) For $\alpha \leq \omega $ we have that $\mathrm{Ext}%
\left( A,\mathbb{Z}^{\alpha }\right) \cong \mathrm{Ext}\left( A,\mathbb{Z}%
\right) ^{\alpha }=0$ and $\mathrm{Ext}\left( A,\left( \mathbb{Z}^{\left(
\omega \right) }\right) ^{\alpha }\right) \cong \mathrm{Ext}\left( A,\left( 
\mathbb{Z}^{\left( \omega \right) }\right) \right) ^{\alpha }=0$. By Lemma %
\ref{Lemma:enough-projective-proLie} there exists a surjective epimorphism $%
\mathbb{R}^{\omega }\oplus \left( \mathbb{Z}^{\left( \omega \right) }\right)
^{\omega }\rightarrow A$. Since $\mathrm{Hom}\left( \mathbb{R}^{\omega
},A\right) =0$, there exists a surjective epimorphism $\pi :\left( \mathbb{Z}%
^{\left( \omega \right) }\right) ^{\omega }\rightarrow A$.\ Let $C:=\mathrm{%
\mathrm{Ker}}\left( \pi \right) $ and observe that $C$ is a closed subgroup
of $\left( \mathbb{Z}^{\left( \omega \right) }\right) ^{\omega }$, and hence
non-Archimedean. Furthermore, $C$ is projective in $\mathbf{proLiePAb}$, and
hence isomorphic to $\mathbb{Z}^{\alpha }\oplus \left( \mathbb{Z}^{\left(
\omega \right) }\right) ^{\beta }$ for some $\alpha ,\beta \leq \omega $.
Thus, by the above remarks we have that $\mathrm{Ext}\left( A,C\right) =0$
and hence $A$ is isomorphic to a topological direct summand of $\left( 
\mathbb{Z}^{\left( \omega \right) }\right) ^{\omega }$.
\end{proof}

\begin{theorem}
\label{Theorem:projective-nA}Let $A$ be a non-Archimedean pro-Lie Polish
abelian group. The following assertions are equivalent:

\begin{enumerate}
\item $A$ is projective in $\mathbf{PAb}_{\mathrm{nA}}$;

\item $A$ is isomorphic to $\mathbb{Z}^{\alpha }\oplus \left( \mathbb{Z}%
^{\left( \omega \right) }\right) ^{\beta }$ for some $\alpha ,\beta \leq
\omega $;

\item $\mathrm{Ext}\left( A,\mathbb{Z}\right) =0$ and $\mathrm{Ext}\left( A,%
\mathbb{Z}^{\left( \omega \right) }\right) =0$.
\end{enumerate}
\end{theorem}

\begin{proof}
This is an immediate consequence of Theorem \ref{Theorem:projective-type-Z}.
\end{proof}

The same proof as Theorem \ref{Theorem:injective-pro-p} gives the following.

\begin{theorem}
\label{Theorem:injective-nA}Suppose that $G$ is a non-Archimedean Polish
abelian group. The following assertions are equivalent:

\begin{enumerate}
\item $G$ is injective in $\mathbf{PAb}_{\mathrm{nA}}$;

\item $G$ is divisible and it has a basis of zero neighborhoods consisting
of divisible subgroups;

\item $G$ is a product of countable divisible abelian groups;

\item $\mathrm{Ext}((\mathbb{Z}\left( p^{\infty }\right) ^{\left( \omega
\right) })^{\omega },G)=0$ for every prime $p$.
\end{enumerate}
\end{theorem}

\begin{theorem}
\label{Theorem:injective-typeZ}Let $A$ be a type $\mathbb{Z}$ pro-Lie Polish
abelian group. The following assertions are equivalent:

\begin{enumerate}
\item $A$ is injective in $\mathbf{proLiePAb}_{\mathbb{Z}}$;

\item $A$ is divisible;

\item $A\cong \mathbb{Q}^{\alpha }\oplus (\mathbb{Q}^{(\omega )})^{\beta }$
for some $\alpha ,\beta \leq \omega $.
\end{enumerate}
\end{theorem}

\begin{proof}
(1)$\Rightarrow $(2) Suppose that $A$ is an injective type $\mathbb{Z}$
pro-Lie Polish abelian group. Then $A$ is isomorphic to a closed subgroup of 
$(\mathbb{Q}^{\left( \omega \right) })^{\omega }$. Since $A$ is injective,
it is a direct summand of $(\mathbb{Q}^{\left( \omega \right) })^{\omega }$,
and hence it is divisible.

(2)$\Rightarrow $(3) If $A$ is divisible, then $A\cong \mathrm{\mathrm{lim}}%
_{k}A_{k}$ where, for every $k\in \omega $, $A_{k}$ is countable, divisible,
and torsion-free and $A_{k+1}\rightarrow A_{k}$ is surjective with divisible
kernel (notice that the kernel of $A_{k+1}\rightarrow A_{k}$ is a pure
subgroup of $A_{k+1}$ since $A_{k}$ is torsion-free, and since $A_{k+1}$ is
torsion-free divisible, the same holds for the kernel of $A_{k+1}\rightarrow
A_{k}$). Since countable divisible groups are injective for countable
abelian groups, it follows that $A$ is isomorphic to the product of
countable divisible groups.

(3)$\Rightarrow $(1) We show that if $D$ is a countable divisible group and $%
\alpha \leq \omega $, then $D^{\alpha }$ is injective in $\mathbf{proLiePAb}%
_{\mathbb{Z}}$. For a type $\mathbb{Z}$ pro-Lie Polish abelian group $A$, we
have that%
\begin{equation*}
\mathrm{Ext}\left( A,D^{\alpha }\right) \cong \mathrm{Ext}\left( A,D\right)
^{\alpha }\text{.}
\end{equation*}%
Thus, it suffices to consider the case when $\alpha =1$. In this case, we
have a short exact sequence $A\rightarrow B\rightarrow C$ where $%
B=\prod_{k}B_{k}$ is a product of countable abelian groups. This induces a
surjective homomorphism%
\begin{equation*}
\mathrm{Ext}\left( B,D\right) \rightarrow \mathrm{Ext}\left( A,D\right) 
\text{.}
\end{equation*}%
Since $D$ is injective for countable abelian groups, we have that $\mathrm{%
Ext}\left( B,D\right) =0$ by Lemma \ref{Lemma:vanishing-Ext-product}.
\end{proof}

\begin{corollary}
\label{Corollary:enough-projectives-typeZ}We have that:

\begin{enumerate}
\item The category $\mathbf{PAb}_{\mathrm{nA}}$ has enough injective objects
and enough projective objects, and homological dimension $1$;

\item The category $\mathbf{proLiePAb}_{\mathbb{Z}}$ has enough projective
objects but not enough injective objects, and homological dimension $1$. An
object in $\mathbf{proLiePAb}_{\mathbb{Z}}$ has an injective resolution if
and only if it is injective.
\end{enumerate}
\end{corollary}

\begin{proof}
(1) The same proof as Lemma \ref{Lemma:enough-projective-proLie} shows that,
for every non-Archimedean Polish abelian group $A$ there exists a surjective
continuous homomorphism $\left( \mathbb{Z}^{\left( \omega \right) }\right)
^{\omega }\rightarrow A$. By Theorem \ref{Theorem:hd-proLiePAb} this implies
that $\mathbf{PAb}_{\mathrm{nA}}$ has enough projective objects and
homological dimension $1$. The same proof as Corollary \ref%
{Corollary:enough-injectives-pro-p} shows that $\mathbf{PAb}_{\mathrm{nA}}$
has enough injective objects.

(2) We have that $\mathbf{proLiePAb}_{\mathbb{Z}}$ has enough projectives by
Lemma \ref{Lemma:tower-quasi-abelian} and Theorem \ref%
{Theorem:projective-type-Z}. Furthermore, it has homological dimension $1$
by Lemma \ref{Lemma:free-subgroups}. Suppose that $A$ is a type $\mathbb{Z}$
Polish abelian group. Suppose that $A\rightarrow D\rightarrow B$ is a short
exact sequence where $D$ and $B$ are type $\mathbb{Z}$ Polish abelian
groups, and $D$ is injective in $\mathbf{proLiePAb}_{\mathbb{Z}}$. By
Theorem \ref{Theorem:injective-typeZ}, we have that $D$ is divisible. Thus,
for $p\in \mathbb{Z}$ we have that $\mathrm{Hom}\left( \mathbb{Z}/p\mathbb{Z}%
,B\right) =0$ and $\mathrm{Ext}\left( \mathbb{Z}/p\mathbb{Z},D\right) =0$.\
This implies that $\mathrm{Ext}\left( \mathbb{Z}/p\mathbb{Z},A\right) =0$
and $A$ is divisible.
\end{proof}

\subsection{\textrm{Ext }as a group with a Polish cover}

Suppose that $A,B$ are countable abelian groups. Then we have that $\mathrm{%
Hom}\left( A,B\right) $ is a Polish abelian group when endowed with the
compact-open topology. We let $\mathbf{PAb}_{\aleph _{0}}$ be the thick
subcategory of $\mathbf{proLiePAb}$ consisting of countable abelian groups.
Since $\mathbf{PAb}_{\aleph _{0}}$ has enough injective and projective
objects and homological dimension $1$, we have that the functor%
\begin{equation*}
\mathrm{Hom}:\mathbf{PAb}_{\aleph _{0}}^{\mathrm{op}}\times \mathbf{PAb}%
_{\aleph _{0}}\rightarrow \mathbf{PAb}
\end{equation*}%
admits a total right derived functor.

Suppose now more generally that $A$ is a \emph{locally compact }Polish
abelian group, and $B$ is a pro-Lie Polish abelian group. Then we still have
that $\mathrm{Hom}\left( A,B\right) $ is a Polish abelian group with respect
to the compact-open topology. We let $\mathbf{LCPAb}$ be the thick
subcategory of $\mathbf{proLiePAb}$ consisting of locally compact Polish
abelian groups, and consider the functor 
\begin{equation*}
\mathrm{Hom}:\mathbf{LCPAb}^{\mathrm{op}}\times \mathbf{proLiePAb}%
\rightarrow \mathbf{PAb}.
\end{equation*}%
We will show that this functor admits a total right derived functor.

\begin{definition}
A pro-Lie Polish abelian group $G$ is \emph{essentially injective }if it has
a closed subgroup $U$ injective in $\mathbf{PAb}$ such that $G/U$ is
non-Archimedean and injective in $\mathbf{PAb}_{\mathrm{nA}}$.
\end{definition}

By Theorem \ref{Theorem:injective-pro-Lie}, a pro-Lie Polish abelian group
is essentially injective if and only if $G$ is isomorphic to $T\oplus
V\oplus A$ where $T$ is a torus, $V$ is a vector group, and $A$ is
non-Archimedean and injective in $\mathbf{PAb}_{\mathrm{nA}}$. We let $%
\mathcal{D}$ be the collection of essentially injective pro-Lie Polish
abelian groups. By Lemma \ref{Lemma:embed-pro-Lie} and Theorem \ref%
{Theorem:injective-nA}, every pro-Lie Polish abelian group is isomorphic to
a closed subgroup of an element of $\mathcal{D}$. We will now show that the
class $\mathcal{D}$ is closed under taking quotients by closed subgroups.
Towards this goal, we isolate a few lemmas concerning $\mathrm{Ext}$ of
countable abelian groups, regarded as a group with a Polish cover.

\begin{lemma}
\label{Lemma:Ext-torsion-tf}Suppose that $T$ is a countable torsion group
and $K$ is a countable torsion-free group. Then $\mathrm{Ext}\left(
T,K\right) $ is a Polish group isomorphic to $\mathrm{Hom}\left(
T,D/K\right) $ where $D$ is the divisible hull of $K$.
\end{lemma}

\begin{proof}
The short exact sequence $K\rightarrow D\rightarrow D/K$ induces an exact
sequence%
\begin{equation*}
\mathrm{Hom}\left( T,D\right) \rightarrow \mathrm{Hom}\left( T,D/K\right)
\rightarrow \mathrm{Ext}\left( T,K\right) \rightarrow \mathrm{Ext}\left(
T,D\right) \text{.}
\end{equation*}%
Since $T$ is torsion and $D$ is torsion-free, $\mathrm{Hom}\left( T,D\right)
=0$. Since $D$ is divisible, $\mathrm{Ext}\left( T,D\right) =0$. The
conclusion follows.
\end{proof}

\begin{lemma}
\label{Lemma:free-Z}Suppose that $A$ is a countable abelian group. If $%
\mathrm{Ext}\left( A,\mathbb{Z}\right) $ is a Polish group, then $A_{\mathbb{%
Z}}$ is free abelian.
\end{lemma}

\begin{proof}
Since $A$ is countable, $F_{\mathbb{Z}}A=A_{\mathrm{t}}$ is the torsion
subgroup of $A$. Consider the exact sequence%
\begin{equation*}
0=\mathrm{Hom}\left( A_{\mathrm{t}},\mathbb{Z}\right) \rightarrow \mathrm{Ext%
}\left( A_{\mathbb{Z}},\mathbb{Z}\right) \rightarrow \mathrm{Ext}\left( A,%
\mathbb{Z}\right) \rightarrow \mathrm{Ext}\left( A_{\mathrm{t}},\mathbb{Z}%
\right) \rightarrow 0\text{.}
\end{equation*}%
Since $A_{\mathrm{t}}$ is torsion, we have that $\mathrm{Ext}\left( A_{%
\mathrm{t}},\mathbb{Z}\right) $ is Polish by Lemma \ref{Lemma:Ext-torsion-tf}%
. It follows that $\mathrm{Ext}\left( A_{\mathbb{Z}},\mathbb{Z}\right) $ is
Polish as well.

Since $A_{\mathbb{Z}}$ is torsion-free, it has no finite subgroups. Thus, by 
\cite[Corollary 11.6]{eilenberg_group_1942}, $\left\{ 0\right\} $ is dense
in $\mathrm{Ext}\left( A_{\mathbb{Z}},\mathbb{Z}\right) $. This implies that 
$\mathrm{Ext}\left( A_{\mathbb{Z}},\mathbb{Z}\right) =0$. Hence, $A_{\mathbb{%
Z}}$ is a free abelian group by \cite[Theorem 3.2]%
{friedenberg_extensions_2013}.
\end{proof}

\begin{lemma}
\label{Lemma:D-quotients}The class $\mathcal{D}$ of essentially injective
pro-Lie Polish abelian groups is closed under taking quotients by closed
subgroups.
\end{lemma}

\begin{proof}
Let $H$ be a pro-Lie Polish group. Then we have that the connected component 
$c\left( H\right) $ of zero in $H$ is $H_{\mathbb{S}^{1}}\oplus H_{\mathbb{R}%
}$, and $H/c\left( H\right) $ is non-Archimedean.\ Thus, we have that $%
H=V\oplus A$ for some vector group $A$ and non-Archimedean abelian Polish
group $A$ if and only if $H_{\mathbb{S}^{1}}=0$.

Suppose that $G$ is an essentially injective pro-Lie Polish abelian group
and $N$ is a closed subgroup of $G$. Set $H:=G/N$. Since $G$ is essentially
injective, we have that $G=T\oplus V\oplus A$ where $V$ is a vector group, $%
T $ is a torus, and $A$ is non-Archimedean and injective in $\mathbf{PAb}_{%
\mathrm{nA}}$. We need to prove that $G/N$ is essentially injective.

We have that $T/\left( T\cap N\right) $ is injective, being the quotient of
a torus group by a closed subgroup. Furthermore, letting $\pi :G\rightarrow
V\oplus A$ be the canonical quotient map, we have that $\pi ^{-1}\left( \pi
\left( N\right) \right) =N+T$ is closed in $G$, and hence $\pi \left(
N\right) $ is closed is $V\oplus A$. Thus, we have a short exact sequence%
\begin{equation*}
\frac{T}{T\cap N}\rightarrow G/N\rightarrow \frac{V\oplus A}{\pi \left(
N\right) }
\end{equation*}%
that splits by injectivity of $T/\left( T\cap N\right) $. It follows that,
after replacing $G$ with $G/T$, we can assume without loss of generality
that $T=0$.

Since $N$ is a closed subgroup of $G=V\oplus A$, we have that $c\left(
N\right) \subseteq V=c(G)$ and hence $c\left( N\right) =N_{\mathbb{R}}$ is a
vector group and $N_{\mathbb{S}^{1}}=0$. Thus, we can write $N=N_{\mathbb{R}%
}\oplus \Xi $ where $\Xi $ is non-Archimedean. Since $N_{\mathbb{R}%
}\subseteq V$ is a vector group, we can assume without loss of generality,
possibly after replacing $G$ with $G/N_{\mathbb{R}}$, that $N_{\mathbb{R}}=0$
and $N=\Xi $ is non-Archimedean.

Then we have that $N\cap V$ is a closed subgroup of $V$ and $V/\left( N\cap
V\right) $ is injective in $\mathbf{PAb}$. Thus, $V/\left( N\cap V\right)
=W\oplus T$ where $T$ is a torus group and $W$ is a vector group. Since $W$
is projective in $\mathbf{PAb}$, we can assume without loss of generality
that $W=0$ and $V/\left( N\cap V\right) $ is a torus group.

Thus, we have that $N/\left( N\cap V\right) $ is a closed subgroup of $G/N$.
Being also injective in $\mathbf{PAb}$, it is a direct summand. Thus, after
replacing $G$ with $G/V$ and $N$ with $N/\left( N\cap V\right) $, we can
assume that $G=A$ is non-Archimedean and injective in $\mathbf{PAb}_{\mathrm{%
nA}}$. In this case, we have that also $G/N$ is non-Archimedean and
injective in $\mathbf{PAb}_{\mathrm{nA}}$, concluding the proof.
\end{proof}

Recall that a locally compact Polish group $G$ is \emph{codivisible }if its
Pontryagin dual $G^{\vee }$ is divisible. We let $\mathcal{C}$ be the
collection of locally compact groups $G$ that are codivisible and satisfy $%
G_{\mathbb{S}^{1}}=0$. Such a group $G$ can be written as $V\oplus A$ where $%
V$ is a vector group and $A$ is totally disconnected. Clearly, $\mathcal{C}$
is closed under taking closed subgroups.

\begin{lemma}
If $C$ is a locally compact Polish group with $C_{\mathbb{Z}}=0$, then there
exists a short exact sequence $C\rightarrow D\rightarrow D/C$ where $D$ is
divisible with $D_{\mathbb{Z}}=0$ and $D/C$ is a countable torsion group.
\end{lemma}

\begin{proof}
Without loss of generality, we can assume that $C$ has no nonzero closed
vector subgroups. By \cite[Proposition 3.8]{hoffmann_homological_2007},
there exists a locally compact Polish group $D$ containing $C$ as an open
subgroup such that $D/C$ is a countable torsion group. Consider the pushout
diagram%
\begin{equation*}
\begin{array}{ccc}
C & \rightarrow & C_{\mathrm{t}} \\ 
\downarrow &  & \downarrow \\ 
D & \rightarrow & H%
\end{array}%
\end{equation*}%
This gives rise to a commutative diagram%
\begin{equation*}
\begin{array}{ccccc}
C_{\mathbb{S}^{1}} & \rightarrow & C & \rightarrow & C_{\mathrm{t}} \\ 
\downarrow &  & \downarrow &  & \downarrow \\ 
C_{\mathbb{S}^{1}} & \rightarrow & D & \rightarrow & H \\ 
&  & \downarrow &  & \downarrow \\ 
&  & D/C & \rightarrow & H/C_{\mathrm{t}}%
\end{array}%
\end{equation*}%
where $D/C\rightarrow H/C_{\mathrm{t}}$ is an isomorphism. Since $C_{\mathrm{%
t}}$ is a topological torsion group and $H/C_{\mathrm{t}}\cong D/C$ is a
countable torsion group (and in particular a topological torsion group), we
conclude that $H$ is a topological torsion group, whence $D_{\mathbb{Z}}=0$.
\end{proof}

\begin{corollary}
\label{Corollary:codivisible}If $C$ is a locally compact Polish group with $%
C_{\mathbb{S}^{1}}=0$, then there exists a short exact sequence $%
A\rightarrow D\rightarrow C$ where $D$ is a codivisible locally compact
Polish abelian group with $D_{\mathbb{S}^{1}}=0$ and $A$ is a profinite
Polish group.
\end{corollary}

\begin{lemma}
\label{Lemma:cogenerating}For every locally compact Polish group $C$ there
exists a continuous surjective homomorphism $G\rightarrow C$ for some
element $G$ of $\mathcal{C}$.
\end{lemma}

\begin{proof}
Without loss of generality, we can assume that $C$ has no nonzero closed
vector groups as a direct summand, whence $C/C_{\mathbb{S}^{1}}$ is totally
disconnected. We have that $C_{\mathbb{S}^{1}}^{\vee }$ is a countable
torsion-free group. Thus there exists a short exact sequence $E\rightarrow
C_{\mathbb{S}^{1}}^{\vee }\rightarrow S$ where $E$ is a countable free
abelian group and $S$ is a countable torsion group. By duality, this gives a
short exact sequence $A\rightarrow C_{\mathbb{S}^{1}}\rightarrow T$, where $%
A=S^{\vee }$ is a profinite abelian Polish group and $T=E^{\vee }$ is a
torus group. Considering the inclusion $C_{\mathbb{S}^{1}}\rightarrow C$ one
obtains by pushout a diagram%
\begin{equation*}
\begin{array}{ccc}
C_{\mathbb{S}^{1}} & \rightarrow & T \\ 
\downarrow &  & \downarrow \\ 
C & \rightarrow & H%
\end{array}%
\end{equation*}%
which gives a commutative diagram%
\begin{equation*}
\begin{array}{ccccc}
A & \rightarrow & C_{\mathbb{S}^{1}} & \rightarrow & T \\ 
\downarrow &  & \downarrow &  & \downarrow \\ 
A & \rightarrow & C & \rightarrow & H \\ 
&  & \downarrow &  & \downarrow \\ 
&  & C/C_{\mathbb{S}^{1}} & \rightarrow & T/H%
\end{array}%
\end{equation*}%
where $A\rightarrow C\rightarrow H$ is a short exact sequence and the
continuous group homomorphism $T\rightarrow H$ is injective with closed
image, and the continuous group homomorphism $C/C_{\mathbb{S}%
^{1}}\rightarrow T/H$ is an isomorphism. By injectivity of $T$, we have that 
$H\cong T\oplus C/C_{\mathbb{S}^{1}}$. Since $T$ is a locally compact torus
group, there exists a continuous surjective homomorphism $V\rightarrow T$
for some locally compact vector group $V$, which induces a continuous
surjective homomorphism $V\oplus C/C_{\mathbb{S}^{1}}\rightarrow H$. We
consider the pullback diagram%
\begin{equation*}
\begin{array}{ccc}
C & \rightarrow & H \\ 
\uparrow &  & \uparrow \\ 
C^{\prime } & \rightarrow & V\oplus C/C_{\mathbb{S}^{1}}%
\end{array}%
\end{equation*}%
which induces a commutative diagram%
\begin{equation*}
\begin{array}{ccccc}
A & \rightarrow & C & \rightarrow & H \\ 
\uparrow &  & \uparrow &  & \uparrow \\ 
A & \rightarrow & C^{\prime } & \rightarrow & V\oplus C/C_{\mathbb{S}^{1}}%
\end{array}%
\end{equation*}%
By projectivity of $V$, we can write $C^{\prime }=V\oplus C^{\prime \prime }$%
, where we have an extension $A\rightarrow C^{\prime \prime }\rightarrow
C/C_{\mathbb{S}^{1}}$. Hence, $C^{\prime \prime }$ is totally disconnected.
The conclusion thus follows from Corollary \ref{Corollary:codivisible}.
\end{proof}

\begin{proposition}
\label{Proposition:resolutions}We have that, for a locally compact Polish
abelian group $C\in \mathcal{C}$ and a pro-Lie Polish abelian group $D\in 
\mathcal{D}$, $\mathrm{Ext}\left( C,D\right) =0$.
\end{proposition}

\begin{proof}
We can write $D=T\oplus V\oplus A$ where $T$ is a torus, $V$ is a vector
group, and $A$ is non-Archimedean and injective in $\mathbf{PAb}_{\mathrm{nA}%
}$. We can also write $C=W\oplus B$ where $W$ is a vector group and $B$ is
totally disconnected. By injectivity of $T$ and $V$, and projectivity of $W$%
, in the category of pro-Lie Polish abelian groups, we have that $\mathrm{Ext%
}\left( C,D\right) \cong \mathrm{Ext}\left( B,A\right) $. By Theorem \ref%
{Theorem:injective-nA}, we have that $A\cong \prod_{n\in \omega }A_{n}$
where $A_{n}$ is a countable divisible group. Thus,%
\begin{equation*}
\mathrm{Ext}\left( B,A\right) \cong \prod_{n\in \omega }\mathrm{Ext}\left(
B,A_{n}\right) =0\text{.}
\end{equation*}%
This concludes the proof.
\end{proof}

It follows from Proposition \ref%
{Proposition:explicitly-right-derivable-bifunctor} in the case when $%
\mathcal{A}=\mathbf{LCPAb}$, $\mathcal{B}=\mathbf{proLiePAb}$, and $\mathcal{%
C}$ and $\mathcal{D}$ are the classes defined above, Proposition \ref%
{Proposition:resolutions}, Lemma \ref{Lemma:cogenerating}, and Lemma \ref%
{Lemma:D-quotients}, that $\mathrm{Hom}^{\bullet }:\mathrm{K}^{b}\left( 
\mathbf{LCPAb}\right) \times \mathrm{K}^{b}\left( \mathbf{proLiePAb}\right)
\rightarrow \mathrm{K}^{b}\left( \mathbf{PAb}\right) $ has a total right
derived functor $\mathrm{RHom}:\mathrm{D}^{b}\left( \mathbf{LCPAb}\right)
\times \mathrm{D}^{b}\left( \mathbf{proLiePAb}\right) \rightarrow \mathrm{D}%
^{b}\left( \mathbf{PAb}\right) $, and $\mathrm{H}^{0}\circ \mathrm{RHom:D}%
^{b}\left( \mathbf{LCPAb}\right) \times \mathrm{D}^{b}\left( \mathbf{%
proLiePAb}\right) \rightarrow \mathrm{LH}\left( \mathbf{PAb}\right) $ is a
cohomological derived functor of $\mathrm{H}^{0}\circ \mathrm{Hom}^{\bullet
}:\mathrm{K}^{b}\left( \mathbf{LCPAb}\right) \times \mathrm{K}^{b}\left( 
\mathbf{proLiePAb}\right) \rightarrow \mathrm{LH}\left( \mathbf{PAb}\right) $%
.

The same argument as in the proof of \cite[Proposition 4.13]%
{bergfalk_applications_2023} yields the following description of $\mathrm{Ext%
}^{1}\left( G,A\right) $ for locally compact Polish abelian group $G$ and
pro-Lie Polish abelian group $A$. Recall that if $\left( X,\mu \right) $ is
a standard probability space and $A$ is a Polish space, then we let $%
L^{0}\left( X,A\right) $ be the Polish space of $\mu $-a.e. classes of
functions $X\rightarrow A$ endowed with the topology of convergence in
measure.

\begin{proposition}
\label{Proposition:Yoneda-Ext}Suppose that $G$ is a locally compact abelian
Polish group and $A$ is a pro-Lie abelian Polish group. Define $\mathrm{Z}%
\left( \mathrm{Ext}_{\mathrm{Yon}}\left( G,A\right) \right) $ to be the
group of Borel $2$-cocycles on $G$ with coefficients in $A$.

Endow $\mathrm{Z}\left( \mathrm{Ext}_{\mathrm{Yon}}\left( G,A\right) \right) 
$ with the topology given by letting a net $\left( c_{i}\right) $ converge
to $c$ if and only if $\left( c_{i}\left( x,y,z\right) \right) $ converges
to $c\left( x,y,z\right) $ in $A$ for every $x,y,z\in G$. Define $\mathrm{B}%
\left( \mathrm{Ext}_{\mathrm{Yon}}\left( G,A\right) \right) $ to be the
subgroup of $\mathrm{Z}\left( \mathrm{Ext}_{\mathrm{Yon}}\left( G,A\right)
\right) $ consisting of Borel cocycles of the form $\delta f\left(
x,y\right) =f\left( y\right) -f\left( x+y\right) +f\left( x\right) $ for
some Borel function $f:G\rightarrow A$. The function 
\begin{equation*}
\Psi :\mathrm{Z}\left( \mathrm{Ext}_{\mathrm{Yon}}^{1}\left( G,A\right)
\right) \rightarrow L^{0}\left( G^{2},A\right)
\end{equation*}%
defined by mapping $c$ to its a.e.-class establishes a continuous
isomorphism with a closed subgroup of $L^{0}\left( G^{2},A\right) $.
Furthermore, $\mathrm{B}\left( \mathrm{Ext}_{\mathrm{Yon}}^{1}\left(
G,A\right) \right) $ is a Polishable subgroup of $\mathrm{Z}\left( \mathrm{%
Ext}_{\mathrm{Yon}}^{1}\left( G,A\right) \right) $. Consider the group with
a Polish cover 
\begin{equation*}
\mathrm{Ext}_{\mathrm{Yon}}^{1}\left( G,A\right) :=\mathrm{Z}\left( \mathrm{%
Ext}_{\mathrm{Yon}}^{1}\left( G,A\right) \right) /\mathrm{B}\left( \mathrm{%
Ext}_{\mathrm{Yon}}^{1}\left( G,A\right) \right) \text{.}
\end{equation*}%
Then $\mathrm{Ext}_{\mathrm{Yon}}^{1}\left( G,A\right) $ is naturally
Borel-definably isomorphic to the group with a Polish cover $\mathrm{Ext}%
^{1}\left( G,A\right) $.
\end{proposition}

\bibliographystyle{amsalpha}
\bibliography{bibliography2}

\end{document}